\documentclass[final,leqno]{siamltex}

\usepackage[leqno]{amsmath}
\usepackage{graphicx}
\usepackage{graphics}
\usepackage{epstopdf}
\usepackage{subfigure}
\usepackage{threeparttable}

\usepackage{bm}
\usepackage{amssymb}
\usepackage{amsmath}
\usepackage{leftidx}
\usepackage{empheq}
\usepackage{color}
\usepackage{makecell}
\allowdisplaybreaks
\renewcommand{\theequation}{\arabic{section}.\arabic{equation}}
\numberwithin{equation}{section}
\newtheorem{remark}{Remark}[section]
\title{The exponential scalar auxiliary variable (E-SAV) approach for phase field models and its explicit computing.
        \thanks{
We would like to acknowledge the assistance of volunteers in putting together this example manuscript and supplement. This work is supported by the Postdoctoral Science Foundation of China under grant numbers BX20190187 and 2019M650152, by National Natural Science Foundation of China (Grant Nos: 11901489, 11971276).}}

      \author{Zhengguang Liu\textsuperscript{*}
             \thanks{Corresponding author: School of Mathematics and Statistics, Shandong Normal University, Jinan, China. Email: liuzhgsdu@yahoo.com}.
                                       \and
             Xiaoli Li
             \thanks{Fujian Provincial Key Laboratory on Mathematical Modeling and High Performance Scientific Computing and School of Mathematical Sciences, Xiamen University, Xiamen, Fujian, 361005, China. Email: xiaolisdu@163.com}. }
\begin{document}

\maketitle

\begin{abstract}
In this paper, we consider an exponential scalar auxiliary variable (E-SAV) approach to obtain energy stable schemes for a class of phase field models. This novel auxiliary variable method based on exponential form of nonlinear free energy potential is more effective and applicable than the traditional SAV method which is very popular to construct energy stable schemes. The first contribution is that the auxiliary variable without square root removes the bounded from below restriction of the nonlinear free energy potential. Then, we prove the unconditional energy stability for the semi-discrete schemes carefully and rigorously. Another contribution is that we can discrete the auxiliary variable combined with the nonlinear term totally explicitly. Such modification is very efficient for fast calculation. Furthermore, the positive property of $r$ can be guaranteed which is very important and reasonable for the models' equivalence. Besides, for complex phase field models with two or more unknown variables and nonlinear terms, we construct a multiple E-SAV (ME-SAV) approach to enhance the applicability of the proposed E-SAV approach. A comparative study of classical SAV and E-SAV approaches is considered to show the accuracy and efficiency. Finally, we present various 2D numerical simulations to demonstrate the stability and accuracy.
\end{abstract}



\pagestyle{myheadings}
\thispagestyle{plain}
\markboth{ZHENGGUANG LIU and XIAOLI LI} {E-SAV approach for phase field models}
  \section{Introduction}
The phase field models are very important equations in physics, material science and mathematics \cite{ambati2015review,guo2015thermodynamically,liu2019efficient,marth2016margination,miehe2010phase,shen2015efficient,wheeler1992phase,wheeler1993computation}. They have been widely used in many fields such as alloy casting, new material preparation, image processing, finance and so on. The phase field model can simulate many physical phenomena, such as the formation process of snowflakes, the dendrite structure formed by water freezing, and the cellular or dendrite structure formed in the welding process, etc. It is helpful to understand the nature and the formation mechanism of various materials and the preparation of new materials. It is of great practical significance to develop new technologies.

In general, mathematically, the phase field models are always derived from the functional variation of free energy which can be written explicitly as follows \cite{ShenA}:
\begin{equation}\label{intro-e1}
E(\phi)=(\phi,\mathcal{L}\phi)+E_1(\phi)=(\phi,\mathcal{L}\phi)+\int_\Omega F(\phi)d\textbf{x},
\end{equation}
where $\mathcal{L}$ is a symmetric non-negative linear operator and $E_1(\phi)$ is nonlinear but with only lower-order derivatives than $\mathcal{L}$. $F(x)$ is the energy density function.

The phase field models from the energetic variation of the above energy functional $E(\phi)$ can be obtained as follows:
\begin{equation}\label{intro-e2}
\frac{\partial\phi}{\partial t}=\mathcal{G}\frac{\delta E}{\delta \phi},
\end{equation}
where $\frac{\delta E}{\delta \phi}$ is variational derivative. $\mathcal{G}$ is a non-positive operator. For example, $\mathcal{G}=-I$ for the Allen-Cahn type system and $\mathcal{G}=\Delta$ for the Cahn-Hilliard type system for the Ginzburg-Landau double-well type potential $F$. The system satisfies the following energy dissipation law naturally:
\begin{equation*}
\frac{d}{dt}E=(\frac{\delta E}{\delta \phi},\frac{\partial\phi}{\partial t})=(\mathcal{G}\frac{\delta E}{\delta \phi},\frac{\delta E}{\delta \phi})\leq0.
\end{equation*}

Energy dissipation law is a very important property for phase field models in physics and mathematics. Thus, this property is essential for numerical schemes. That is to say, the discrete energy of the proposed numerical discrete schemes should also maintain dissipative properties. Up to now, many scholars considered a series of efficient and popular time discretization approaches to construct energy stable schemes for phase field models such as convex splitting approach \cite{eyre1998unconditionally,shen2012second,shin2016first}, linear stabilized approach \cite{shen2010numerical,yang2017numerical}, exponential time differencing (ETD) approach \cite{du2019maximum,WangEfficient}, invariant energy quadratization (IEQ) approach \cite{chen2019efficient,chen2019fast,yang2016linear}, scalar auxiliary variable (SAV) approach \cite{xiaoli2019energy,shen2018scalar,ShenA} and so on. Specifically, the convex splitting method leads to a convex minimization problem at each time step and the scheme is unconditionally energy stable and uniquely solvable. But it still needs to solve a nonlinear system, and it is difficult to construct higher-order scheme. The linear stabilized method can effectively solve a linear system of the phase field models, but the additional stabilized term leads to additional error, which makes it difficult to construct the higher-order scheme. Both IEQ and SAV methods are unconditional energy stabilization methods which are developed in recent years. The IEQ method is inspired by Lagrange multiplier method but makes a big leap. By introducing a auxiliary variable, X. Yang et. al. \cite{YangNumerical,YangEfficient} successfully avoided the difficulty of discretization of the nonlinear term. The IEQ approach has been proven to keep many advantages such as linear, easy to obtain second order scheme and unconditional energy stability. This method has been successfully applied to the numerical simulation for many complex phase field models. The SAV method was proposed by J. Shen and his collaborators \cite{shen2018scalar,ShenA} which is another very popular and efficient approach. It is worth mentioning that SAV method keeps all the advantages of the IEQ approach. Furthermore, it weakens the assumptions of the bounded below restriction of nonlinear free energy potential which makes it to be a new important way to simulate the phase field models.

The first main contribution of this paper is that we find a proper way to get rid of the assumption of nonlinear free energy potential in SAV approach. In order to show and give a comparative study for our novel E-SAV approach, we provide below a brief review of SAV approach to construct energy stable schemes for phase field models. In general, the phase field models \eqref{intro-e2} can always be written as the following by denoting the chemical potential $\mu=\frac{\delta E}{\delta \phi}$:
\begin{equation*}
  \left\{
   \begin{array}{rll}
\displaystyle\frac{\partial \phi}{\partial t}&=&\mathcal{G}\mu,\\
\mu&=&\displaystyle\mathcal{L}\phi+F'(\phi),
   \end{array}
   \right.
\end{equation*}
subject to periodic boundary conditions or $\frac{\partial \phi}{\partial\textbf{n}}|_{\partial\Omega}=\frac{\partial \mu}{\partial\textbf{n}}|_{\partial\Omega}=0$.

The key of the SAV approach is to transform the nonlinear potential into a simple quadratic form. This transformation makes the nonlinear term much easier to handle. Assuming that $E_1(\phi)$ is bounded from below which means that there exists a constant $C$ to make $E_1(\phi)+C>0$. Define a scalar auxiliary variable
\begin{equation*}
r(t)=\sqrt{E_1(\phi)+C}=\sqrt{\int_\Omega F(\phi)d\textbf{x}+C}>0.
\end{equation*}
Then, the nonlinear functional $F'(\phi)$ can be transformed into the following equivalent formulation:
\begin{equation*}
F'(\phi)=\frac{r}{r}F'(\phi)=\frac{r}{\sqrt{E_1(\phi)+C}}F'(\phi).
\end{equation*}
Thus, an equivalent system of phase field models with scalar auxiliary variable can be rewritten as follows
\begin{equation}\label{intro-e3}
  \left\{
   \begin{array}{rll}
\displaystyle\frac{\partial \phi}{\partial t}&=&\mathcal{G}\mu,\\
\mu&=&\displaystyle\mathcal{L}\phi+\frac{r}{\sqrt{E_1(\phi)+C}}F'(\phi),\\
r_t&=&\displaystyle\frac{1}{2\sqrt{E_1(\phi)+C}}\int_{\Omega}F'(\phi)\phi_td\textbf{x}.
   \end{array}
   \right.
\end{equation}

Taking the inner products of the above equations with $\mu$, $\phi_t$ and $2r$, respectively, we obtain that the above equivalent system satisfies a modified energy dissipation law:
\begin{equation*}
\frac{d}{dt}\left[\frac12(\phi,\mathcal{L}\phi)+r^2\right]=(\mathcal{G}\mu,\mu)\leq0.
\end{equation*}

The above phase field models with SAV scheme \eqref{intro-e3} is very easy to construct linear, second order and unconditional energy stable schemes. For example, a second-order semi-discrete scheme based on the Crank-Nicolson method, reads as follows
\begin{equation}\label{intro-e4}
  \left\{
   \begin{array}{rll}
\displaystyle\frac{\phi^{n+1}-\phi^{n}}{\Delta t}&=&\mathcal{G}\mu^{n+1/2},\\
\mu^{n+1/2}&=&\displaystyle\mathcal{L}\left(\frac{\phi^{n+1}+\phi^n}{2}\right)+\frac{r^{n+1}+r^n}{2\sqrt{E_1(\tilde{\phi}^{n+1/2})+C}}F'(\tilde{\phi}^{n+1/2}),\\
\displaystyle\frac{r^{n+1}-r^n}{\Delta t}&=&\displaystyle\frac{1}{2\sqrt{E_1(\tilde{\phi}^{n+1/2})+C}}\int_{\Omega}F'(\tilde{\phi}^{n+1/2})\frac{\phi^{n+1}-\phi^{n}}{\Delta t}d\textbf{x},
   \end{array}
   \right.
\end{equation}
where $\tilde{\phi}^{n+\frac{1}{2}}$ is any explicit $O(\Delta t^2)$ approximation for $\phi(t^{n+\frac{1}{2}})$, which can be flexible according to the problem.

It is not difficult to prove that the above scheme is unconditionally energy stable in the sense that
\begin{equation*}
\aligned
\left[\frac12(\mathcal{L}\phi^{n+1},\phi^{n+1})+|r^{n+1}|^2\right]-\left[\frac12(\mathcal{L}\phi^{n},\phi^{n})+|r^{n}|^2\right]\leq\Delta t(\mathcal{G}\mu^{n+1/2},\mu^{n+1/2})\leq0.
\endaligned
\end{equation*}

The SAV approach has been treated as a very efficient and powerful way to construct energy stable schemes and it is easy to calculate. However, there is one obvious shortcoming that the models need to satisfy an assumption which the nonlinear free energy $E_1(\phi)$ is bounded from below. In order to ensure the correctness of the discrete scheme, we have to give a very big positive $C$ before calculation. However, It is observed that the $C$ value seems to have an influence on the accuracy of the simulation results. Relative study can be found in \cite{lin2019numerical}. In their study, they found that the error histories for $C=0.01$ and $C=500$ exhibit quite different characteristics. The error corresponding to $C$ decreases quickly but the error corresponding to $C=500$ decreases extremely slowly at this stage. To enhance the applicability of the SAV method, we aim to find reasonable procedure to avoid using a estimated number $C$ during the calculation. We consider an E-SAV method by using the constant positive properties of exponential functions to obtain energy stable schemes. We prove the unconditional energy stability for the semi-discrete schemes carefully and rigorously. The second contribution is that the discrete scheme based on E-SAV approach is very easy to construct explicit numerical scheme. Such modification is very efficient for fast calculation. Besides, for complex phase field models with two or more unknown variables and nonlinear terms, we construct a multiple E-SAV (ME-SAV) approach to enhance the applicability of the proposed E-SAV approach. A comparative study of classical SAV and E-SAV approaches is considered to show the accuracy and efficiency. Finally, we present various 2D numerical simulations to demonstrate the stability and accuracy.

In summary, the constructed E-SAV approach has the following five advantages compared with the recently proposed SAV approach:

$(i)$ The E-SAV approach does not need any assumptions while the nonlinear free energy potential has to be bounded from below in SAV approach;

$(ii)$ The novel auxiliary variable $r$ and $r^n$ are always positive in E-SAV schemes. However, such positive property of $r$ and $r^n$ can not be guaranteed in SAV approach;

$(iii)$ The totally explicit schemes of the auxiliary variable combined with the nonlinear term with unconditionally energy stability based on the E-SAV approach can be constructed easily while such explicit schemes are not energy stable for SAV approach;

$(iv)$ The computations of $\phi$ and the auxiliary variable $r$ can be solved step-by-step based on E-SAV approach while we have to compute an inner product previously to obtain $\phi^{n+1}$ in the SAV schemes;

$(v)$ Schemes based on the E-SAV approach dissipate the original energy, as opposed to a modified energy in the SAV approach.

The paper is organized as follows. In Sect.2, we introduce the E-SAV approach for phase field models and give two numerical discrete schemes. Then, we prove the unconditional energy stability for the semi-discrete scheme. In Sect.3, the E-SAV approach of the phase field models of several functions are considered. In Sect.4, considering that the exponential function is a rapidly increasing function which carries the risk of failure for the E-SAV approach, we give some modified technique to improve the scope of application. To enhance the applicability of the proposed E-SAV approach for complex phase field models, we construct a multiple E-SAV approach in Sect.5. In the last Sect.6, various 2D numerical simulations are demonstrated to verify the accuracy and efficiency of our proposed schemes.

\section{E-SAV approach for phase field models}
In this section, we will consider an E-SAV approach for phase field models to construct energy stable numerical schemes. Exponential function is a special function that keeps the range constant positive. Thus, we introduce an exponential scalar auxiliary variable (E-SAV):
\begin{equation}\label{esav-e1}
\aligned
r(t)=\exp\left(E_1(\phi)\right)=\exp\left(\int_\Omega F(\phi)d\textbf{x}\right).
\endaligned
\end{equation}
It is obviously $r(t)>0$ for any $t$. Then, the nonlinear functional $F'(\phi)$ can be transformed as the following equivalent formulation:
\begin{equation*}
F'(\phi)=\frac{r}{r}F'(\phi)=\frac{r}{\exp\left(E_1(\phi)\right)}F'(\phi).
\end{equation*}
By taking derivative of \eqref{esav-e1} with respect to $t$ and replacing $F'(\phi)$ with the above expression, we obtain
\begin{equation*}
\frac{d r}{dt}=r\int_{\Omega}F'(\phi)\phi_td\textbf{x}=\displaystyle\frac{r^2}{\exp\left(E_1(\phi)\right)}\int_{\Omega}F'(\phi)\phi_td\textbf{x}.
\end{equation*}
Thus, \eqref{intro-e4} can be rewritten as the following equivalent system:
\begin{equation}\label{esav-e2}
  \left\{
   \begin{array}{rll}
\displaystyle\frac{\partial \phi}{\partial t}&=&\mathcal{G}\mu,\\
\mu&=&\displaystyle\mathcal{L}\phi+\frac{r}{\exp\left(E_1(\phi)\right)}F'(\phi),\\
r_t&=&\displaystyle\frac{r^2}{\exp\left(E_1(\phi)\right)}\int_{\Omega}F'(\phi)\phi_td\textbf{x}.
   \end{array}
   \right.
\end{equation}
To simplify the notations, we define
\begin{equation*}
b^{r,\phi}=\frac{r}{\exp\left(E_1(\phi)\right)}F'(\phi).
\end{equation*}
Then, the above system \eqref{esav-e2} can be transformed as follows:
\begin{equation}\label{esav-e3}
  \left\{
   \begin{array}{rll}
\displaystyle\frac{\partial \phi}{\partial t}&=&\mathcal{G}\mu,\\
\mu&=&\displaystyle\mathcal{L}\phi+b^{r,\phi},\\
r_t&=&r(b^{r,\phi},\phi_t).
   \end{array}
   \right.
\end{equation}
Taking the inner products of the first two equations with $\mu$ and $\phi_t$ in \eqref{esav-e3} respectively, we obtain that
\begin{equation}\label{esav-e4}
\displaystyle\left(\frac{\partial \phi}{\partial t},\mu\right)=(\mathcal{G}\mu,\mu)\leq0,
\end{equation}
and
\begin{equation}\label{esav-e5}
\displaystyle\left(\frac{\partial \phi}{\partial t},\mu\right)=\displaystyle\frac12\frac{d}{dt}(\mathcal{L}\phi,\phi)+(b^{r,\phi},\phi_t).
\end{equation}

For the third equation in \eqref{esav-e3}, noting that $r>0$, then, it can be transformed as follows:
\begin{equation}\label{esav-e6}
\displaystyle \frac{d\ln(r)}{dt}=(b^{r,\phi},\phi_t).
\end{equation}

Combining the equations \eqref{esav-e4}-\eqref{esav-e5} with equation \eqref{esav-e6}, we obtain the following energy dissipation law:
\begin{equation*}
\frac{d}{dt}\left[\frac12(\mathcal{L}\phi,\phi)+\ln(r)\right]=(\mathcal{G}\mu,\mu)\leq0.
\end{equation*}

\begin{remark}\label{esav-re0}
For SAV approach in \eqref{intro-e4}, the modified energy dissipation law is not equal to the original one because of $\left[\frac12(\mathcal{L}\phi^{n},\phi^{n})+|r^{n}|^2\right]=E(\phi)+C$. However, notice that $\ln(r)=\ln(\exp(E_1(\phi)))=E_1(\phi)$. Thus, we have $\frac12(\mathcal{L}\phi,\phi)+\ln(r)=E(\phi)$ which means the above energy inequality is totally equal to the original energy dissipation law.
\end{remark}

Next, we will consider some numerical schemes to illustrate that the proposed E-SAV approach is very easy to obtain linear and unconditionally energy stable schemes. More importantly, it can be found that both first-order and second-order explicit numerical schemes with unconditionally energy stability can be constructed easily.

Before giving a semi-discrete formulation, we let $N>0$ be a positive integer and set
\begin{equation*}
\Delta t=T/N,\quad t^n=n\Delta t,\quad \text{for}\quad n\leq N.
\end{equation*}
\subsection{The first-order scheme}
A first order scheme for solving the system \eqref{esav-e2} can be readily derived by the backward Euler¡¯s method. The first-order scheme can be written as follows:
\begin{equation}\label{esav-first-e1}
  \left\{
   \begin{array}{rll}
\displaystyle\frac{\phi^{n+1}-\phi^{n}}{\Delta t}&=&\mathcal{G}\mu^{n+1},\\
\mu^{n+1}&=&\displaystyle\mathcal{L}\phi^{n+1}+b^{r^n,\phi^n},\\
\displaystyle\frac{\ln(r^{n+1})-\ln(r^n)}{\Delta t}&=&\displaystyle\left(b^{r^n,\phi^n},\frac{\phi^{n+1}-\phi^{n}}{\Delta t}\right),
   \end{array}
   \right.
\end{equation}
Multiplying the first two equations in \eqref{esav-first-e1} with $\mu^{n+1}$ and $(\phi^{n+1}-\phi^{n})/\Delta t$, combining them with the third equation in \eqref{esav-first-e1}, and noting the following equation
$$(a-b,a)=\frac12|a|^2-\frac12|b|^2+\frac12|a-b|^2$$
we obtain the discrete energy law:
\begin{equation}\label{esav-first-e2}
\aligned
\frac{1}{\Delta t}\left[E_{1st}^{n+1}-E^{n}_{1st}\right]\leq(\mathcal{G}\mu^{n+1},\mu^{n+1})-\frac{1}{2\Delta t}(\phi^{n+1}-\phi^n,\mathcal{L}(\phi^{n+1}-\phi^n))\leq0,
\endaligned
\end{equation}
where the modified discrete version of the energy is defined by
\begin{equation*}
\aligned
E_{1st}^{n}=\frac12(\phi^n,\mathcal{L}\phi^{n})+\ln(r^n).
\endaligned
\end{equation*}

\begin{remark}
The first-order E-SAV scheme \eqref{esav-first-e1} is much easier to implement than classical SAV. Because the totally explicit computing of $\phi^{n+1}$ can be achieved. Furthermore, the implicit processing of $r^{n+1}$ ensures the energy stability of the discrete format. For SAV scheme which can be seen in \cite{ShenA}, we have to compute inner product $(b^n,\phi^{n+1})$ before obtaining $\phi^{n+1}$ where $b^n=F'(\phi^n)/\sqrt{E_1(\phi^n)}$. However, for E-SAV scheme, we do not need to do this. We can compute $\phi^{n+1}$ directly by the first two equations in \eqref{esav-first-e1}, then $r^{n+1}$ can be very easy to obtain by computing $\left(b^{r^n,\phi^n},\phi^{n+1}-\phi^{n}\right)$. That is to say, the computations of $\phi$ and $r$ are totally decoupled. Thus, Compared with SAV algorithm, the E-SAV algorithm greatly simplifies the calculation which is conducive to rapid simulation.
\end{remark}

In particular, the first two equation in \eqref{esav-first-e1} can be written as:
\begin{equation}\label{esav-first-e3}
\aligned
(I-\Delta t\mathcal{G}\mathcal{L})\phi^{n+1}=\phi^n+\Delta t\mathcal{G}b^{r^n,\phi^n}.
\endaligned
\end{equation}
Multiplying \eqref{esav-first-e3} with $(I-\Delta t\mathcal{G}\mathcal{L})^{-1}$, we can obtain $\phi^{n+1}$ directly:
\begin{equation}\label{esav-first-e4}
\aligned
\phi^{n+1}=(I-\Delta t\mathcal{G}\mathcal{L})^{-1}\phi^n+\Delta t(I-\Delta t\mathcal{G}\mathcal{L})^{-1}\mathcal{G}b^{r^n,\phi^n}.
\endaligned
\end{equation}
Substitute equation \eqref{esav-first-e4} into the third equation in \eqref{esav-first-e3}, we can compute $r^{n+1}$:
\begin{equation}\label{esav-first-e5}
r^{n+1}=\displaystyle\exp\left[\ln(r^n)+\left(b^{r^n,\phi^n},\phi^{n+1}-\phi^{n}\right)\right].
\end{equation}

\begin{remark}
The logarithmic function in equation \eqref{esav-e6}  guarantees the positive property of the auxiliary variable $r$. Meanwhile, for the discrete scheme, the exponential function in equation \eqref{esav-first-e5} guarantees the constant positive property of $r^{n+1}$, which makes $\ln(r^{n+1})$ reasonable to obtain $r^{n+2}$.
\end{remark}

To summarize, we implement \eqref{esav-first-e1} as follows:

1. Compute $r^n$ and $\phi^n$;

2. Compute $b^{r^n,\phi^n}$ from $\displaystyle b^{r^n,\phi^n}=\frac{r^n}{\exp\left(E_1(\phi^n)\right)}F'(\phi^n)$;

3. Compute $\phi^{n+1}$ from \eqref{esav-first-e4};

4. Compute $r^{n+1}$ from \eqref{esav-first-e5}.

\subsection{The second-order scheme}
A linear, second-order, sequentially solved and unconditionally stable E-SAV scheme is also very easy to constructed.  A semi-implicit E-SAV scheme based on the second order Crank-Nicolson formula (CN) for \eqref{esav-e3} reads as: for $n\geq1$,
\begin{equation}\label{esav-second-e1}
  \left\{
   \begin{array}{rll}
\displaystyle\frac{\phi^{n+1}-\phi^{n}}{\Delta t}&=&\mathcal{G}\mu^{n+\frac12},\\
\mu^{n+\frac12}&=&\displaystyle\mathcal{L}\frac{\phi^{n+1}+\phi^{n}}{2}+b^{\widetilde{r}^{n+\frac12},\widetilde{\phi}^{n+\frac12}},\\
\displaystyle\frac{\ln(r^{n+1})-\ln(r^n)}{\Delta t}&=&\displaystyle\left(b^{\widetilde{r}^{n+\frac12},\widetilde{\phi}^{n+\frac12}},\frac{\phi^{n+1}-\phi^{n}}{\Delta t}\right),
   \end{array}
   \right.
\end{equation}
where $\tilde{\phi}^{n+\frac{1}{2}}$ is any explicit $O(\Delta t^2)$ approximation for $\phi(t^{n+\frac{1}{2}})$, and $\tilde{r}^{n+\frac{1}{2}}$ is any explicit $O(\Delta t^2)$ approximation for $r(t^{n+\frac{1}{2}})$, which can be flexible according to the problem. Here, we choose
\begin{equation}\label{esav-second-e2}
\aligned
&\tilde{\phi}^{n+\frac{1}{2}}=\frac32\phi^n-\frac12\phi^{n-1}, \quad n\geq1,\\
&\tilde{r}^{n+\frac{1}{2}}=\frac32r^n-\frac12r^{n-1}, \quad n\geq1,
\endaligned
\end{equation}
and for $n=0$, we compute $\widetilde{\phi}^{\frac{1}{2}}$ as follows:
\begin{equation}\label{esav-second-e3}
\aligned
&\displaystyle\frac{\widetilde{\phi}^{\frac{1}{2}}-\phi^0}{(\Delta t)/2}=\mathcal{G}\left[\mathcal{L}\widetilde{\phi}^{\frac{1}{2}}+F^{'}(\phi^0)\right],
\endaligned
\end{equation}
which has a local truncation error of $O(\Delta t^2)$.

Then, we can compute $\widetilde{r}^{\frac{1}{2}}$ from
\begin{equation}\label{esav-second-e4}
\displaystyle\widetilde{r}^{\frac{1}{2}}=\exp\left[\int_\Omega F(\widetilde{\phi}^{\frac{1}{2}})d\textbf{x}\right].
\end{equation}

Similarly, $\phi^{n+1}$ can be solved by the following:
\begin{equation}\label{esav-second-e5}
\aligned
\phi^{n+1}=(I-\frac12\Delta t\mathcal{G}\mathcal{L})^{-1}\phi^n+\frac12(I-\frac12\Delta t\mathcal{G}\mathcal{L})^{-1}\mathcal{G}\mathcal{L}\phi^n+\Delta t(I-\Delta t\mathcal{G}\mathcal{L})^{-1}\mathcal{G}b^{\widetilde{r}^{n+\frac12},\widetilde{\phi}^{n+\frac12}}.
\endaligned
\end{equation}
Then, we can compute $r^{n+1}$ by $\phi^{n+1}$:
\begin{equation}\label{esav-second-e6}
r^{n+1}=\displaystyle\exp\left[\ln(r^n)+\left(b^{\widetilde{r}^{n+\frac12},\widetilde{\phi}^{n+\frac12}},\phi^{n+1}-\phi^{n}\right)\right].
\end{equation}
Multiplying the first two equations in \eqref{esav-second-e1} with $\mu^{n+\frac12}$ and $(\phi^{n+1}-\phi^{n})/\Delta t$, and combining  them with the third equation in \eqref{esav-second-e1}, we derive the following:
\begin{theorem}\label{esav-th1}
The scheme \eqref{esav-second-e1} for the equivalent phase field system \eqref{esav-e3} is second-order accurate, unconditionally energy stable in the sense that
\begin{equation*}
\aligned
\frac{1}{\Delta t}\left[E_{E-SAV-CN}^{n+1}-E^{n}_{E-SAV-CN}\right]\leq(\mathcal{G}\mu^{n+\frac12},\mu^{n+\frac12})\leq0.
\endaligned
\end{equation*}
where the modified discrete version of the energy is defined by
\begin{equation*}
\aligned
E_{E-SAV-CN}^{n}=\frac12(\phi^n,\mathcal{L}\phi^{n})+\ln(r^n).
\endaligned
\end{equation*}
\end{theorem}
\begin{remark}\label{esav-re1}
The second-order E-SAV scheme \eqref{esav-second-e1} based on Crank-Nicolson can be implemented sequentially as follows: (i) Compute the initial values of $\phi^0$ and $r^0$; (ii) Compute $\tilde{\phi}^{\frac{1}{2}}$ from \eqref{esav-second-e3} and $\tilde{r}^{\frac{1}{2}}$ from \eqref{esav-second-e4}; (iii) Compute $\phi^1$ from \eqref{esav-second-e5}; (iv) Compute $r^1$ from \eqref{esav-second-e6}; (v) Compute $\phi^n$ from \eqref{esav-second-e5} for $n\geq2$; (vi) Compute $r^n$ from \eqref{esav-second-e6} for $n\geq2$.
\end{remark}
\section{E-SAV approach for phase field models of several functions}
In this section, we consider the E-SAV approach for phase field models of multiple functions $\phi_1$, $\phi_2$, $\cdots$, $\phi_k$ $(k\geq2)$. The energy functional will be \cite{ShenA}:
\begin{equation}\label{se3-esav-e1}
E(\phi_1,\phi_2,\ldots,\phi_k)=\sum\limits_{i,j=1}^kd_{i,j}(\mathcal{L}\phi_i,\phi_j)+\sum\limits_{j=1}^k\int_\Omega F(\phi_j)d\textbf{x},
\end{equation}
where $\mathcal{L}$ is a self-adjoint non-negative linear operator, the constant matrix $A=(d_{i,j})$ is symmetric positive definite.

Some applications involve coupled linear operators which render the phase field models of several functions very difficult to solve numerically by existing methods. By introducing a scalar auxiliary variable, SAV approach can solve this problem very efficiently. In this section, we try to use E-SAV approach to obtain a more efficient and easier algorithm to solve phase field models of multiple functions.

We set $E_1(\phi)=\sum\limits_{j=1}^k\int_\Omega F(\phi_j)d\textbf{x}$, then introduce an exponential scalar auxiliary variable:
\begin{equation}\label{se3-esav-e2}
\aligned
r(t)=\exp\left(E_1(\phi_1,\phi_2,\cdots,\phi_k)\right)=\exp\left(\sum\limits_{j=1}^k\int_\Omega F(\phi_j)d\textbf{x}\right).
\endaligned
\end{equation}
Then, we can obtain the phase field models from the energetic variation of the energy functional $E(\phi)$ in \eqref{se3-esav-e1} as follows:
\begin{equation}\label{se3-esav-e3}
  \left\{
   \begin{array}{rll}
\displaystyle\frac{\partial \phi_i}{\partial t}&=&\mathcal{G}\mu_i,\\
\mu_i&=&\displaystyle2\sum\limits_{j=1}^kd_{i,j}\mathcal{L}\phi_j+\frac{r}{\exp\left(E_1(\phi_1,\phi_2,\cdots,\phi_k)\right)}F'_i,\\
\displaystyle\frac{d\ln r}{dt}&=&\displaystyle\frac{r}{\exp\left(E_1(\phi_1,\phi_2,\cdots,\phi_k)\right)}\int_\Omega F'(\phi_i)\frac{\partial \phi_i}{\partial t}d\textbf{x}.
   \end{array}
   \right.
\end{equation}
To simplify the notations, we define
\begin{equation*}
b^{r,\phi_i}=\frac{r}{\exp\left(E_1(\phi_1,\phi_2,\cdots,\phi_k)\right)}F_i'.
\end{equation*}

Then, the above system \eqref{se3-esav-e3} can be transformed as follows:
\begin{equation}\label{se3-esav-e4}
  \left\{
   \begin{array}{rll}
\displaystyle\frac{\partial \phi_i}{\partial t}&=&\mathcal{G}\mu_i,\\
\mu_i&=&\displaystyle2\sum\limits_{j=1}^kd_{i,j}\mathcal{L}\phi_j+b^{r,\phi_i},\\
\displaystyle\frac{d\ln r}{dt}&=&\displaystyle\left(b^{r,\phi_i},\frac{\partial \phi_i}{\partial t}\right).
   \end{array}
   \right.
\end{equation}
Taking the inner products of the first two equations with $\mu_i$ and $\frac{d\phi_i}{dt}$ in \eqref{se3-esav-e4} respectively, combining them with the third equation in \eqref{se3-esav-e4} and summing over $i$, we obtain that
\begin{equation*}
\frac{d}{dt}\left[\sum\limits_{i,j=1}^kd_{i,j}(\mathcal{L}\phi_i,\phi_j)+\ln r\right]=\frac{d}{dt}E(\phi_1,\phi_2,\ldots,\phi_k)=\sum\limits_{i=1}^k(\mathcal{G}\mu_i,\mu_i)\leq0.
\end{equation*}

A linear, second-order, sequentially solved and unconditionally stable E-SAV scheme based on Crank-Nicolson formula can be constructed as follows:
\begin{equation}\label{se3-esav-e5}
  \left\{
   \begin{array}{rll}
\displaystyle\frac{\phi^{n+1}_i-\phi^{n}_i}{\Delta t}&=&\mathcal{G}\mu^{n+\frac12}_i,\\
\mu^{n+\frac12}_i&=&\displaystyle\sum\limits_{j=1}^kd_{i,j}\mathcal{L}(\phi^{n+1}_j+\phi^n_j)+b^{\widetilde{r}^{n+\frac12},\widetilde{\phi}_i^{n+\frac12}},\\
\displaystyle\frac{\ln(r^{n+1})-\ln(r^n)}{\Delta t}&=&\displaystyle\left(b^{\widetilde{r}^{n+\frac12},\widetilde{\phi}_i^{n+\frac12}},\frac{\phi^{n+1}_i-\phi^{n}_i}{\Delta t}\right),
   \end{array}
   \right.
\end{equation}
where $\tilde{\phi}_i^{n+\frac{1}{2}}$ is any explicit $O(\Delta t^2)$ approximation for $\phi_i(t^{n+\frac{1}{2}})$, and $\tilde{r}^{n+\frac{1}{2}}$ is any explicit $O(\Delta t^2)$ approximation for $r(t^{n+\frac{1}{2}})$, which can be flexible according to the problem.

Multiplying the first two equations in \eqref{se3-esav-e5} with $\mu_i^{n+\frac12}$ and $(\phi_i^{n+1}-\phi_i^{n})/\Delta t$, and combining them with the third equation in \eqref{se3-esav-e5} and sum over $i$, we derive the following discrete energy law:
\begin{equation*}
\frac{1}{\Delta t}\left[\sum\limits_{i,j=1}^kd_{i,j}(\mathcal{L}\phi^{n+1}_i,\phi^{n+1}_j)+\ln(r^{n+1})-\sum\limits_{i,j=1}^kd_{i,j}(\mathcal{L}\phi^{n}_i,\phi^{n}_j)-\ln(r^{n})\right]=\sum\limits_{i=1}^k(\mathcal{G}\mu_i^{n+\frac12},\mu_i^{n+\frac12})\leq0.
\end{equation*}

Next, we describe how the scheme \eqref{se3-esav-e5} can be efficiently implemented. Denote
\begin{equation}\label{se3-esav-e6}
\aligned
&\Phi^{n}=(\phi_1^n,\phi_2^n,\ldots,\phi_k^n),\quad B^{n}=(b^{\widetilde{r}^{n+\frac12},\widetilde{\phi}_1^{n+\frac12}},b^{\widetilde{r}^{n+\frac12},\widetilde{\phi}_2^{n+\frac12}},\ldots,b^{\widetilde{r}^{n+\frac12},\widetilde{\phi}_k^{n+\frac12}}).
\endaligned
\end{equation}
Then, $\Phi^{n+1}$ can be computed by the following equation:
\begin{equation}\label{se3-esav-e7}
\aligned
\Phi^{n+1}=(I-\Delta t\mathcal{G}A\mathcal{L})^{-1}\Phi^n+\frac12(I-\Delta t\mathcal{G}A\mathcal{L})^{-1}\mathcal{G}A\mathcal{L}\Phi^n+\Delta t(I-\Delta t\mathcal{G}A\mathcal{L})^{-1}\mathcal{G}B^{n}.
\endaligned
\end{equation}
Then, we can compute $r^{n+1}$ by any $\phi_i^{n+1}$:
\begin{equation}\label{se3-esav-e8}
r^{n+1}=\displaystyle\exp\left[\ln(r^n)+\left(b^{\widetilde{r}^{n+\frac12},\widetilde{\phi}_i^{n+\frac12}},\phi_i^{n+1}-\phi_i^{n}\right)\right].
\end{equation}
\begin{remark}\label{esav-re2}
The second-order E-SAV scheme \eqref{se3-esav-e5} based on Crank-Nicolson can be implemented sequentially as follows: (i) Compute the initial values of $\phi_i^0$ for $i=1,\ldots, k$ and $r^0$; (ii) Compute $\tilde{\phi}_i^{\frac{1}{2}}$ for $i=1,\ldots, k$ and $\tilde{r}^{\frac{1}{2}}$; (iii) Compute $\Phi^1$ from \eqref{se3-esav-e7}; (iv) Compute $r^1$ from \eqref{se3-esav-e8}; (v) Compute $\Phi^n$ from \eqref{se3-esav-e7} for $n\geq2$; (iv) Compute $r^n$ from \eqref{se3-esav-e8} for $n\geq2$.
\end{remark}

From above remark, it is not difficult to find that $\Phi^{n+1}$ and $r^{n+1}$ can be computed sequentially, we do not need to compute the inner products $(b^{\widetilde{r}^{n+\frac12},\widetilde{\phi}_i^{n+\frac12}},\phi_i^{n+1})$ for $i=1,\ldots, k$ such as classical SAV scheme before obtaining $\Phi^{n+1}$ which simplified the calculation greatly.
\section{Modified E-SAV approach}
In calculation, we notice that the exponential function is a rapidly increasing function which carries the risk of failure for the E-SAV approach. In this section, by adding a positive constant $C$ in exponential scalar auxiliary variable, we can improve it greatly. In particular, define a new exponential scalar auxiliary variable:
\begin{equation}\label{mesav-e1}
\aligned
r(t)=\exp\left(\frac{E_1(\phi)}{C}\right)=\exp\left(\frac1C\int_\Omega F(\phi)d\textbf{x}\right).
\endaligned
\end{equation}

Then, the phase field system \eqref{esav-e2} can be transformed as follows:
\begin{equation}\label{mesav-e2}
  \left\{
   \begin{array}{rll}
\displaystyle\frac{\partial \phi}{\partial t}&=&\mathcal{G}\mu,\\
\mu&=&\displaystyle\mathcal{L}\phi+b^{r,\phi},\\
\displaystyle\frac{d\ln r}{dt}&=&\displaystyle\frac1C(b^{r,\phi},\phi_t),
   \end{array}
   \right.
\end{equation}
where we set $b^{r,\phi}=[rF'(\phi)]/\exp\left(E_1(\phi)/C\right)$.

Taking the inner products of the above first two equations with $\mu$, $\phi_t$, respectively, and combining them with the third equation, we obtain that the above equivalent system satisfies a modified energy dissipation law:
\begin{equation*}
\frac{d}{dt}\left[\frac12(\phi,\mathcal{L}\phi)+C\ln r\right]=(\mathcal{G}\mu,\mu)\leq0.
\end{equation*}

Similar first-order and second-order discrete schemes can be obtained immediately. For example, the first-order E-SAV scheme can be written:
\begin{equation}\label{mesav-first-e1}
  \left\{
   \begin{array}{rll}
\displaystyle\frac{\phi^{n+1}-\phi^{n}}{\Delta t}&=&\mathcal{G}\mu^{n+1},\\
\mu^{n+1}&=&\displaystyle\mathcal{L}\phi^{n+1}+b^{r^n,\phi^n},\\
\displaystyle\frac{\ln(r^{n+1})-\ln(r^n)}{\Delta t}&=&\displaystyle\frac1C\left(b^{r^n,\phi^n},\frac{\phi^{n+1}-\phi^{n}}{\Delta t}\right).
   \end{array}
   \right.
\end{equation}

The positive constant $C$ is not difficult to obtain. For phase field models, the dissipative energy law means $\frac{d}{dt}E(\phi)\leq0$. Then, an obvious property will hold as follows
\begin{equation}\label{mesav-e3}
E(\phi(\textbf{x},0))\geq E(\phi(\textbf{x},t)), \quad \forall \textbf{x}\in\Omega,t\geq0.
\end{equation}

Considering the definition of the energy and noting that $\mathcal{L}$ is a symmetric non-negative linear operator, it is not difficult to obtain the following inequality
\begin{equation}\label{mesav-e4}
E(\phi(\textbf{x},0))-E_1(\phi(\textbf{x},t))=E(\phi(\textbf{x},0))-E(\phi(\textbf{x},t))+(\phi,\mathcal{L}\phi)\geq(\phi,\mathcal{L}\phi)\geq0, \quad \forall \textbf{x}\in\Omega,t\geq0.
\end{equation}

Thus, $C=|E_1(\phi(\textbf{x},0))|$, $C=|E(\phi(\textbf{x},0))|$ and a very big positive constant will satisfy requirements.

\section{Multiple E-SAV approach}
Many complex phase field models include two or more unknown variables and nonlinear terms. A single scalar auxiliary variable cannot adequately describe the two or more evolution processes. In \cite{cheng2018multiple}, the authors consider multiple SAV approach for phase-field vesicle membrane model. To enhance the applicability of the proposed E-SAV approach, we construct multiple E-SAV (ME-SAV) approach in this section in a general setting.

Mathematically, the complex phase field model is derived from the functional variation of free energy. In general, the free energy $E(\phi)$ contains the sum of an integral phase of some nonlinear functionals and a quadratic term:
\begin{equation*}
E(\phi)=\frac12(\phi,\mathcal{L}\phi)+\int_\Omega\sum\limits_{i=1}^kF_i(\phi)d\textbf{x},
\end{equation*}
where $\mathcal{L}$ is a symmetric non-negative linear operator. Denote the chemical potential $\mu=\frac{\delta E}{\delta \phi}$, then, the phase field models from the energetic variation of the energy functional $E(\phi)$ can be obtained as follows:
\begin{equation}\label{me-sav-e1}
  \left\{
   \begin{array}{rll}
\displaystyle\frac{\partial \phi}{\partial t}&=&\mathcal{G}\mu,\\
\mu&=&\displaystyle\mathcal{L}\phi+\sum\limits_{i=1}^kF_i'(\phi).
   \end{array}
   \right.
\end{equation}
Introduce the following exponential scalar auxiliary variables
\begin{equation}\label{me-sav-e2}
r_i(t)=\exp\left(\frac{1}{C}\int_\Omega F_i(\phi)d\textbf{x}\right),
\end{equation}
where $C$ is a big constant to make $r_i(t)$ be a not very big number for every $t$. Similar as before, an equivalent system of \eqref{me-sav-e1} can be written as follows:
\begin{equation}\label{me-sav-e3}
  \left\{
   \begin{array}{rll}
\displaystyle\frac{\partial \phi}{\partial t}&=&\mathcal{G}\mu,\\
\mu&=&\displaystyle\mathcal{L}\phi+\sum\limits_{i=1}^kb_i^{r_i,\phi}\\
\displaystyle\frac{d\ln r_i}{dt}&=&\displaystyle\frac{1}{C}\left(b_i^{r_i,\phi},\frac{\partial \phi}{\partial t}\right)\\
b_i^{r_i,\phi}&=&\displaystyle\frac{r_i(t)}{\exp(\frac{1}{C}\int_\Omega F_i(\phi)d\textbf{x})}F_i'(\phi).
   \end{array}
   \right.
\end{equation}
Taking the inner products of the above first two equations with $\mu$, $\phi_t$, respectively, and summing up for $i$ from 1 to $k$ for the third equation, we obtain that the above equivalent system satisfies a modified energy dissipation law:
\begin{equation*}
\frac{d}{dt}\left[\frac12(\phi,\mathcal{L}\phi)+C\sum\limits_{i=1}^k\ln r_i\right]=(\mathcal{G}\mu,\mu)\leq0.
\end{equation*}

The first order scheme derived by the backward Euler¡¯s method and the second order scheme based on Crank-Nicolson formula are very easy to obtain. In detail, the explicit first order scheme is as follows:
\begin{equation}\label{me-sav-first-e1}
  \left\{
   \begin{array}{rll}
\displaystyle\frac{\phi^{n+1}-\phi^{n}}{\Delta t}&=&\mathcal{G}\mu^{n+1},\\
\mu^{n+1}&=&\displaystyle\mathcal{L}\phi^{n+1}+\sum\limits_{i=1}^kb_i^{r_i^n,\phi^n},\\
\displaystyle\frac{\ln(r_i^{n+1})-\ln(r_i^n)}{\Delta t}&=&\displaystyle\left(b_i^{r_i^n,\phi^n},\frac{\phi^{n+1}-\phi^{n}}{\Delta t}\right),\\
b_i^{r_i^n,\phi^n}&=&\displaystyle\frac{r_i^n}{\exp(\frac{1}{C}\int_\Omega F_i(\phi^n)d\textbf{x})}F_i'(\phi^n),
   \end{array}
   \right.
\end{equation}
and the explicit second order scheme is as the following:
\begin{equation}\label{me-sav-second-e1}
  \left\{
   \begin{array}{rll}
\displaystyle\frac{\phi^{n+1}-\phi^{n}}{\Delta t}&=&\mathcal{G}\mu^{n+1},\\
\mu^{n+1}&=&\displaystyle\mathcal{L}\phi^{n+1}+\sum\limits_{i=1}^kb_i^{\widetilde{r}_i^{n+\frac12},\widetilde{\phi}^{n+\frac12}},\\
\displaystyle\frac{\ln(r_i^{n+1})-\ln(r_i^n)}{\Delta t}&=&\displaystyle\left(b_i^{\widetilde{r}_i^{n+\frac12},\widetilde{\phi}^{n+\frac12}},\frac{\phi^{n+1}-\phi^{n}}{\Delta t}\right),\\
b_i^{r_i^{n+\frac12},\widetilde{\phi}^{n+\frac12}}&=&\displaystyle\frac{\widetilde{r}_i^{n+\frac12}}{\exp(\frac{1}{C}\int_\Omega F_i(\widetilde{\phi}^{n+\frac12})d\textbf{x})}F_i'(\widetilde{\phi}^{n+\frac12}),
   \end{array}
   \right.
\end{equation}
where $\tilde{\phi}^{n+\frac{1}{2}}$ is any explicit $O(\Delta t^2)$ approximation for $\phi(t^{n+\frac{1}{2}})$, and $\tilde{r}_i^{n+\frac{1}{2}}$ is any explicit $O(\Delta t^2)$ approximation for $r_i(t^{n+\frac{1}{2}})$, which can be flexible according to the problem. Here, we choose
\begin{equation}\label{me-sav-second-e2}
\aligned
&\tilde{\phi}^{n+\frac{1}{2}}=\frac32\phi^n-\frac12\phi^{n-1}, \quad n\geq1,\\
&\tilde{r}_i^{n+\frac{1}{2}}=\frac32r_i^n-\frac12r_i^{n-1}, \quad n\geq1,\quad i=1,2,\ldots,k.
\endaligned
\end{equation}
It is not difficult to obtain the unconditional energy stability of above two schemes.
\begin{theorem}\label{me-sav-th1}
The scheme \eqref{me-sav-first-e1} for the equivalent phase field system \eqref{me-sav-e3} is linear, first-order accurate, unconditionally energy stable in the sense that
\begin{equation*}
\aligned
\frac{1}{\Delta t}\left[E_{ME-SAV-1st}^{n+1}-E^{n}_{ME-SAV-1st}\right]\leq(\mathcal{G}\mu^{n+\frac12},\mu^{n+\frac12})-\frac{1}{2\Delta t}(\phi^{n+1}-\phi^n,\mathcal{L}(\phi^{n+1}-\phi^n))\leq0.
\endaligned
\end{equation*}
where the modified discrete version of the energy is defined by
\begin{equation*}
\aligned
E_{ME-SAV-1st}^{n}=\frac12(\phi^n,\mathcal{L}\phi^{n})+C\sum\limits_{i=1}^k\ln(r_i^n),
\endaligned
\end{equation*}
and scheme \eqref{me-sav-second-e1} for the equivalent phase field system \eqref{me-sav-e3} is linear, second-order accurate, unconditionally energy stable in the sense that
\begin{equation*}
\aligned
\frac{1}{\Delta t}\left[E_{ME-SAV-CN}^{n+1}-E^{n}_{ME-SAV-CN}\right]\leq(\mathcal{G}\mu^{n+\frac12},\mu^{n+\frac12})\leq0.
\endaligned
\end{equation*}
where the modified discrete version of the energy is defined by
\begin{equation*}
\aligned
E_{ME-SAV-CN}^{n}=\frac12(\phi^n,\mathcal{L}\phi^{n})+C\sum\limits_{i=1}^k\ln(r_i^n).
\endaligned
\end{equation*}
\end{theorem}
\subsection{ME-SAV approach for the Cahn-Hilliard phase field model of the binary fluid-surfactant system}
In this section, we consider the proposed ME-SAV approach for the commonly used binary fluid-surfactant phase field model with two coupled Cahn-Hilliard equations. In particular, the free energy of the system is given as follows:
\begin{equation}\label{bfss-e1}
E(\phi,\rho)=\int_\Omega\displaystyle\left(\frac12|\nabla\phi|^2+\frac{\alpha}{2}(\Delta \phi)^2+\frac{1}{4\epsilon^2}F(\phi)+\frac{\beta}{2}|\nabla\rho|^2+\frac{1}{4\eta^2}G(\rho)-\theta\rho|\nabla\phi|^2\right)d\textbf{x},
\end{equation}
where the double well Ginzburg-Landau potential $F(\phi)=(\phi^2-1)^2$ and $G(\rho)=\rho^2(\rho-\rho_s)^2$, where $\alpha$, $\beta$, $\epsilon$, $\rho_s$ and $\theta$ are all positive parameters.

Considering a gradient flow in $H^{-1}$ which is derived from the functional variation of free energy \eqref{bfss-e1} and introducing two chemical potentials $\nu_\phi$ and $\mu_\rho$, one can obtain the following Cahn-Hilliard phase field model of the binary fluid-surfactant system:
\begin{equation}\label{bfss-e2}
  \left\{
   \begin{array}{rll}
\displaystyle\frac{\partial \phi}{\partial t}&=&M_\phi\Delta\mu_\phi,\\
\mu_\phi&=&\displaystyle-\Delta\phi+\alpha\Delta^2\phi+\frac{1}{\epsilon^2}F'(\phi)+2\theta\nabla\cdot(\rho\nabla\phi),\\
\displaystyle\frac{\partial \rho}{\partial t}&=&M_\rho\Delta\mu_\rho,\\
\mu_\rho&=&\displaystyle-\beta\Delta\rho+\frac{1}{\eta^2}G'(\rho)+\theta|\nabla\phi|^2.
   \end{array}
   \right.
\end{equation}

The system satisfies an energy dissipation law:
\begin{equation*}
\frac{d}{dt}E(\phi,\rho)=-M_\phi\|\nabla\mu_\phi\|^2-M_\rho\|\nabla\mu_\rho\|^2\leq0.
\end{equation*}

One can notice that the above coupled system has two nonlinear terms $F'(\phi)$ and $G'(\rho)$ which makes it very hard to handle with only one SAV. Thus, we introduce two exponential scalar auxiliary variables as follows:
\begin{equation}\label{bfss-e3}
\aligned
&r(t)={\exp(E_F)}={\exp\left(\int_\Omega F(\phi)d\textbf{x}\right)},\\
&q(t)={\exp(E_G)}={\exp\left(\int_\Omega G(\rho)d\textbf{x}\right)}.
\endaligned
\end{equation}

Combining the equations \eqref{bfss-e3} with the coupled system \eqref{bfss-e2}, and define
\begin{equation*}
\aligned
&b^{r,\phi}=\frac{r(t)}{\exp(E_F)}F^{'}(\phi),\\
&d^{q,\rho}=\frac{q(t)}{\exp(E_G)}G^{'}(\rho).
\endaligned
\end{equation*}
we can obtain the following equivalent PDE system as follows:
\begin{equation}\label{bfss-e4}
  \left\{
   \begin{array}{rll}
\displaystyle\frac{\partial \phi}{\partial t}&=&M_\phi\Delta\mu_\phi,\\
\mu_\phi&=&\displaystyle-\Delta\phi+\alpha\Delta^2\phi+\frac{1}{\epsilon^2}b^{r,\phi}+2\theta\nabla\cdot(\rho\nabla\phi),\\
\displaystyle\frac{\partial \rho}{\partial t}&=&M_\rho\Delta\mu_\rho,\\
\mu_\rho&=&\displaystyle-\beta\Delta\rho+\frac{1}{\eta^2}d^{q,\rho}+\theta|\nabla\phi|^2,\\
\displaystyle\frac{d\ln r}{dt}&=&\displaystyle(b^{r,\phi},\frac{\partial \phi}{\partial t}),\\
\displaystyle\frac{d\ln q}{dt}&=&\displaystyle(d^{q,\rho},\frac{\partial \rho}{\partial t}).
   \end{array}
   \right.
\end{equation}
The free energy \eqref{bfss-e1} can be rewritten as
\begin{equation}\label{bfss-e5}
E(\phi,\rho,r,q)=\int_\Omega\displaystyle\left(\frac12|\nabla\phi|^2+\frac{\alpha}{2}(\Delta \phi)^2+\frac{\beta}{2}|\nabla\rho|^2-\theta\rho|\nabla\phi|^2\right)d\textbf{x}+\frac{1}{4\epsilon^2}\ln r+\frac{1}{4\eta^2}\ln q.
\end{equation}

The phase field model is usually supplemented with the periodic boundary condition. So, for system \eqref{bfss-e4}, we assume that the density field $\phi$ and $\rho$ are periodic on $\Omega$. The initial conditions read as
\begin{equation}\label{bfss-e6}
\phi|_{t=0}=\phi_0,\quad \rho_{t=0}=\rho_0,\quad r|_{t=0}=\exp(E_F(\phi_0)),\quad q|_{t=0}=\exp(E_G(\rho_0)).
\end{equation}
Taking the $L^2$ inner product of the first four equations in \eqref{bfss-e4} with $\mu_\phi$, $\phi_t$, $\mu_\rho$, and $\rho_t$, respectively, and combining them with the last two equations in \eqref{bfss-e4}, we can obtain the energy dissipation law immediately:
 \begin{equation*}
\frac{d}{dt}E(\phi,\rho,r,q)=-M_\phi\|\nabla\mu_\phi\|^2-M_\rho\|\nabla\mu_\rho\|^2\leq0.
\end{equation*}

Next, we will give a first-order ME-SAV scheme and prove the unconditional energy stability. The second-order scheme based on ME-SAV approach can be obtained similarly as before. In detail, the first-order scheme can be written as follows:
\begin{flalign}\label{bfss-e7}
\begin{split}
\bm{Step~~I:}\left\{
   \begin{array}{rll}
\displaystyle\frac{\rho^{n+1}-\rho^n}{\Delta t}&=&M_\rho\Delta\mu^{n+1}_\rho,\\
\mu^{n+1}_\rho&=&\displaystyle-\beta\Delta\rho^{n+1}+\frac{1}{\eta^2}d^{q^n,\rho^n}+\theta|\nabla\phi^n|^2,\\
\displaystyle\frac{\ln q^{n+1}-\ln q^n}{\Delta t}&=&\displaystyle(d^{q^n,\rho^n},\frac{\rho^{n+1}-\rho^n}{\Delta t}).
   \end{array}
   \right.
\end{split}&
\end{flalign}

\begin{flalign}\label{bfss-e8}
\begin{split}
\bm{Step~~II:}\left\{
   \begin{array}{rll}
\displaystyle\frac{\phi^{n+1}-\phi^n}{\Delta t}&=&M_\phi\Delta\mu^{n+1}_\phi,\\
\mu^{n+1}_\phi&=&\displaystyle-\Delta\phi^{n+1}+\alpha\Delta^2\phi^{n+1}+\frac{1}{\epsilon^2}b^{r^n,\phi^n}+2\theta\nabla\cdot(\rho^{n+1}\nabla\frac{\phi^n+1+\phi^n}{2}),\\
\displaystyle\frac{\ln r^{n+1}-\ln r^n}{\Delta t}&=&\displaystyle(b^{r^n,\phi^n},\frac{\phi^{n+1}-\phi^n}{\Delta t}).
   \end{array}
   \right.
\end{split}&
\end{flalign}
\begin{remark}\label{bfss-re1}
The computations of $\phi$, $\rho$, $r$ and $q$ are totally decoupled by above two steps. Firstly, we only need $\phi^n$ to compute $\rho^{n+1}$ in step $i$. Then, $q^{n+1}$ can be obtained by computing $(d^{q^n,\rho^n},\rho^{n+1}-\rho^n)$. Next, when computing $\phi^{n+1}$ in step $ii$, $\rho^{n+1}$ has already been obtained from step $i$. Last, $r^{n+1}$ can be obtained by computing $(b^{r^n,\phi^n},\phi^{n+1}-\phi^n)$.
\end{remark}

\begin{theorem}\label{bfss-th1}
The scheme \eqref{bfss-e7}-\eqref{bfss-e8} for the equivalent system \eqref{bfss-e4} is unconditionally energy stable in the sense that
\begin{equation*}
\aligned
\frac{1}{\Delta t}\left[E_{ME-SAV}^{n+1}-E^{n}_{ME-SAV}\right]\leq-M_\phi\|\nabla\mu^{n+1}_\phi\|^2-M_\rho\|\nabla\mu^{n+1}_\rho\|^2\leq0.
\endaligned
\end{equation*}
where the modified discrete version of the energy is defined by
\begin{equation*}
\aligned
E_{ME-SAV}^{n}=\frac{\beta}{2\Delta t}\|\nabla\rho^{n}\|^2+\frac{1}{\eta^2}{\ln q^{n}}+\frac{1}{2\Delta t}\|\nabla\phi^{n}\|^2+\frac{\alpha}{2\Delta t}\|\Delta\phi^{n}\|^2+\frac{1}{\epsilon^2\Delta t}{\ln r^{n}}-\frac{\theta}{\Delta t}(|\nabla\phi^{n}|^2,\rho^{n}).
\endaligned
\end{equation*}
\end{theorem}
\begin{proof}
By taking the $L^2$ inner product with $\mu^{n+1}_\rho$ of the first equation in \eqref{bfss-e7}, we obtain
\begin{equation}\label{bfss-e9}
\aligned
\frac{1}{\Delta t}(\rho^{n+1}-\rho^n,\mu_\rho^{n+1})=-M_\rho\|\nabla\mu^{n+1}_\rho\|^2.
\endaligned
\end{equation}
By taking the $L^2$ inner product of the second equation in \eqref{bfss-e7} with $\frac{1}{\Delta t}(\rho^{n+1}-\rho^n)$, and noticing that
\begin{equation*}
(x,x-y)=\frac{1}{2}|x|^2+\frac12|y|^2+\frac12|x-y|^2,
\end{equation*}
then, combining them with the third equation in \eqref{bfss-e7}, we obtain
\begin{equation}\label{bfss-e10}
\aligned
\frac{1}{\Delta t}(\rho^{n+1}-\rho^n,\mu_\rho^{n+1})=
&\frac{\beta}{2\Delta t}(\|\nabla\rho^{n+1}\|^2-\|\nabla\rho^{n}\|^2+\|\nabla\rho^{n+1}-\nabla\rho^n\|^2)+\frac{1}{\eta^2\Delta t}(\ln q^{n+1}-\ln q^n)\\
&+\frac{\theta}{\Delta t}(|\nabla\phi^n|^2,\rho^{n+1}-\rho^n).
\endaligned
\end{equation}
By taking the $L^2$ inner product with $\mu^{n+1}_\phi$ of the first equation in \eqref{bfss-e8}, we obtain
\begin{equation}\label{bfss-e11}
\aligned
\frac{1}{\Delta t}(\phi^{n+1}-\phi^n,\mu_\phi^{n+1})=-M_\phi\|\nabla\mu^{n+1}_\phi\|^2.
\endaligned
\end{equation}
By taking the $L^2$ inner product of the second equation in \eqref{bfss-e8} with $\frac{1}{\Delta t}(\phi^{n+1}-\phi^n)$, and combining them with the third equation in \eqref{bfss-e8}, we obtain
\begin{equation}\label{bfss-e12}
\aligned
\frac{1}{\Delta t}(\phi^{n+1}-\phi^n,\mu_\phi^{n+1})=
&\frac{1}{2\Delta t}(\|\nabla\phi^{n+1}\|^2-\|\nabla\phi^{n}\|^2+\|\nabla\phi^{n+1}-\nabla\phi^n\|^2)\\
&\frac{\alpha}{2\Delta t}(\|\Delta\phi^{n+1}\|^2-\|\Delta\phi^{n}\|^2+\|\Delta\phi^{n+1}-\Delta\phi^n\|^2)\\
&+\frac{1}{\epsilon^2\Delta t}(\ln r^{n+1}-\ln r^n)+\frac{\theta}{\Delta t}(|\nabla\phi^{n+1}|^2-|\nabla\phi^n|^2,\rho^{n+1}).
\endaligned
\end{equation}
Combining the equations \eqref{bfss-e9}-\eqref{bfss-e12} and using the following inequality \cite{yang2018numerical}:
\begin{equation}\label{bfss-e13}
\aligned
\frac{\theta}{\Delta t}(|\nabla\phi^n|^2,\rho^{n+1}-\rho^n)+\frac{\theta}{\Delta t}(|\nabla\phi^{n+1}|^2-|\nabla\phi^n|^2,\rho^{n+1})=\frac{\theta}{\Delta t}(|\nabla\phi^{n+1}|^2,\rho^{n+1})-\frac{\theta}{\Delta t}(|\nabla\phi^n|^2,\rho^{n}),
\endaligned
\end{equation}
we have
\begin{equation}\label{bfss-e14}
\aligned
&\frac{1}{\Delta t}\left[E_{ME-SAV}^{n+1}-E^{n}_{ME-SAV}\right]\\
&=\frac{\beta}{2\Delta t}\|\nabla\rho^{n+1}\|^2+\frac{1}{\eta^2}{\ln q^{n+1}}+\frac{1}{2\Delta t}\|\nabla\phi^{n+1}\|^2+\frac{\alpha}{2\Delta t}\|\Delta\phi^{n+1}\|^2+\frac{1}{\epsilon^2\Delta t}{\ln r^{n+1}}-\frac{\theta}{\Delta t}(|\nabla\phi^{n+1}|^2,\rho^{n+1})\\
&-\frac{\beta}{2\Delta t}\|\nabla\rho^{n}\|^2-\frac{1}{\eta^2}{\ln q^{n}}-\frac{1}{2\Delta t}\|\nabla\phi^{n}\|^2-\frac{\alpha}{2\Delta t}\|\Delta\phi^{n}\|^2-\frac{1}{\epsilon^2\Delta t}{\ln r^{n}}+\frac{\theta}{\Delta t}(|\nabla\phi^{n}|^2,\rho^{n})\\
&=-M_\phi\|\nabla\mu^{n+1}_\phi\|^2-M_\rho\|\nabla\mu^{n+1}_\rho\|^2-\frac{\beta}{2\Delta t}\|\nabla\rho^{n+1}-\nabla\rho^n\|^2-\frac{1}{2\Delta t}\|\nabla\phi^{n+1}-\nabla\phi^n\|^2-\frac{\alpha}{2\Delta t}\|\Delta\phi^{n+1}-\Delta\phi^n\|^2.\\
\endaligned
\end{equation}
\end{proof}
\section{Examples and discussion}
In this section, we use several numerical examples to demonstrate the accuracy, energy stability and efficiency of the proposed schemes when applying to the some classical phase field models such as Allen-Cahn equation, Cahn-Hilliard equation, phase field crystal model and so on. In all examples, we assume periodic boundary conditions and use a fourier spectral method for space variables. To test the efficiency of fast calculation, all the solvers are implemented using Matlab and all the numerical experiments are performed on a computer with 8-GB memory.

\subsection{Allen-Cahn and Cahn-Hilliard equations}
Both Allen-Cahn and Cahn-Hilliard equations are very classical phase field models and have been widely used in many fields involving physics, materials science, finance and image processing \cite{chen2018accurate,chen2018power,du2018stabilized}.

In particular, Allen-Cahn equation has been widely used to model various phenomena in nature which was introduced by M. Allen and W. Cahn in \cite{allen1979microscopic}:
\begin{equation}\label{example-e1}
  \left\{
   \begin{array}{rlr}
\displaystyle\frac{\partial \phi}{\partial t}&=-M\mu,     &(\textbf{x},t)\in\Omega\times J,\\
                                          \mu&=\displaystyle-\Delta \phi+\frac{1}{\epsilon^2}f(\phi),&(\textbf{x},t)\in\Omega\times J,
   \end{array}
   \right.
  \end{equation}
and Cahn-Hilliard equation is as follows which was introduced by John W. Cahn and John E. Hilliard in \cite{cahn1958free} to describe the process of phase separation:
\begin{equation}\label{example-e2}
  \left\{
   \begin{array}{rlr}
\displaystyle\frac{\partial \phi}{\partial t}&=M\Delta\mu,     &(\textbf{x},t)\in\Omega\times J,\\
                                          \mu&=\displaystyle-\epsilon\Delta \phi+\frac{1}{\epsilon}f(\phi),&(\textbf{x},t)\in\Omega\times J,
   \end{array}
   \right.
  \end{equation}
where $J=(0,T]$, $M$ is the mobility constant, $\mu$ is the chemical potential, and $f(\phi)=F'(\phi)$, $F(\phi)$ is a non-convex potential density function. In this paper, we consider the following double well potential function $F(\phi)=\frac14(\phi^2-1)^2$.

\textbf{Example 1}: Consider the above Allen-Cahn and Cahn-Hilliard equations in $\Omega=[0,2\pi]$ with $\epsilon=0.1$, $T=0.032$, $M=1$ in Allen-Cahn equation and $M=0.1$ in Cahn-Hilliard equation, and the following initial condition \cite{ShenA}:
\begin{equation*}
\aligned
\phi(x,y,0)=0.05sin(x)sin(y).
\endaligned
\end{equation*}

We use the Fourier spectral Galerkin method for spatial discretization with $N=128$. The true solution is unknown and we therefore use the
Fourier Galerkin approximation in the case $\Delta t=1e-8$ as a reference solution.

For Allen-Cahn equation, we consider first order time discrete schemes based on both SAV approach in \cite{ShenA} and the proposed E-SAV approach in this article. The computational results are shown in Table \ref{tab:tab1}. The numerical results indicate that both SAV and E-SAV scheme are indeed of first order in time but the later scheme keeps the error smaller. Specially, the E-SAV scheme is much easier to calculate than SAV scheme. The CPU time shows that the E-SAV scheme is about half as time-consuming as SAV scheme.

For Cahn Hilliard model, a comparative study of classical SAV and E-SAV approaches based on second order time discrete schemes is considered in Table \ref{tab:tab2}. The two methods obtain almost identical error and convergence rates. However, the E-SAV scheme also saves half the time compared with SAV scheme.
\begin{table}[h!b!p!]
\small
\centering
\caption{\small The $L_2$ errors, convergence rates for first order scheme in time for SAV and E-SAV approaches of Allen Cahn equation.}\label{tab:tab1}
\begin{tabular}{cccccccccc}
\hline
          &&SAV&&&&E-SAV&\\
\cline{1-8}
$\Delta t$          &$L_2$ error&Rate&Cpu-Time(s)&&$L_2$ error&Rate&Cpu-Time(s)\\
\cline{1-8}
$1.6e-4$ &1.5839e-2   &---   &1.17    &&8.8644e-3   &---   &0.65\\
$8e-5$   &7.9574e-3   &0.9946&2.26    &&4.4254e-3   &1.0022&1.29\\
$4e-5$   &3.9680e-3   &1.0052&4.51    &&2.1899e-3   &1.0149&2.60\\
$2e-5$   &1.9610e-3   &1.0187&9.16    &&1.0681e-3   &1.0358&4.78\\
$1e-5$   &9.5446e-4   &1.0422&17.83   &&5.0627e-4   &1.0771&9.95\\
\cline{1-8}
\end{tabular}
\end{table}

\begin{table}[h!b!p!]
\small
\centering
\caption{\small The $L_2$ errors, convergence rates for second order scheme in time for SAV and E-SAV approaches of Cahn Hilliard equation.}\label{tab:tab2}
\begin{tabular}{cccccccccc}
\hline
          &&SAV&&&&E-SAV&\\
\cline{1-8}
$\Delta t$          &$L_2$ error&Rate&Cpu-Time(s)&&$L_2$ error&Rate&Cpu-Time(s)\\
\cline{1-8}
$1.6e-4$ &5.1526e-8   &---   &1.43    &&5.1474e-8   &---   &0.81\\
$8e-5$   &1.2873e-8   &2.0009&2.94    &&1.2848e-8   &2.0023&1.43\\
$4e-5$   &3.2156e-9   &2.0011&5.07    &&3.2081e-9   &2.0017&2.63\\
$2e-5$   &8.0229e-10  &2.0028&10.32   &&8.0147e-10  &2.0009&5.63\\
$1e-5$   &1.9906e-10  &2.0109&20.12   &&1.9971e-10  &2.0047&10.73\\
\cline{1-8}
\end{tabular}
\end{table}
\textbf{Example 2}: In the following, we solve a benchmark problem for the Allen-Cahn equation which can be seen in many articles such as \cite{ShenA}. We take $\epsilon=0.01$, $M=1$. The initial condition is chosen as
\begin{equation*}
\aligned
\phi_0(x,y,0)=\sum\limits_{i=1}^2-\tanh\left(\frac{\sqrt{(x-x_i)^2+(y-y_i)^2}-R_0}{\sqrt{2}\epsilon}\right)+1.
\endaligned
\end{equation*}
with the radius $R_0=0.19$, $(x_1,y_1)=(0.3,0.5)$ and $(x_2,y_2)=(0.7,0.5)$. Initially, two bubbles, centered at $(0.3,0.5)$ and $(0.7,0.5)$, respectively, are osculating or "kissing".

As is known to all, the Allen-Cahn equation does not conserve mass. So, in Figure \ref{fig:fig1}, we can see that as time evolves, the two bubbles coalesce into a single bubble, then, shrinks and finally disappears. A correct simulation of this phenomenon shows the effectiveness of our E-SAV approach. In Figure \ref{fig:fig2}, we plot the time evolution of the energy functional with different time step size of $\Delta t=0.001$, $0.01$, $0.1$, $1$ and $2$ using the first order scheme based on E-SAV approach. All energy curves show the monotonic decays for all time steps that confirms that the algorithm E-SAV is unconditionally energy stable. Time evolution of the total free energy based on SAV and E-SAV approaches is computed by using the time step $\Delta t=0.001$ in the left figure in Figure \ref{fig:fig2}. These two energy curves are almost identical and they both decay monotonically at all times.
\begin{figure}[htp]
\centering
\subfigure[t=0]{
\includegraphics[width=4cm,height=4cm]{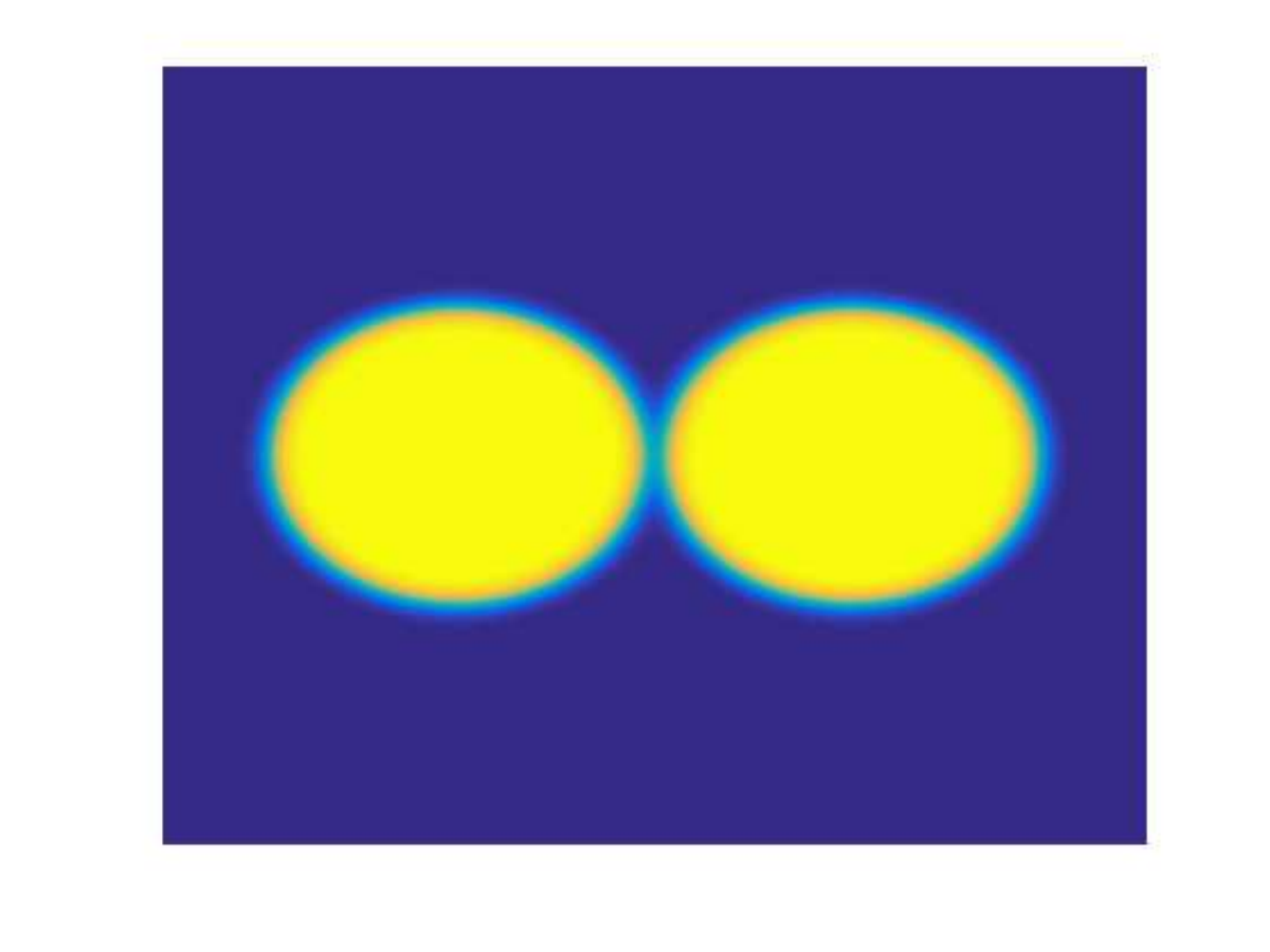}
}
\subfigure[t=0.02]
{
\includegraphics[width=4cm,height=4cm]{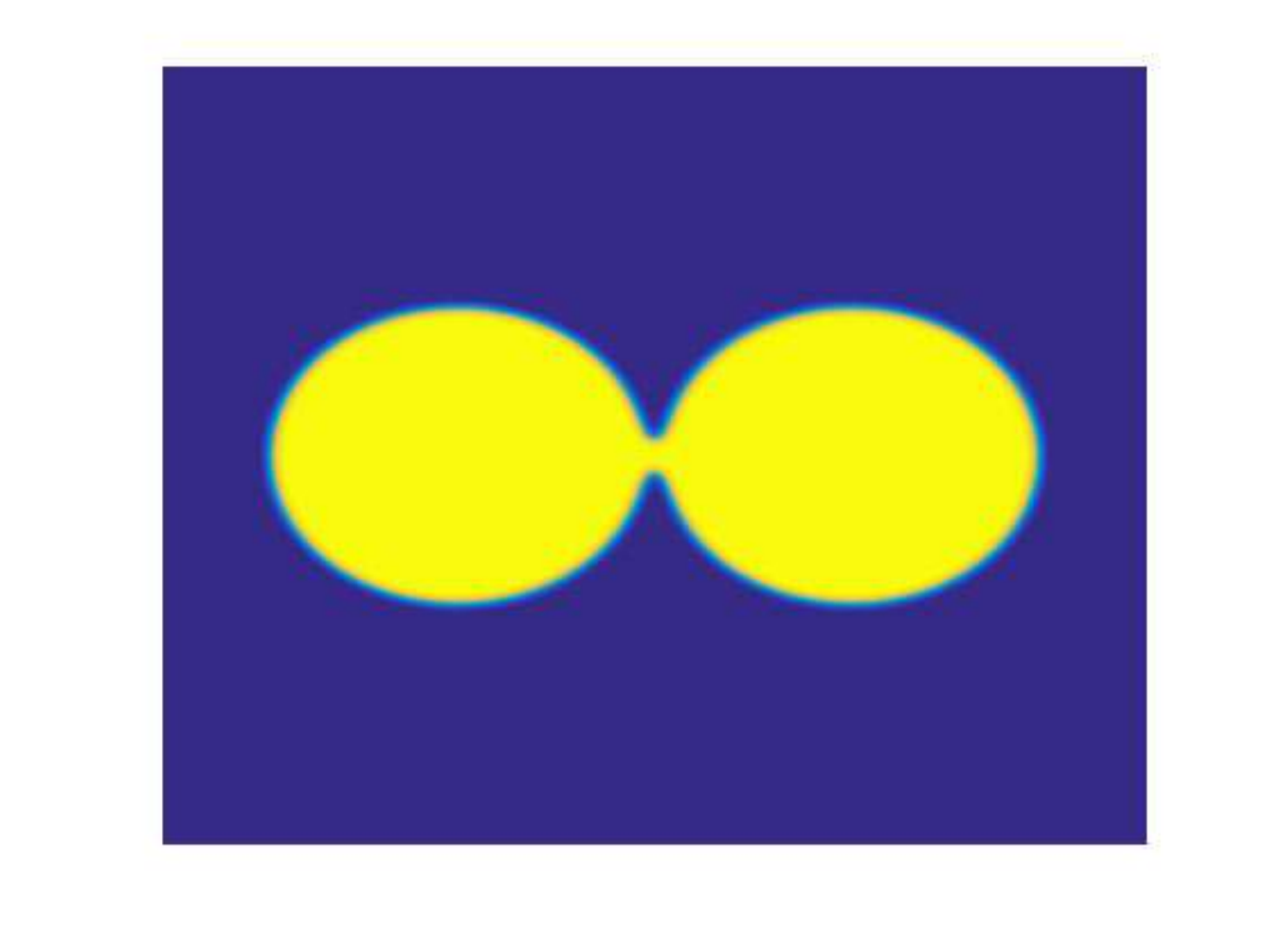}
}
\subfigure[t=0.5]
{
\includegraphics[width=4cm,height=4cm]{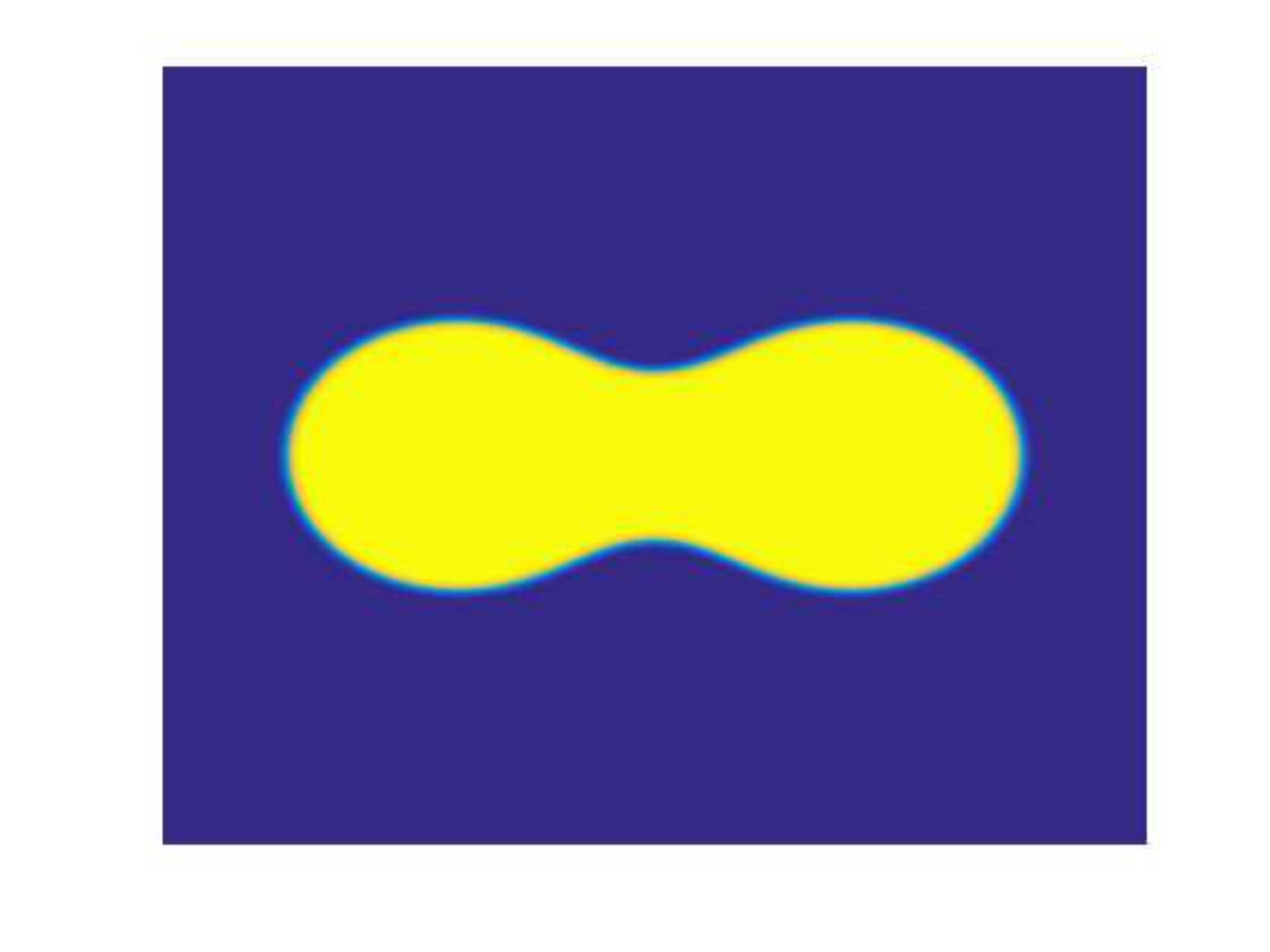}
}
\quad
\subfigure[t=1.5]{
\includegraphics[width=4cm,height=4cm]{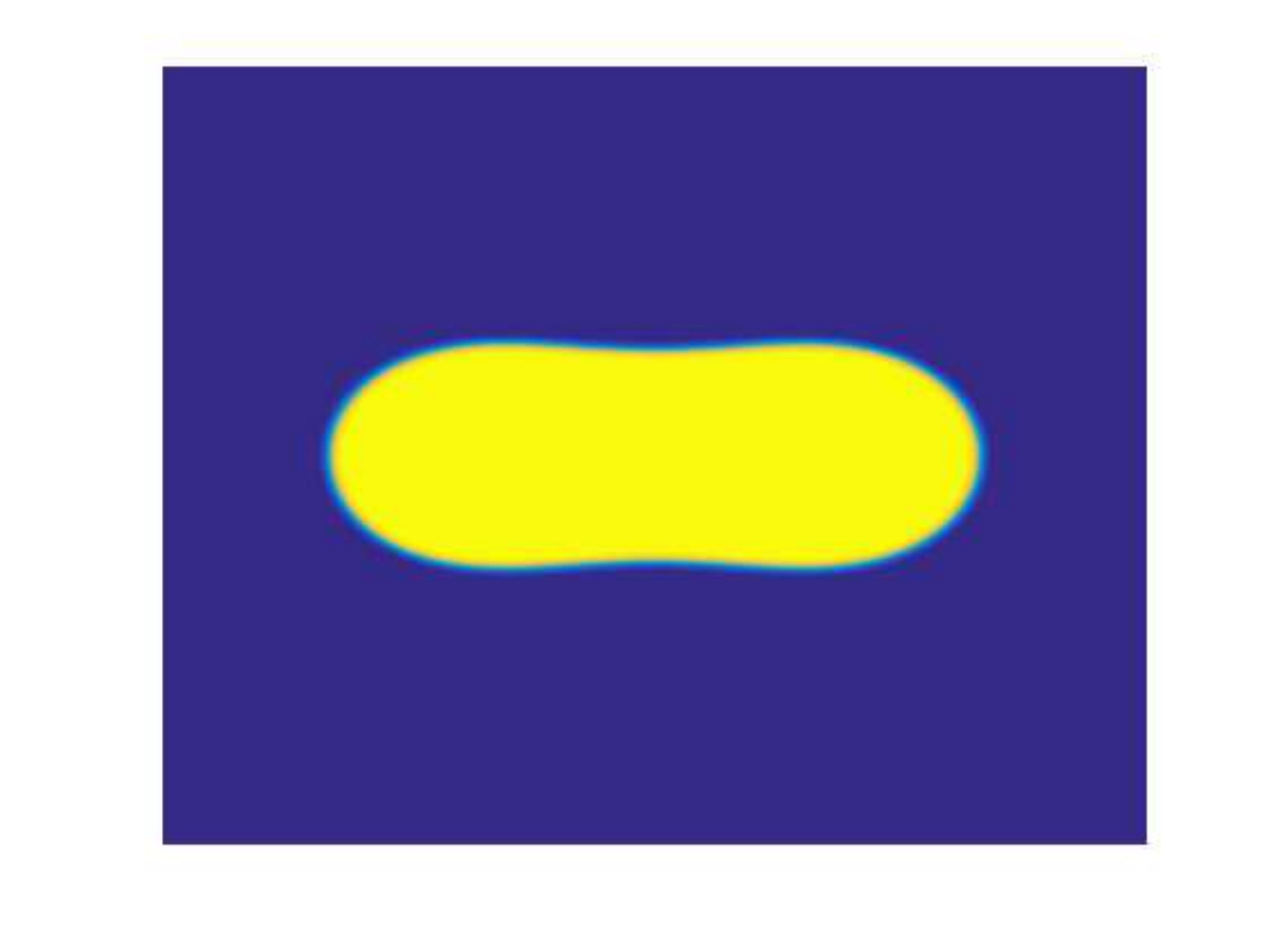}
}
\subfigure[t=4]
{
\includegraphics[width=4cm,height=4cm]{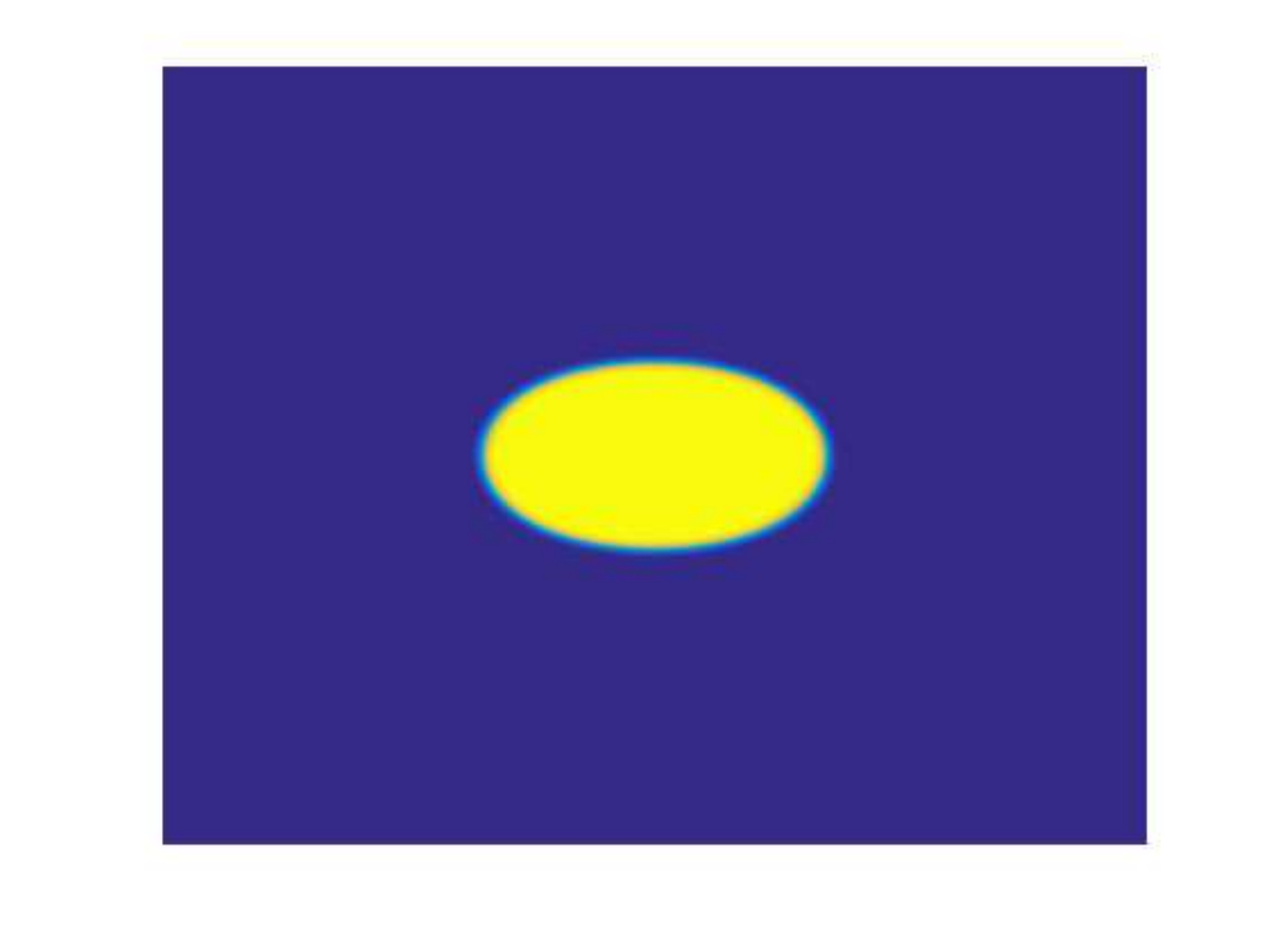}
}
\subfigure[t=5.6]
{
\includegraphics[width=4cm,height=4cm]{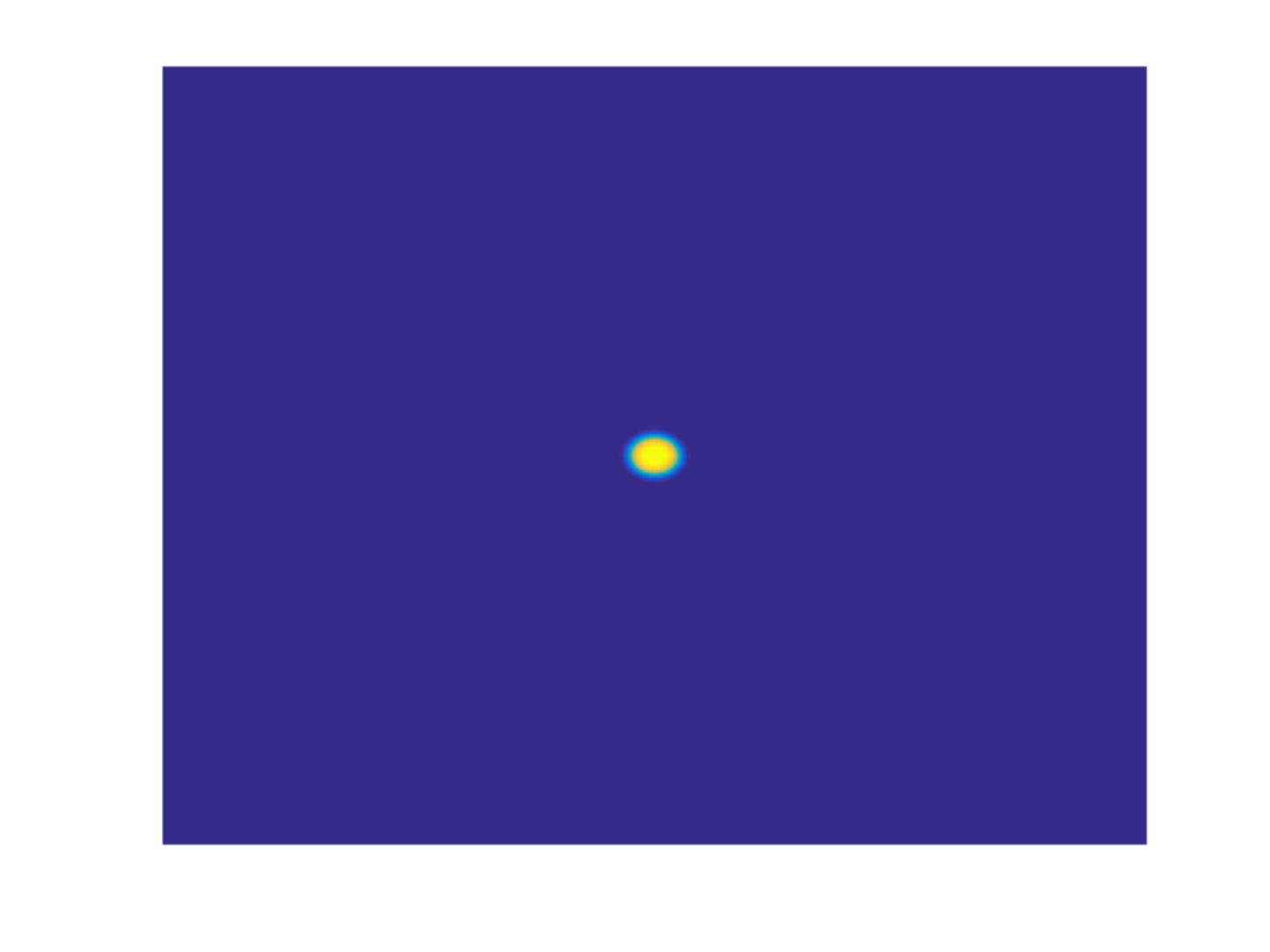}
}
\caption{Snapshots of the phase variable $\phi$ are taken at t=0, 0.02, 0.5, 1.5, 4, 5.6 for example 2.}\label{fig:fig1}
\end{figure}
\begin{figure}[htp]
\centering
\includegraphics[width=7cm,height=7cm]{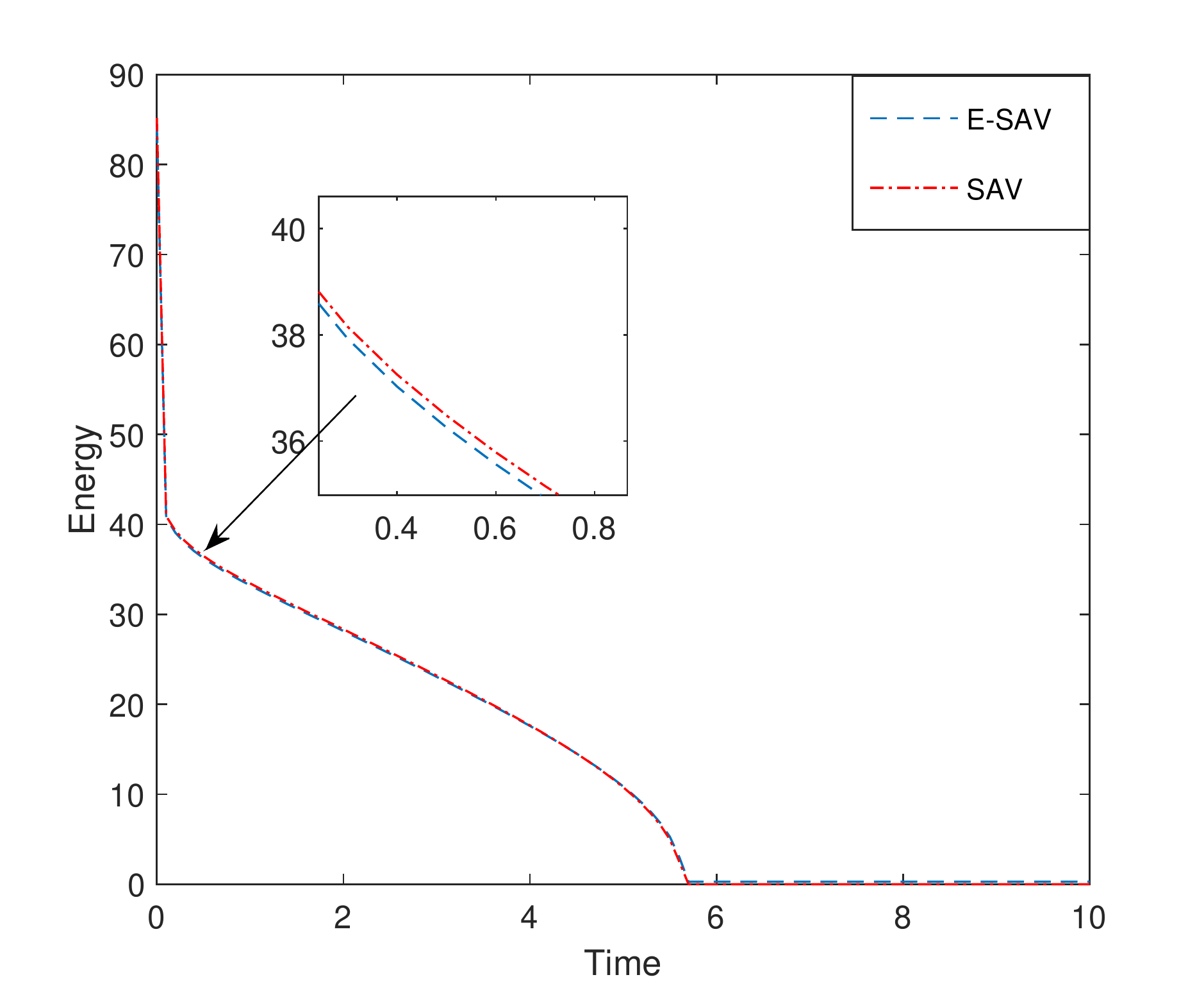}
\includegraphics[width=7cm,height=7cm]{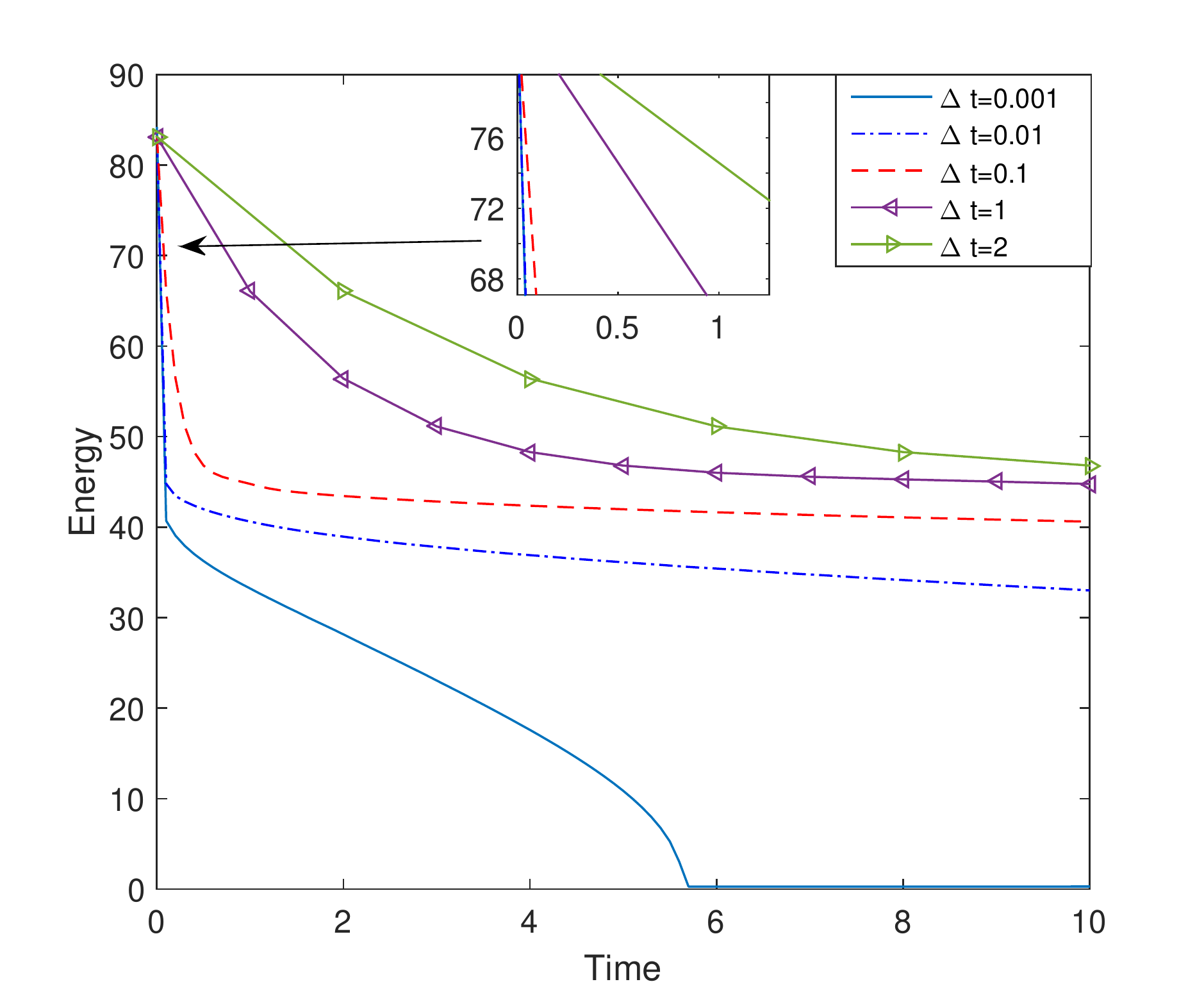}
\caption{Left: energy evolution of E-SAV and SAV approaches for example 2 with $\Delta t=0.001$. Right: time evolution of the energy functional for five different time steps of $\Delta t=0.001$, $0.01$, $0.1$, $1$ and $2$.}\label{fig:fig2}
\end{figure}

\textbf{Example 3}: In the following, we solve a benchmark problem for the Cahn-Hilliard equation on $[0,2\pi)^2$ which can also be seen in many articles such as \cite{shen2018scalar}. We take $\epsilon=0.02$, $M=0.1$ and discretize the space by the Fourier spectral method with $256\times256$ modes. The initial condition is chosen as the following
\begin{equation*}
\aligned
\phi_0(x,y,0)=0.25+0.4Rand(x,y),
\endaligned
\end{equation*}
where $Rand(x,y)$ is a randomly generated function.

Snapshots of the phase variable $\phi$ taken at $t=0.02$, $0.5$, $3$, and $20$ with $\Delta t=0.01$ are shown in Figure \ref{fig:fig3}. The phase separation and coarsening process can be observed very simply which is consistent with the results in \cite{ShenA}.
\begin{figure}[htp]
\centering
\includegraphics[width=4cm,height=4cm]{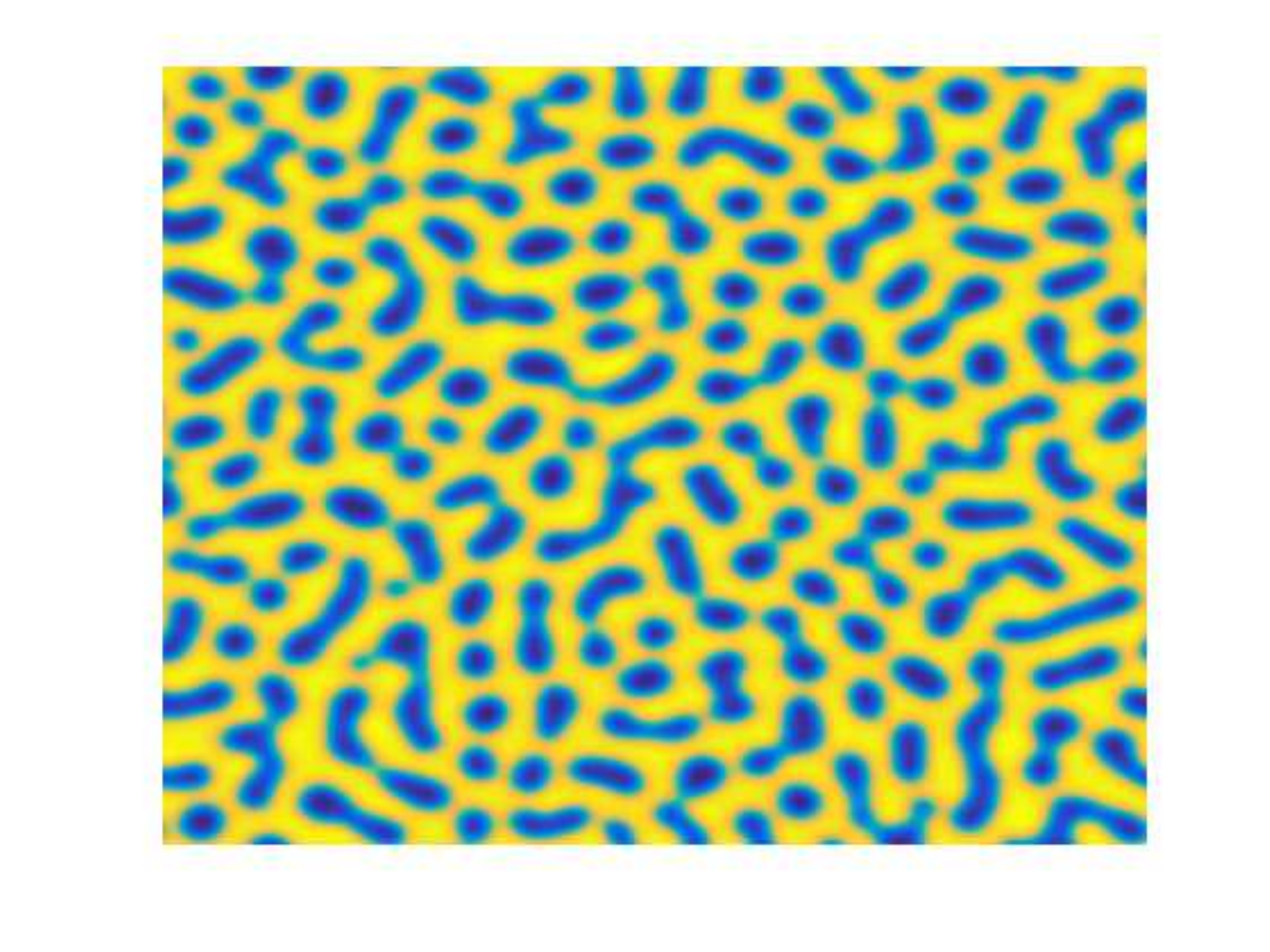}
\includegraphics[width=4cm,height=4cm]{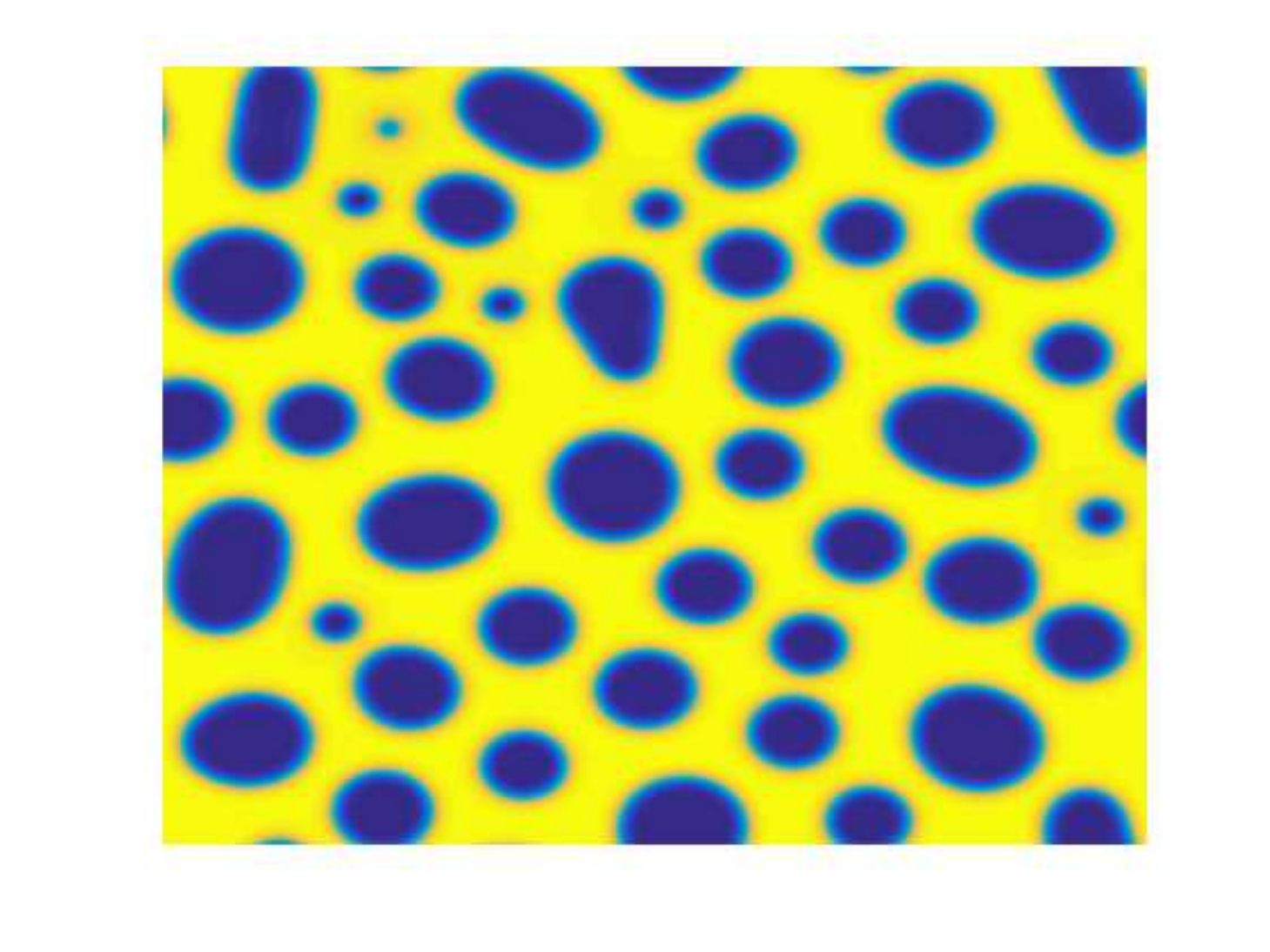}
\includegraphics[width=4cm,height=4cm]{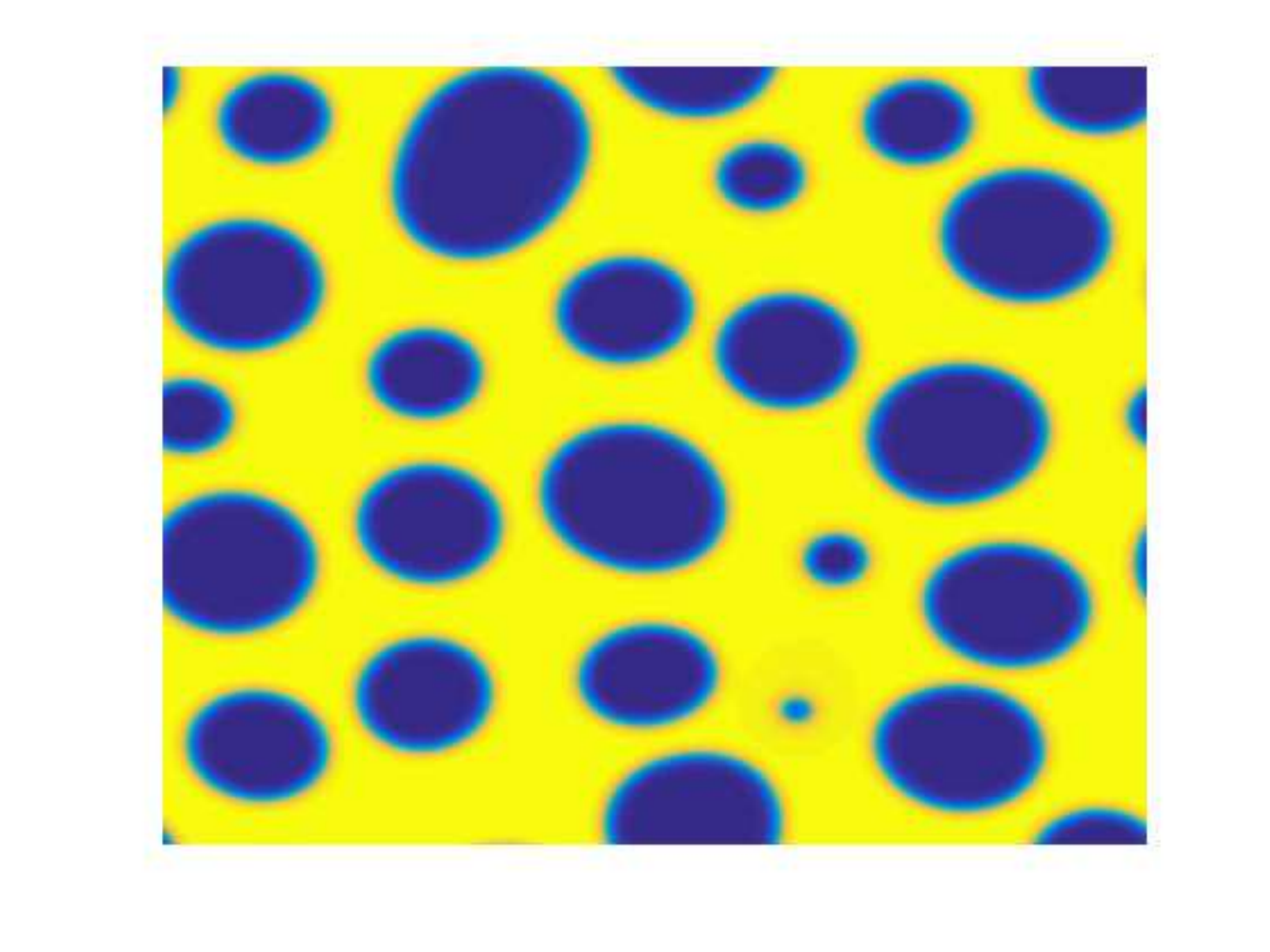}
\includegraphics[width=4cm,height=4cm]{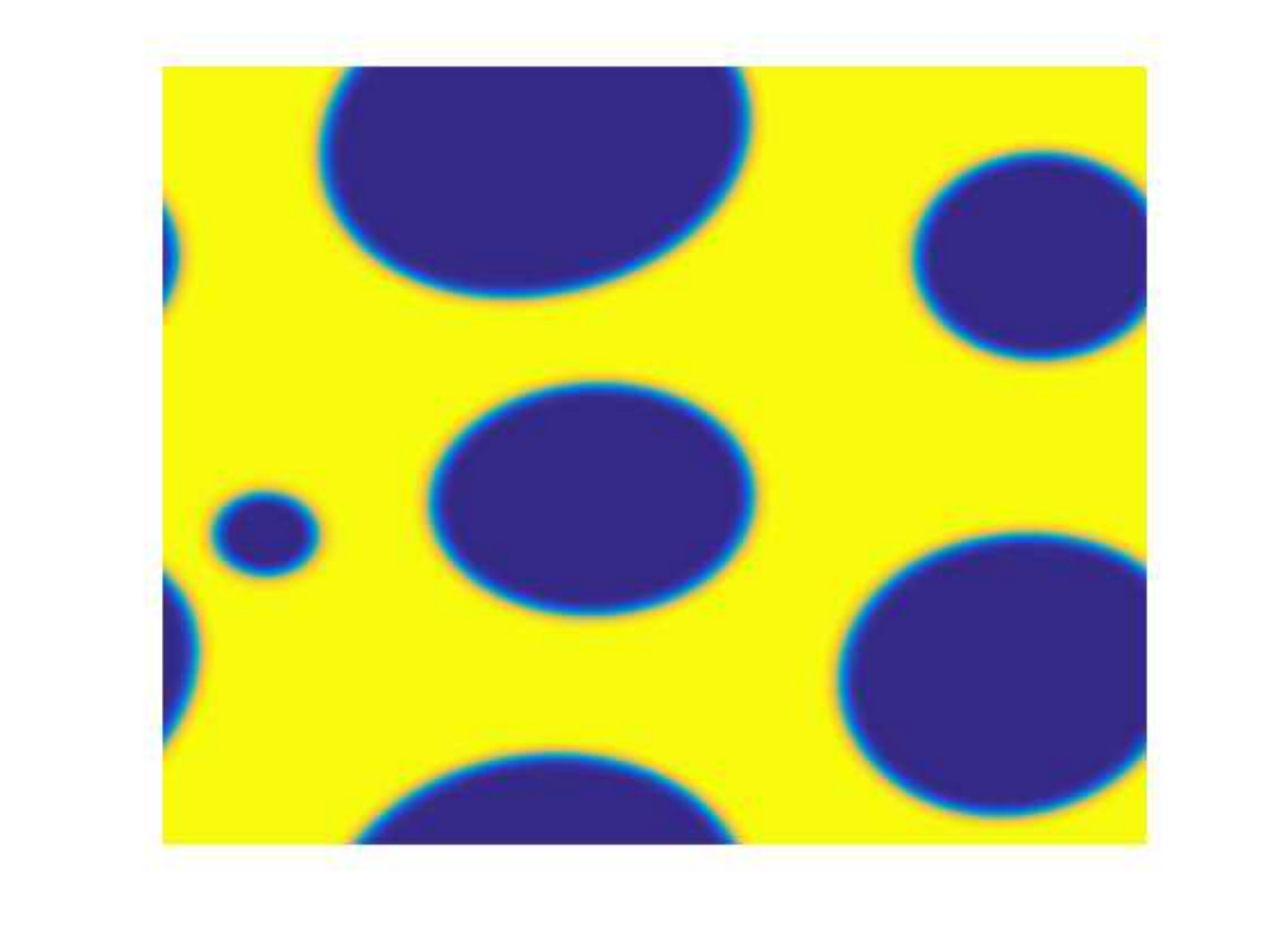}
\caption{Snapshots of the phase variable $\phi$ are taken at t=0.02, 0.5, 3, 20 for example 3.}\label{fig:fig3}
\end{figure}
\subsection{Phase field crystal equations}
A weakness of the traditional phase field methodology is that it is usually formulated in terms of fields that are spatially uniform in equilibrium. Elder \cite{elder2002modeling} firstly proposed the phase field crystal (PFC) model based on density functional theory in 2002. This model can simulate the evolution of crystalline microstructure on atomistic length and diffusive time scales. It naturally incorporates elastic and plastic deformations and multiple crystal orientations, and can be applied to many different physical phenomena.

In particular, consider the following Swift-Hohenberg free energy:
\begin{equation*}
E(\phi)=\int_{\Omega}\left(\frac{1}{4}\phi^4+\frac{1}{2}\phi\left(-\epsilon+(1+\Delta)^2\right)\phi\right)d\textbf{x},
\end{equation*}
where $\textbf{x} \in \Omega \subseteq \mathbb{R}^d$, $\phi$ is the density field and $\epsilon$ is a positive bifurcation constant with physical significance. $\Delta$ is the Laplacian operator.

Considering a gradient flow in $H^{-1}$, one can obtain the phase field crystal equation under the constraint of mass conservation as follows:
\begin{equation*}
\frac{\partial \phi}{\partial t}=\Delta\mu=\Delta\left(\phi^3-\epsilon\phi+(1+\Delta)^2\phi\right), \quad(\textbf{x},t)\in\Omega\times Q,
\end{equation*}
which is a sixth-order nonlinear parabolic equation and can be applied to simulate various phenomena such as crystal growth, material hardness and phase transition. Here $Q=(0,T]$, $\mu=\frac{\delta E}{\delta \phi}$ is called the chemical potential.

Next, we plan to simulate the phase transition behavior of the phase field crystal model. The similar numerical example can be found in many articles such as \cite{li2017efficient,yang2017linearly}. In the following example 4, we study the crystal growth in a supercooled liquid in two dimension. This example serves to show the applicability of our phase field crystal model to a physical problem \cite{li2017efficient}.

\textbf{Example 4}: In the following, we take $\epsilon=0.25$, to start our simulation on a domain $[0,800]\times[0,800]$ with a $512\times512$ mesh grid by Fourier spectral method in space and first-order E-SAV scheme in time. We generated the three crystallites using random perturbations on four small square pathes. The following expression will be used to define the crystallites such as in \cite{zhang2019efficient}:
\begin{equation*}
\phi(x_l,y_l)=\overline{\phi}+C\left(\cos(\frac{q}{\sqrt{3}}y_l)\cos(qx_l)-\frac12\cos(\frac{2q}{\sqrt{3}}y_l)\right),
\end{equation*}
where $x_l$, $y_l$ define a local system of cartesian coordinates that is oriented with the crystallite lattice. The parameters $\overline{\phi}=0.285$, $C=0.446$ and $q=0.66$. The local cartesian system is defined as
\begin{equation*}
\aligned
&x_l(x,y)=xsin\theta+ycos\theta,\\
&y_l(x,y)=-xcos\theta+ysin\theta.
\endaligned
\end{equation*}
The centers of three pathes are located at $(150,150)$, $(200,250)$ and $(250,150)$ with $\theta=\pi/4$, $0$, and $-\pi/4$. The length of each square is 40. Figure \ref{fig:fig4} shows the snapshots of the density field $\phi$ at different times. We observe the growth of the crystalline phase. We plot the energy dissipative curve in \ref{fig:fig5} using three different time steps of $\delta t=0.01$, $0,1$ and $1$. One can observe that the energy decreases at all times no matter big or small time steps. This expression of unconditionally energy stable proved the efficiency of our proposed algorithm, as predicted by the former theory.
\begin{figure}[htp]
\centering
\subfigure[t=0]{
\includegraphics[width=4cm,height=4cm]{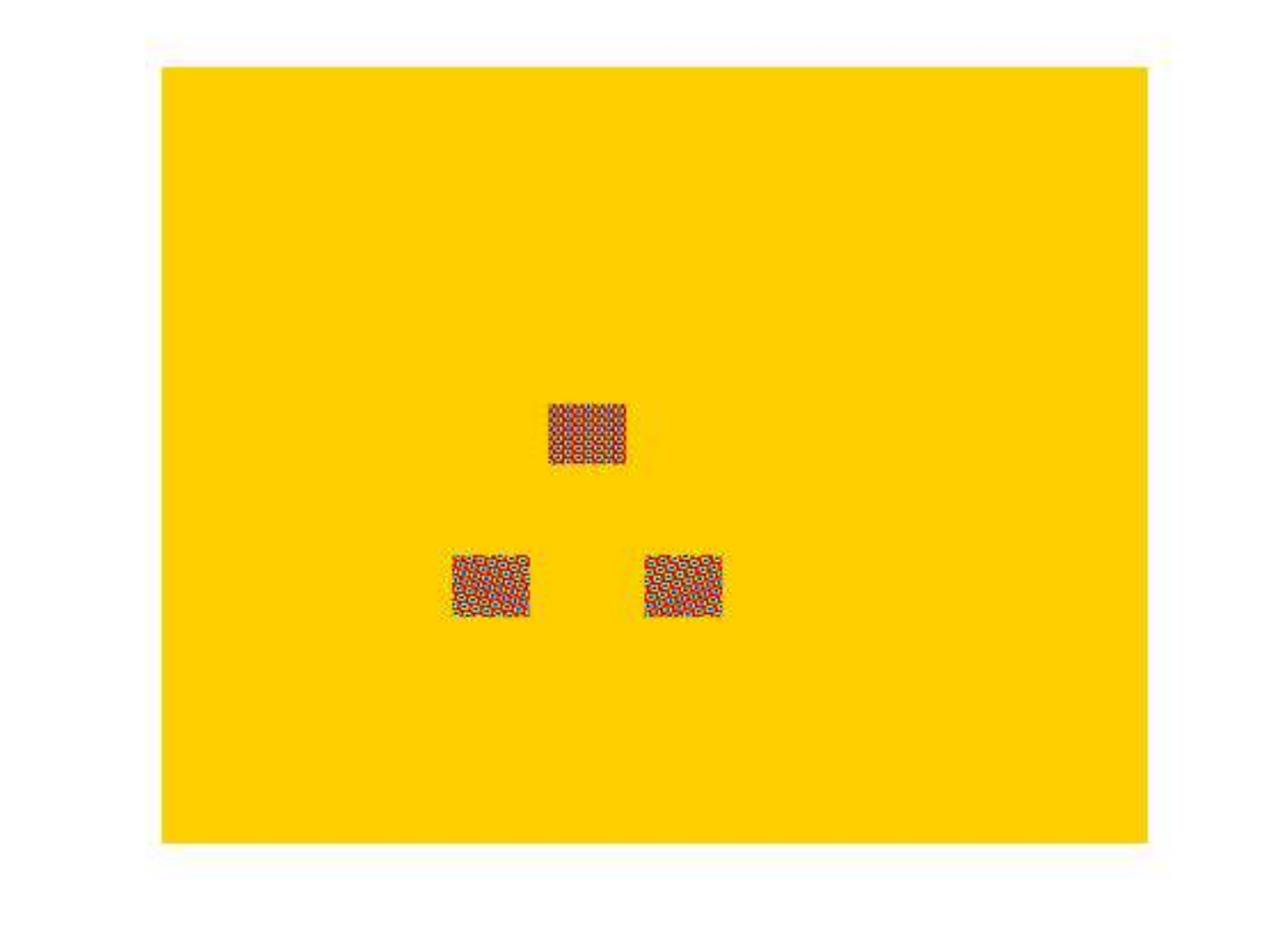}
}
\subfigure[t=150]
{
\includegraphics[width=4cm,height=4cm]{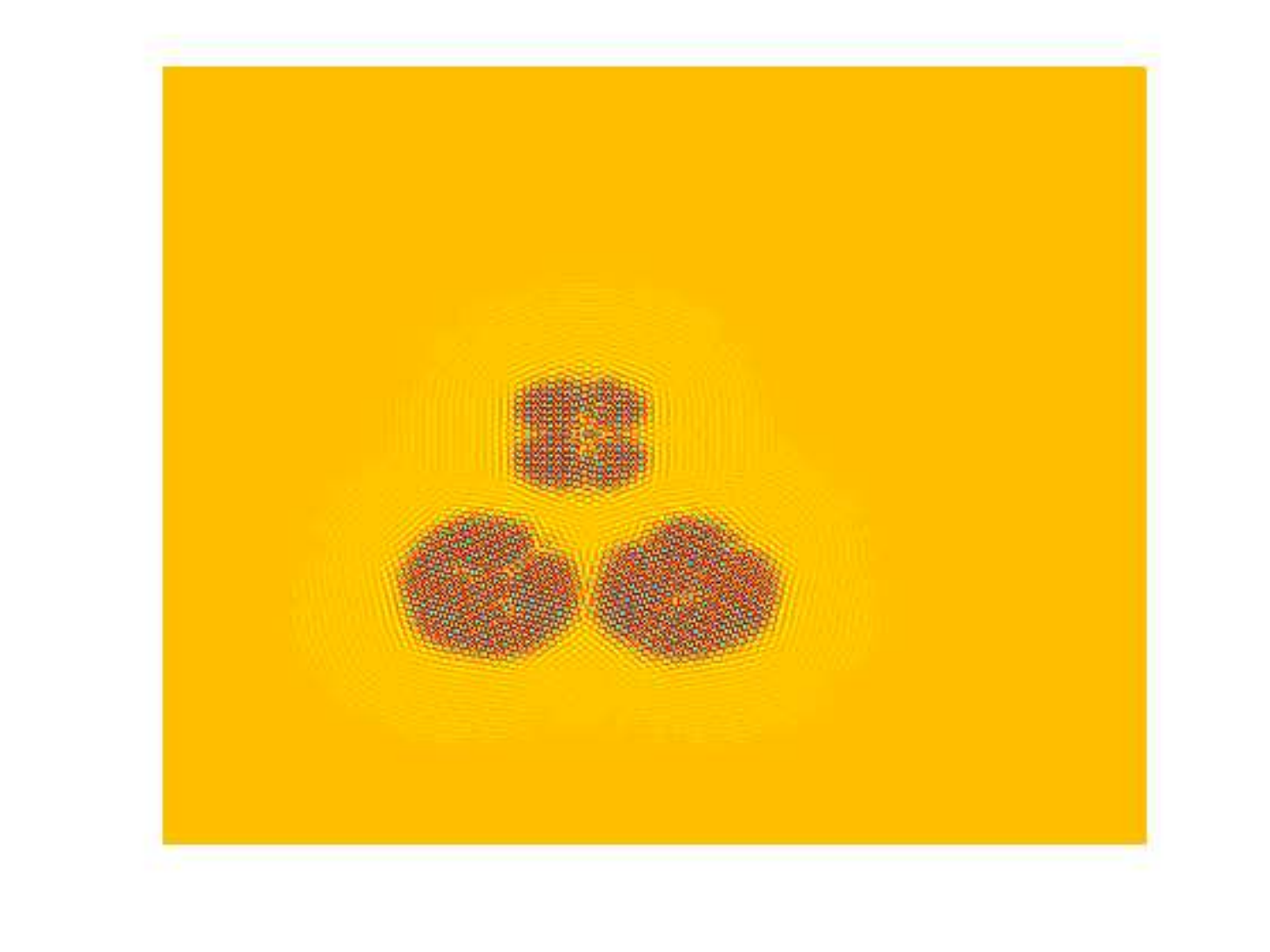}
}
\subfigure[t=400]
{
\includegraphics[width=4cm,height=4cm]{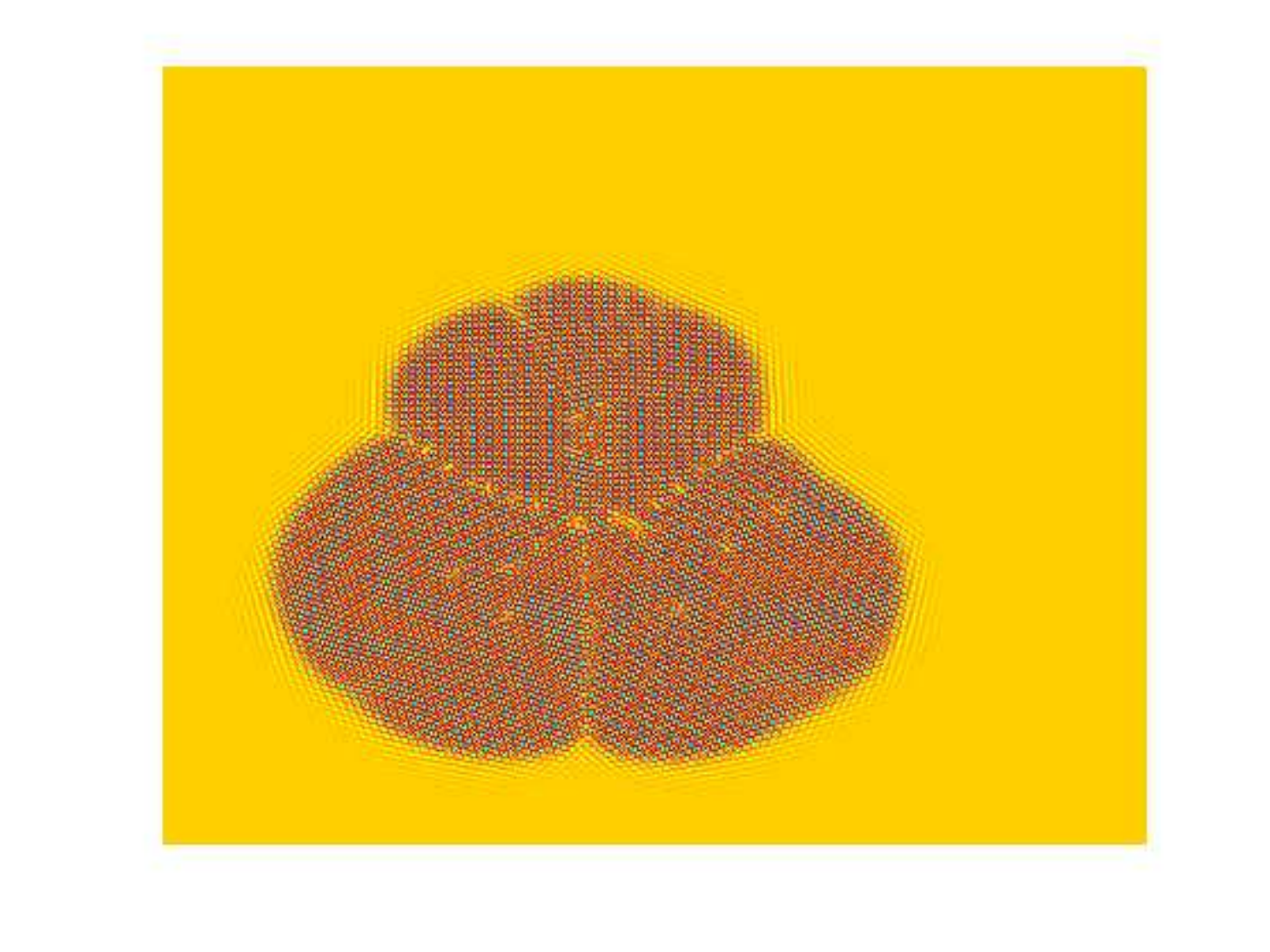}
}
\quad
\subfigure[t=600]{
\includegraphics[width=4cm,height=4cm]{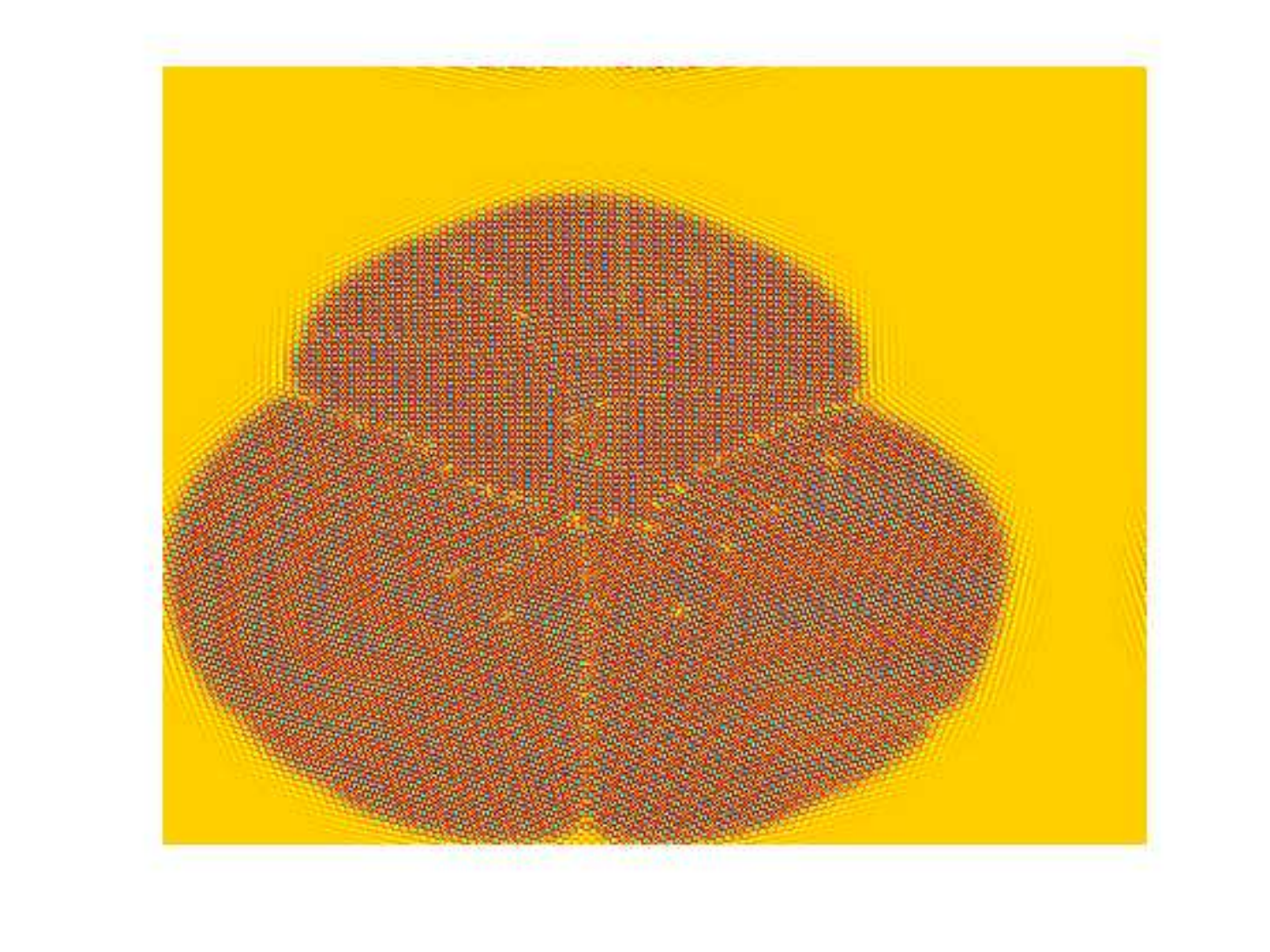}
}
\subfigure[t=900]
{
\includegraphics[width=4cm,height=4cm]{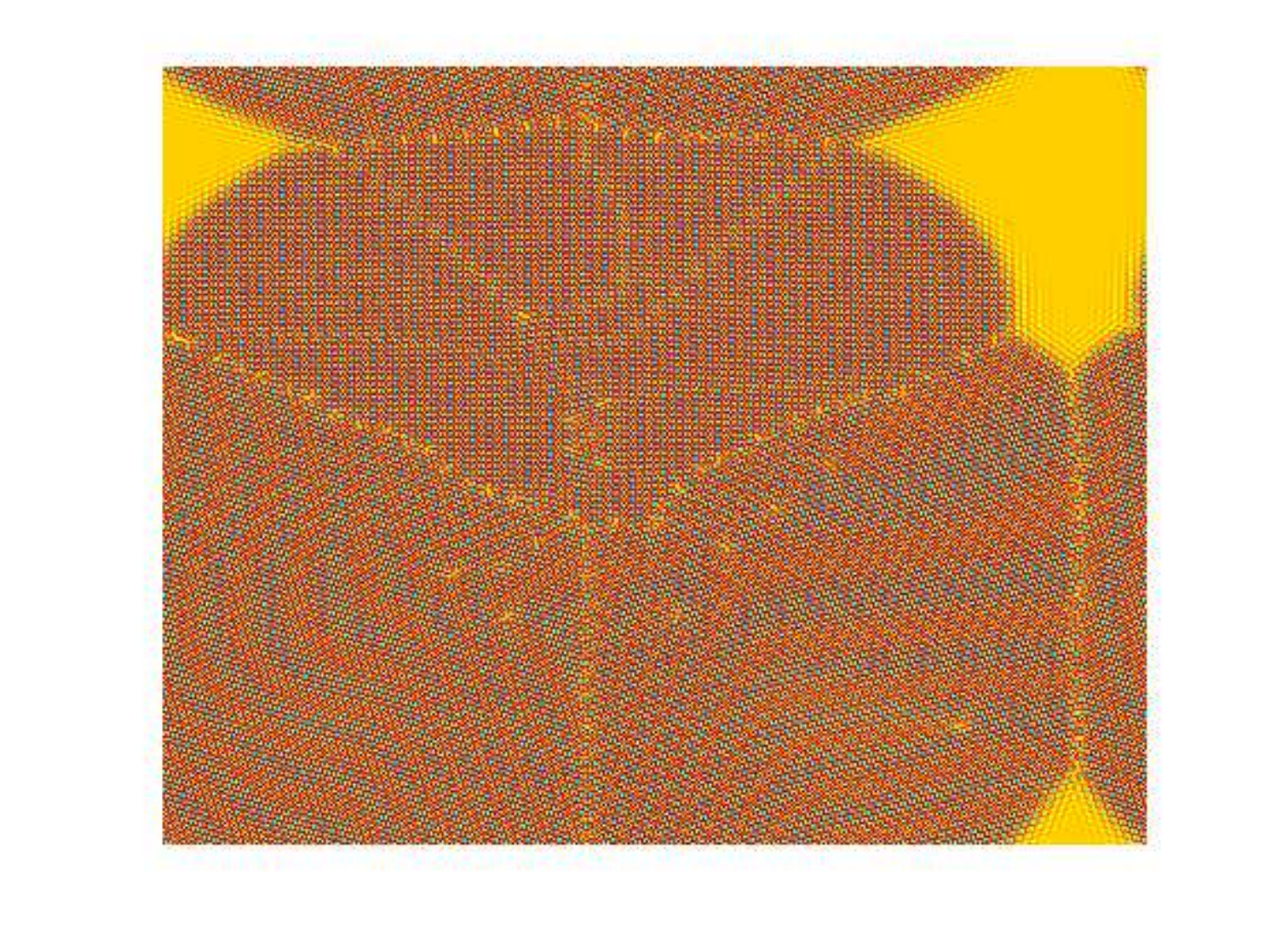}
}
\subfigure[t=1200]
{
\includegraphics[width=4cm,height=4cm]{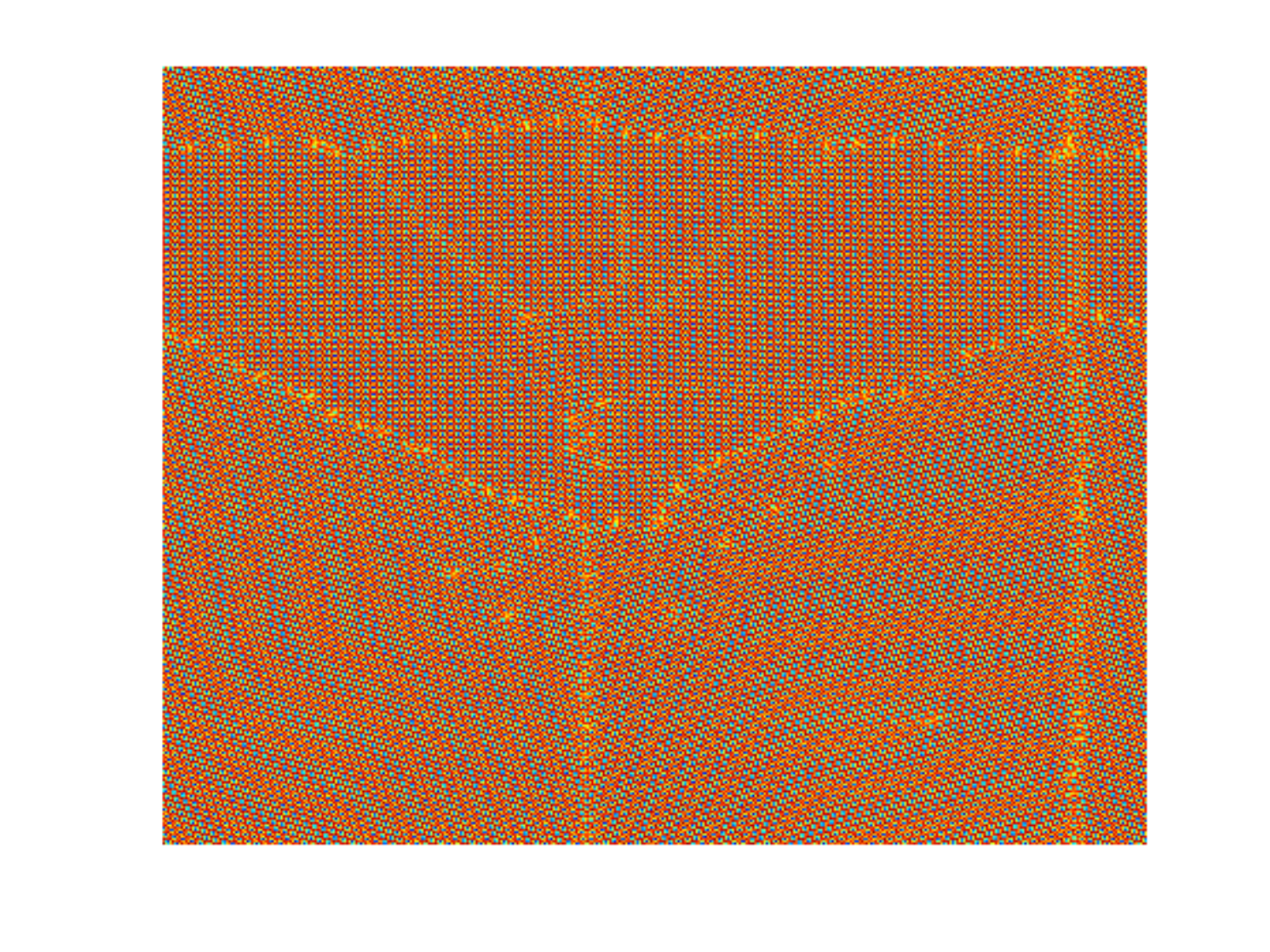}
}
\caption{Snapshots of the phase variable $\phi$ are taken at t=0, 150, 400, 600, 900, 1200 for example 4.}\label{fig:fig4}
\end{figure}
\begin{figure}[htp]
\centering
\includegraphics[width=10cm,height=7cm]{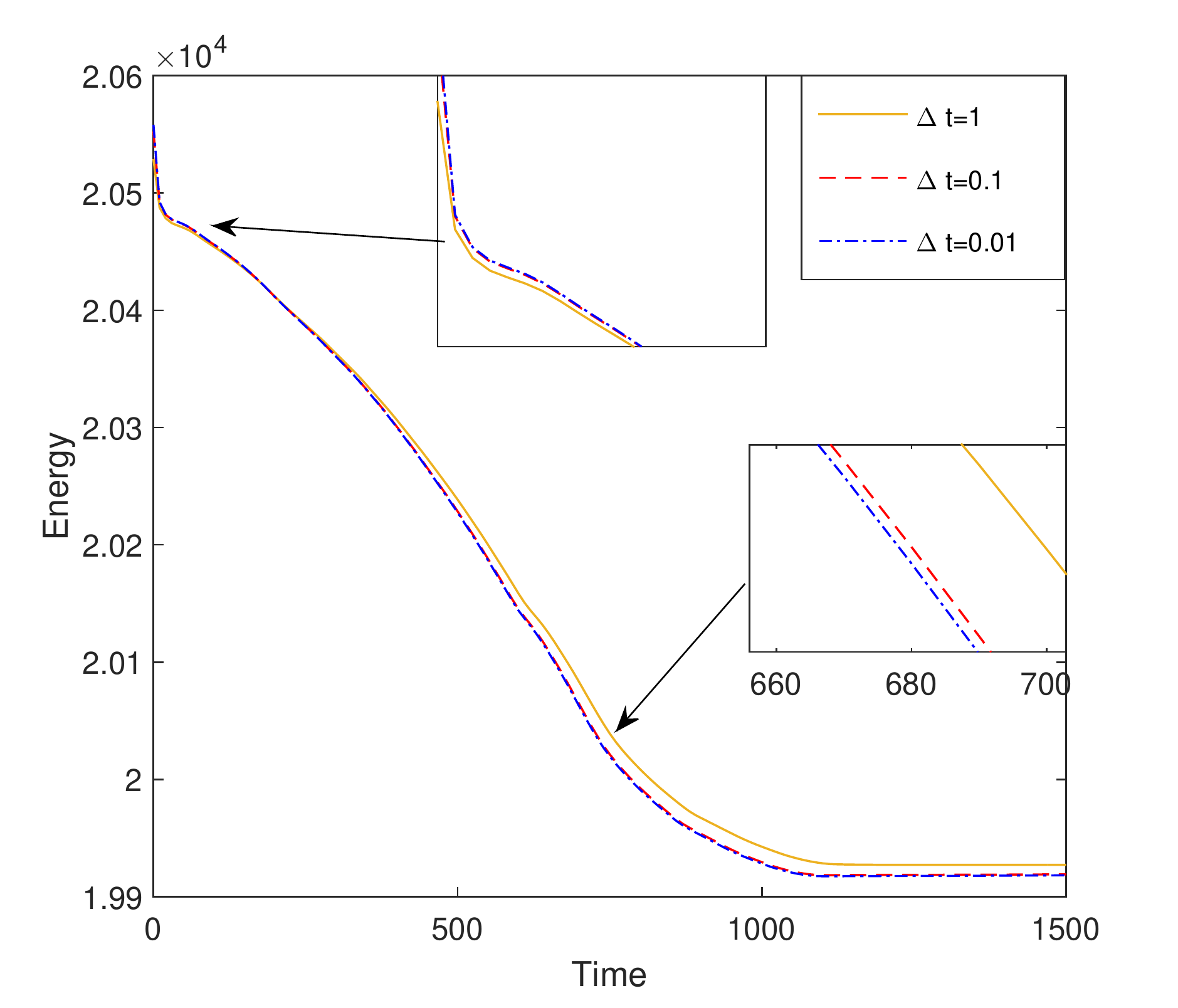}
\caption{Energy evolution for the phase field crystal equation for example 4 using E-SAV approaches with different time steps of $\delta t=0.01$, $0.1$ and $1$.}\label{fig:fig5}
\end{figure}

In the following, we will consider example 5 to check the difference of phase transition behavior between the proposed E-SAV method and traditional SAV approach.

\textbf{Example 5}: The initial condition is
\begin{equation*}
\aligned
&\phi_0(x,y)=0.07+0.07Rand(x,y),
\endaligned
\end{equation*}
where the $Rand(x,y)$ is the random number in $[-1,1]$ with zero mean. The order parameter is $\epsilon=0.025$, the computational domain $\Omega=[0,128]^2$. we set $256^2$ Fourier modes to discretize the two dimensional space.

We show the phase transition behavior of the density field at various times in Figure \ref{fig:fig6}. Similar computation results for phase field crystal model can be found in \cite{ShenA,yang2017linearly}. We investigate the process of crystallization in a supercool liquid by using both SAV and the proposed E-SAV schemes. No visible difference is observed.
\begin{figure}[htp]
\centering
\begin{tabular}{ccccc}
SAV&\includegraphics[width=3.6cm,height=3.6cm]{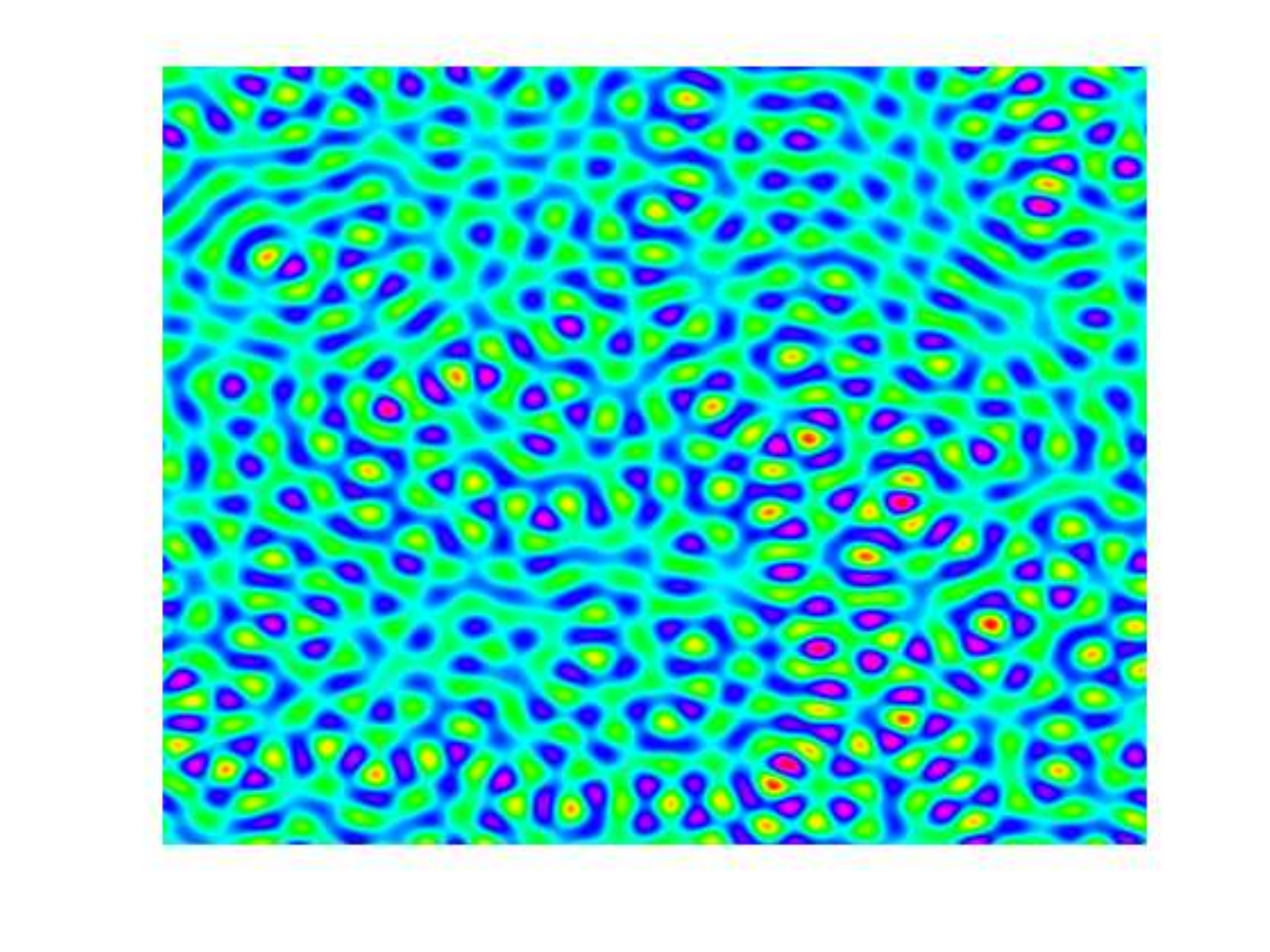}\includegraphics[width=3.6cm,height=3.6cm]{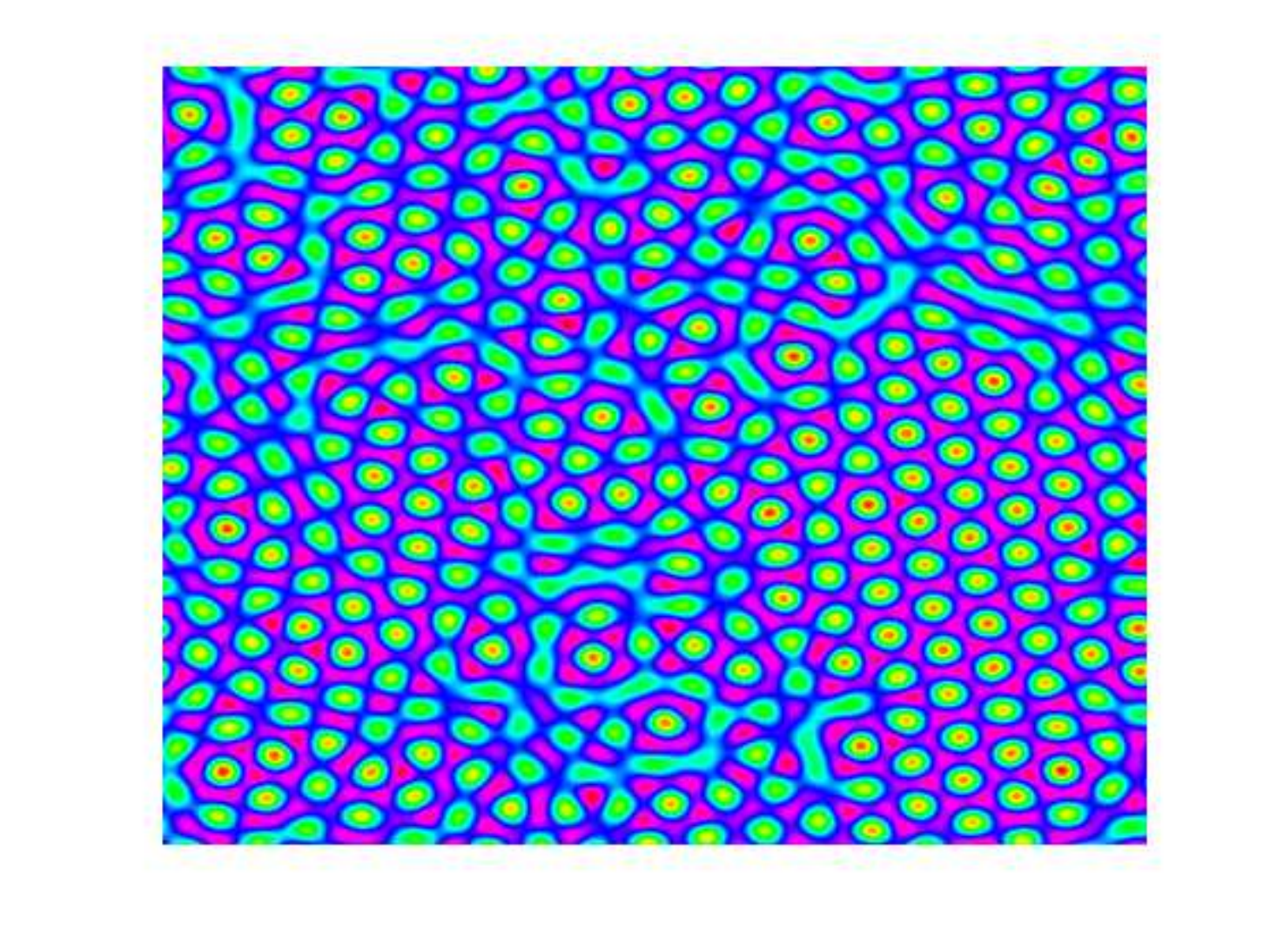}\includegraphics[width=3.6cm,height=3.6cm]{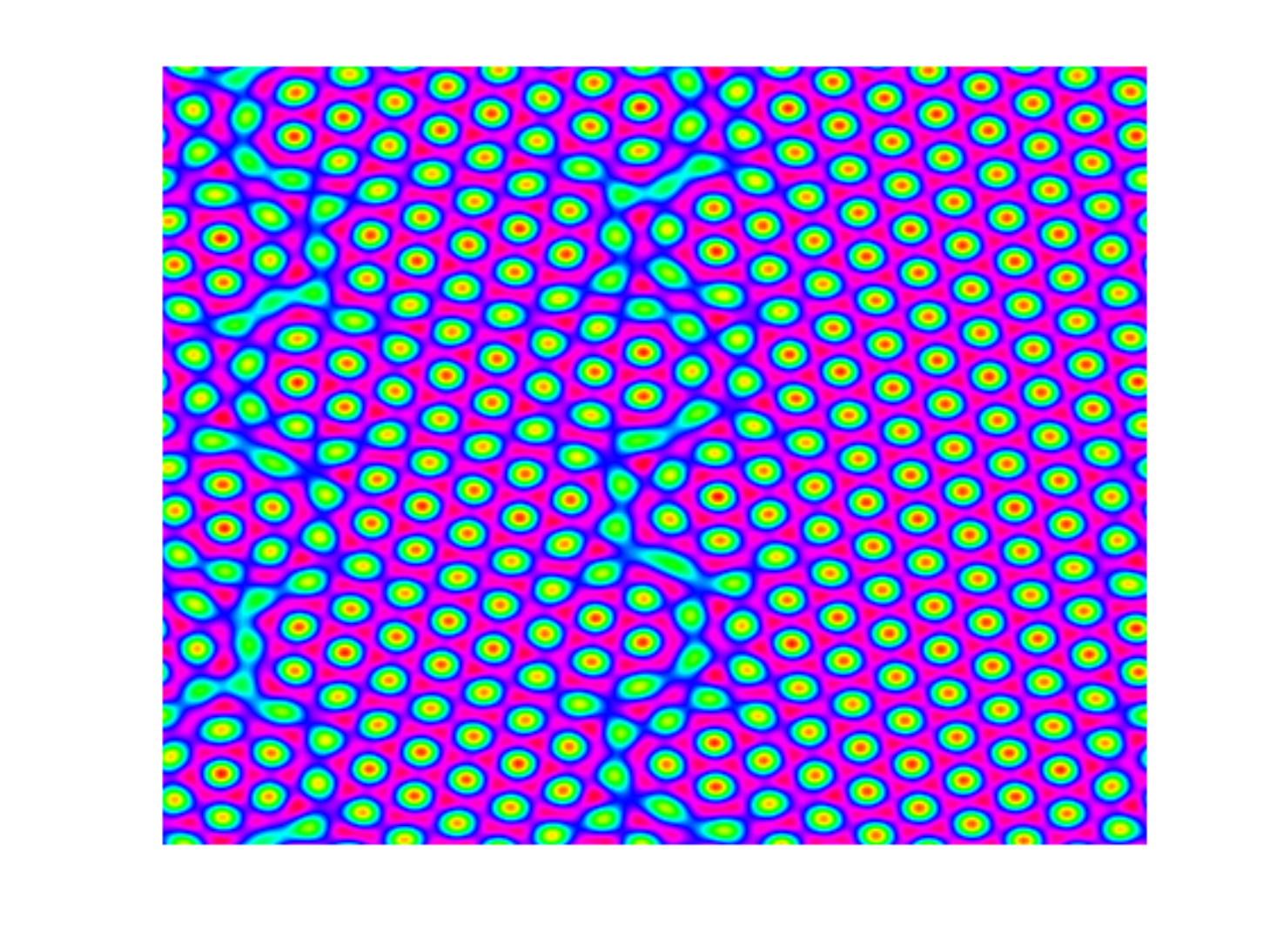}\includegraphics[width=3.6cm,height=3.6cm]{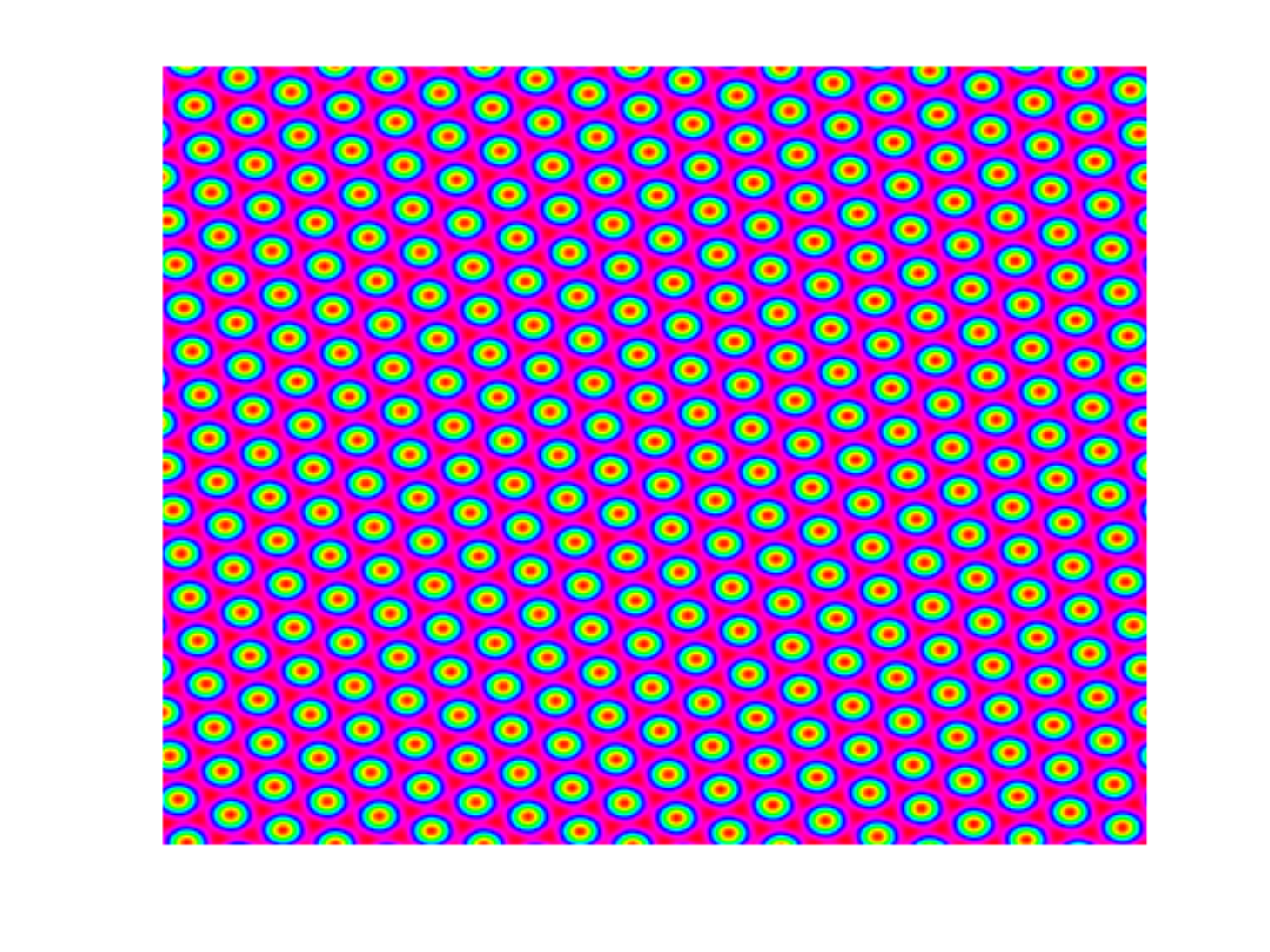}\\
\Xhline{1.2pt}
E-SAV&\includegraphics[width=3.6cm,height=3.6cm]{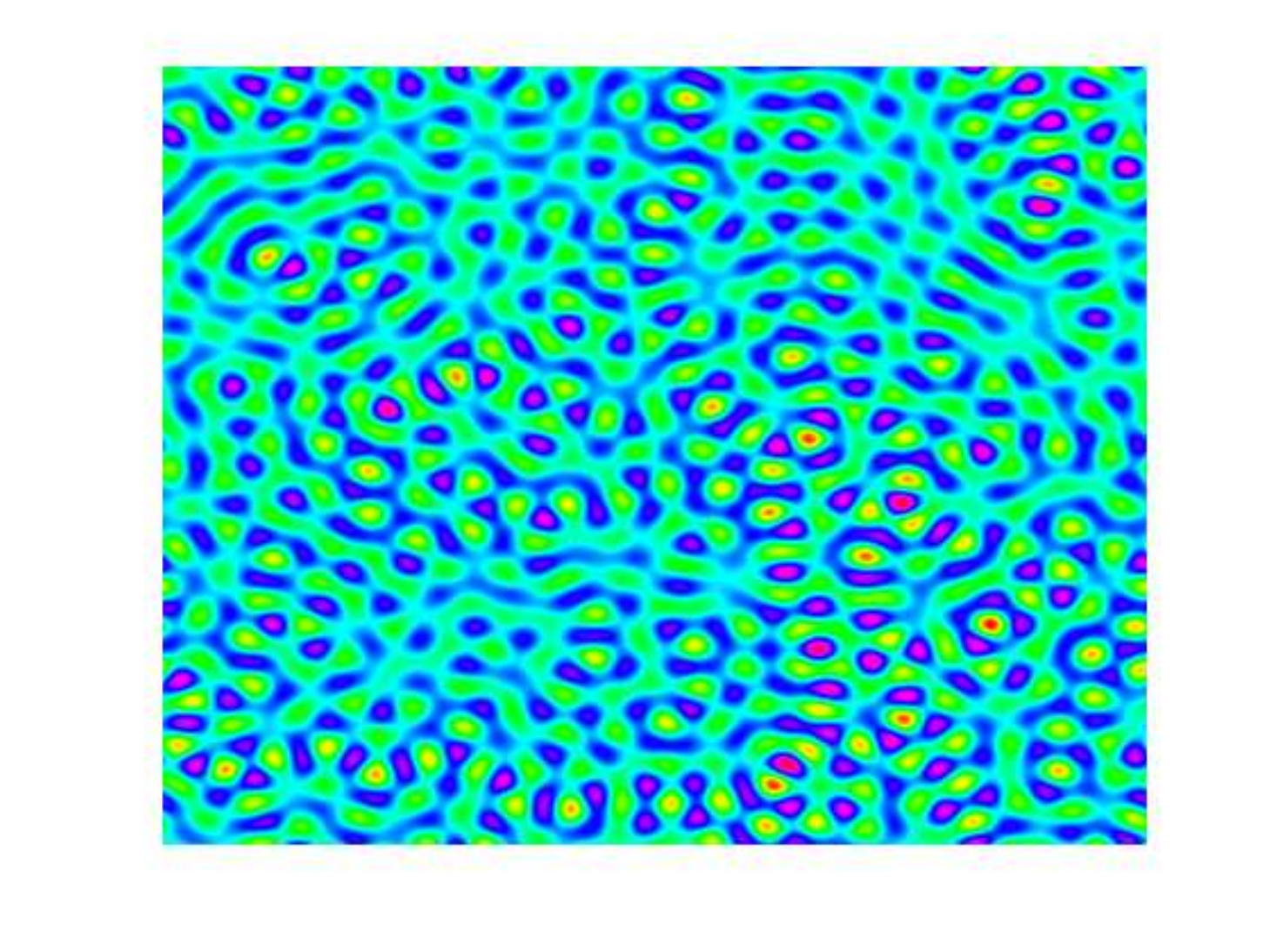}\includegraphics[width=3.6cm,height=3.6cm]{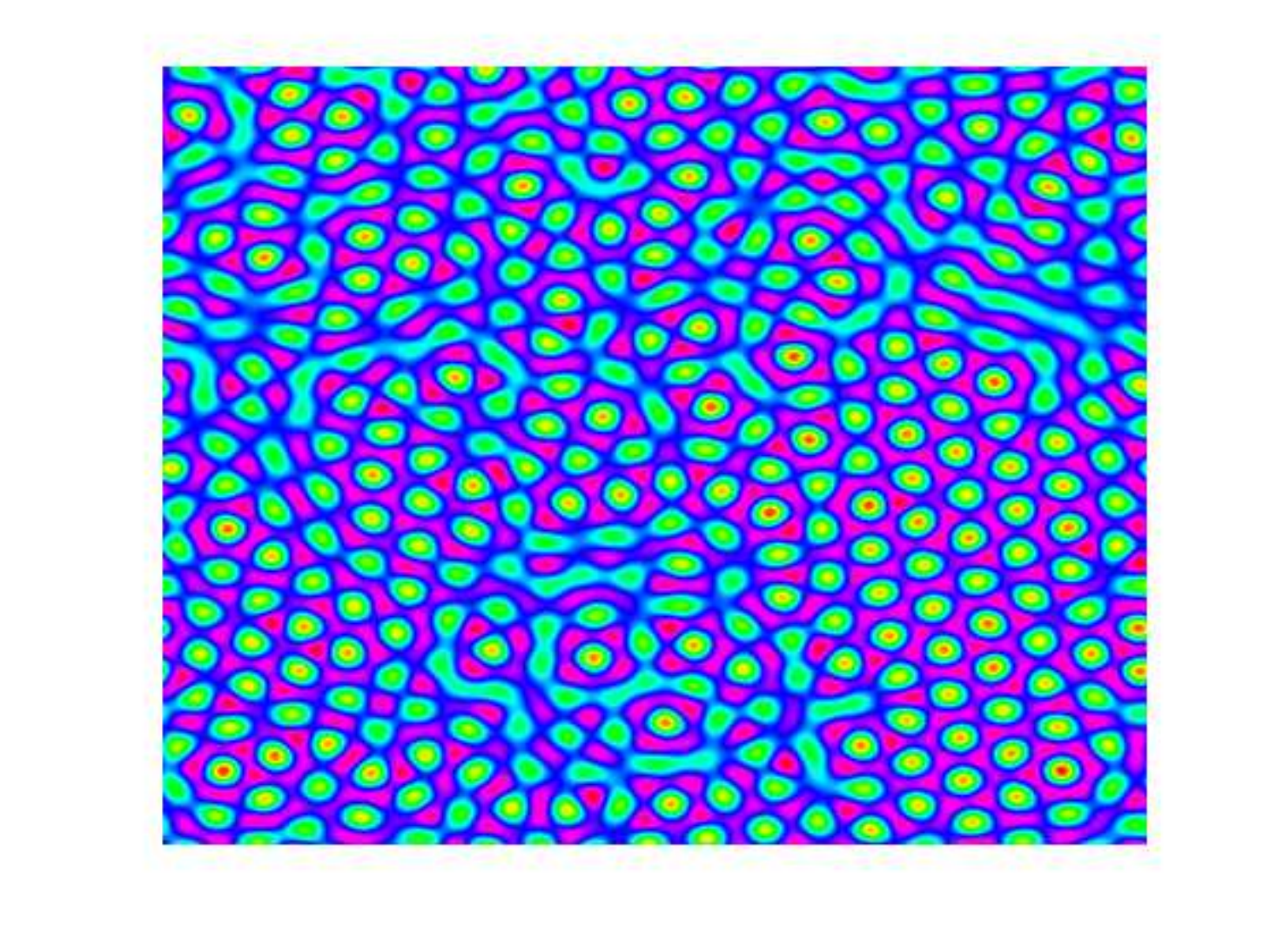}\includegraphics[width=3.6cm,height=3.6cm]{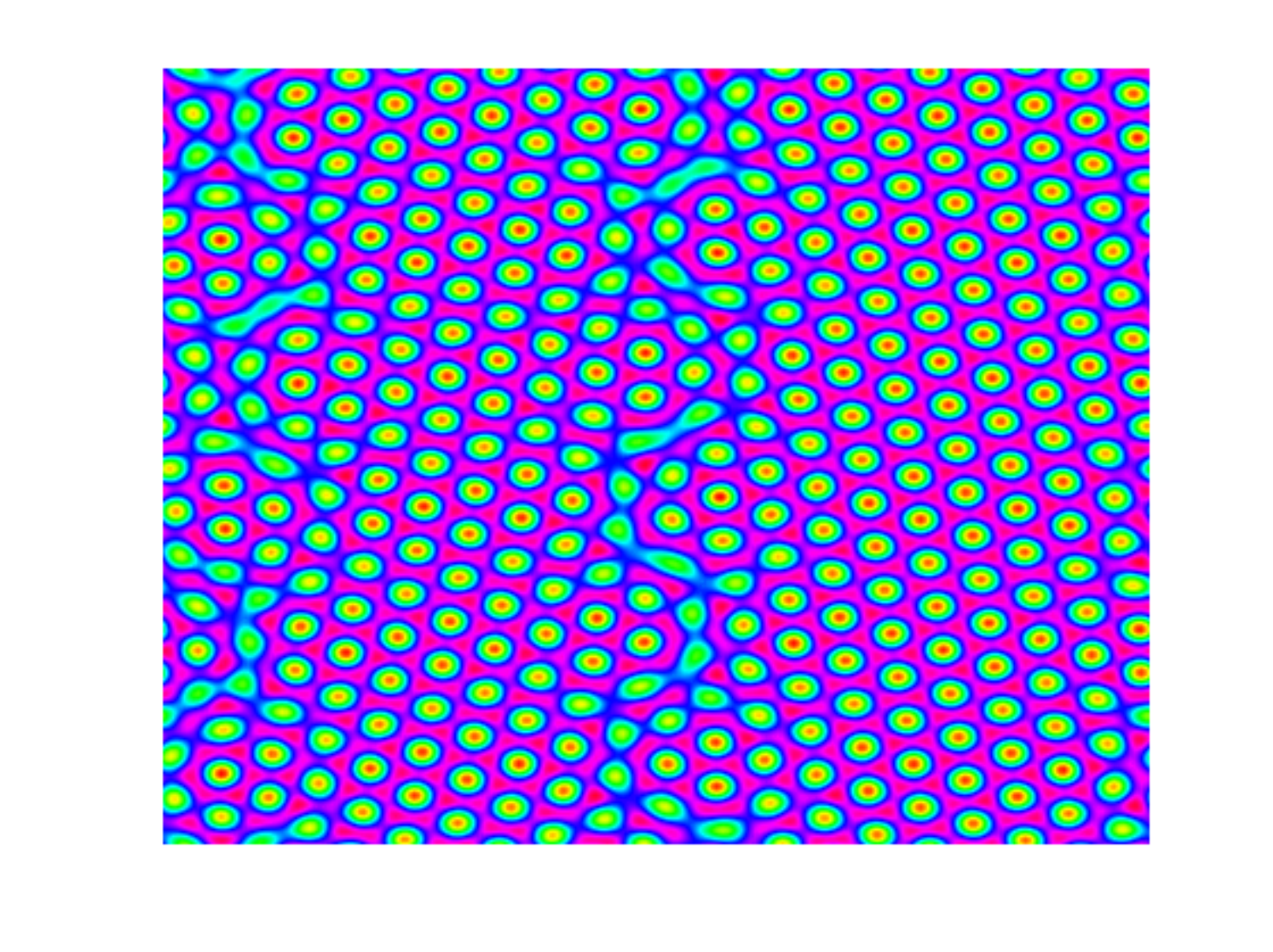}\includegraphics[width=3.6cm,height=3.6cm]{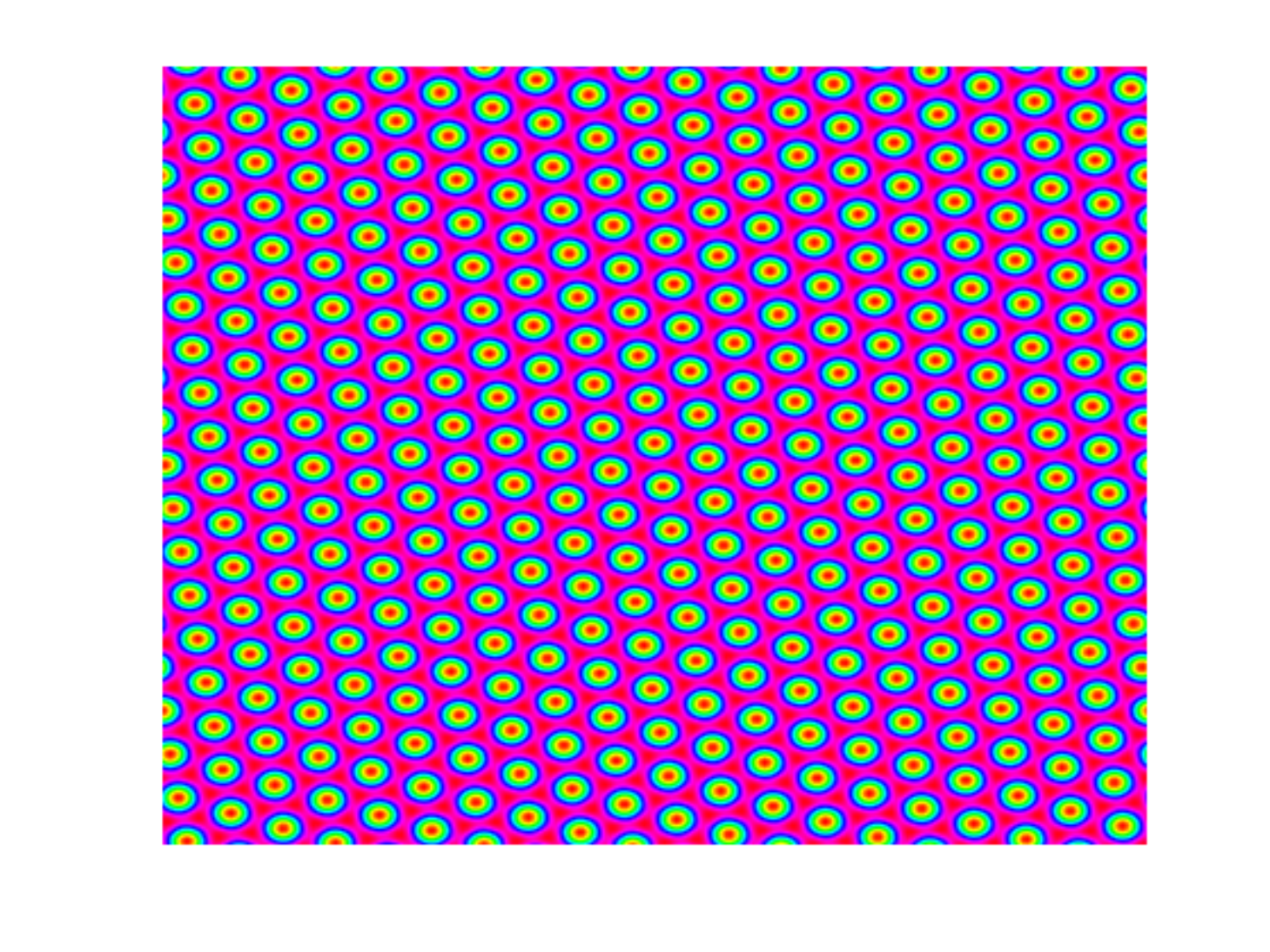}\\
\Xhline{1.2pt}
\end{tabular}
\caption{Configuration evolutions for PFC model by SAV and E-SAV schemes are taken at $t=200$, $500$, $1200$, and $6000$.}\label{fig:fig6}
\end{figure}

\subsection{The Cahn-Hilliard phase field model of the binary fluid-surfactant system}
In this subsection, we use several numerical examples to demonstrate the accuracy, energy stability and efficiency of the proposed schemes based on ME-SAV approach for the Cahn-Hilliard phase field model of the binary fluid-surfactant system. In the following two examples, we set the domain $\Omega=[0,2\pi]\times[0,2\pi]$.  Other than that, the default values of parameters are given as follows,
$$M_\phi=M_\rho=2.5e-4,\quad\alpha=2.5e-4,\quad\beta=1,\quad\theta=0.3,\quad\epsilon=0.05,\quad\eta=0.08,\quad\rho_s=1.$$

\textbf{Example 6}: we first test the error and the convergent rates of the proposed first-order ME-SAV scheme. The initial conditions are as follows
\begin{equation*}
\aligned
&\phi_0(x,y)=0.3{\cos(3x)}+0.5{\cos(y)},\\
&\rho_0(x,y)=0.2{\cos(2x)}+0.25{\sin(y)}.
\endaligned
\end{equation*}

We use the Fourier spectral Galerkin method for spatial discretization with $N=128$. The true solution is unknown and we therefore use the
Fourier Galerkin approximation in the case $\Delta t=1e-5$ and $T=0.1$ as a reference solution. Then, we show the $L^2$ errors of two phase variables $\phi$ and $\rho$ between the numerical solutions and the reference solutions with different time step sizes in Table \ref{tab:tab3}. We observe that the convergence rates of both variables $\phi$ and $\rho$ are all first order accurate.

\begin{table}[h!b!p!]
\small
\centering
\caption{\small The $L_2$ errors, convergence rates for first order scheme in time for ME-SAV approach.}\label{tab:tab3}
\begin{tabular}{cccccccc}
\hline
          &&$\|\phi-\phi^n\|$&&&$\|\rho-\rho^n\|$\\
\cline{1-7}
&$\Delta t$          &$L_2$ error&Rate&&$L_2$ error&Rate\\
\cline{1-7}
           &$1e-2$    &2.5127e-3   &---       &&1.0355e-4   &---   \\
           &$5e-3$    &1.3078e-3   &0.9421    &&5.2007e-5   &0.9935\\
1st-ME-SAV &$2.5e-3$  &6.6740e-4   &0.9705    &&2.6019e-5   &0.9991\\
           &$1.25e-3$ &3.3628e-4   &0.9889    &&1.2972e-5   &1.0041\\
           &$6.25e-4$ &1.6779e-4   &1.0030    &&6.4364e-6   &1.0111\\
\cline{1-7}
\end{tabular}
\end{table}

In next example, we study the phase separation behaviors in the two dimensional space that are called spinodal decomposition by using the first order ME-SAV scheme.

\textbf{Example 7}: The initial conditions are taken as the randomly perturbed concentration fields:
\begin{equation*}
\aligned
&\phi_0(x,y)=0.001rand(x,y)\\
&\rho_0(x,y)=0.2+0.001rand(x,y).
\endaligned
\end{equation*}

We set $\epsilon=0.02$, $\eta=0.005$. In Figure \ref{fig:fig7}, we show the snapshots of two phase variables $\phi$ and $\rho$ which are taken at t = 1, 10, 20, 50, 100, 200, 400, 1000, 1500 and 2000. One can see that the two fluids are well mixed at the beginning. As time goes on, because of the influence of the surface tensions, the two fluids start to decompose to equilibrium. However, a relatively high value of the concentration variable $\rho$ is always driven to be located at the fluid interface. We also plot the evolution of energy curves in Figure \ref{fig:fig8} for Example 7 which indicates that the energy monotonically decays with respect to the time.
\begin{figure}[htp]
\centering
\subfigure[t=1]
{
\includegraphics[width=3.5cm,height=3.5cm]{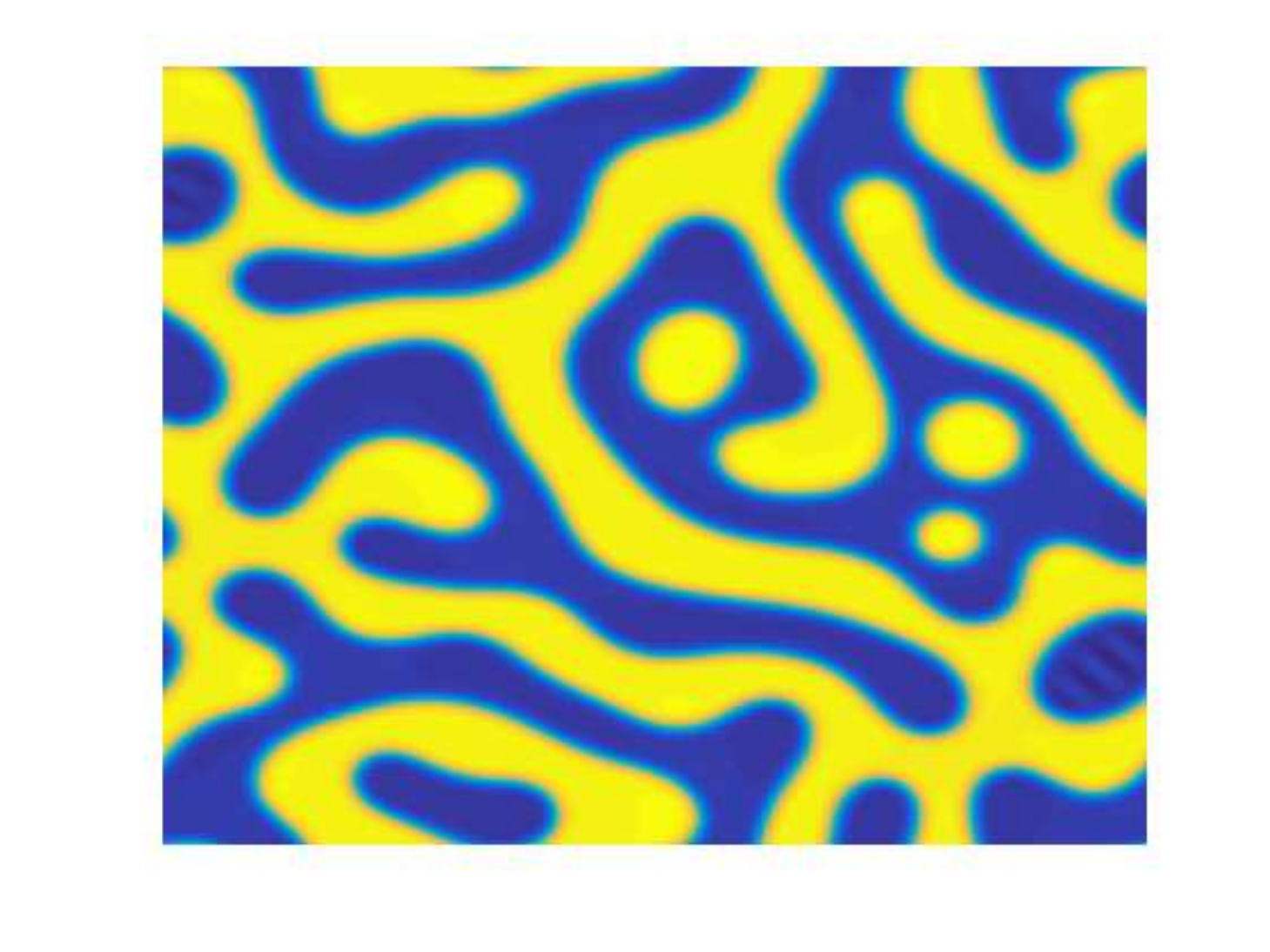}
\includegraphics[width=3.5cm,height=3.5cm]{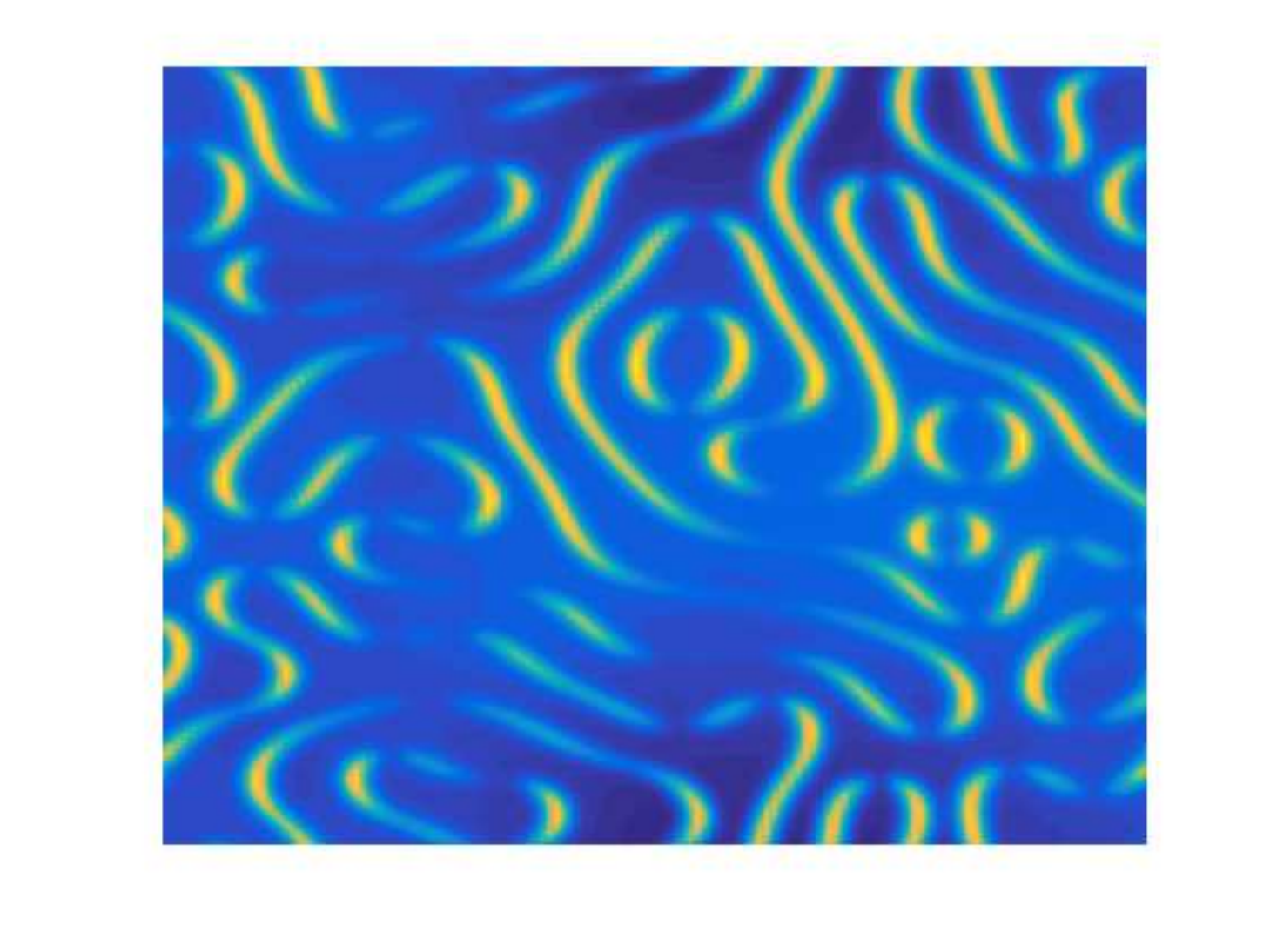}
}
\subfigure[t=100]{
\includegraphics[width=3.5cm,height=3.5cm]{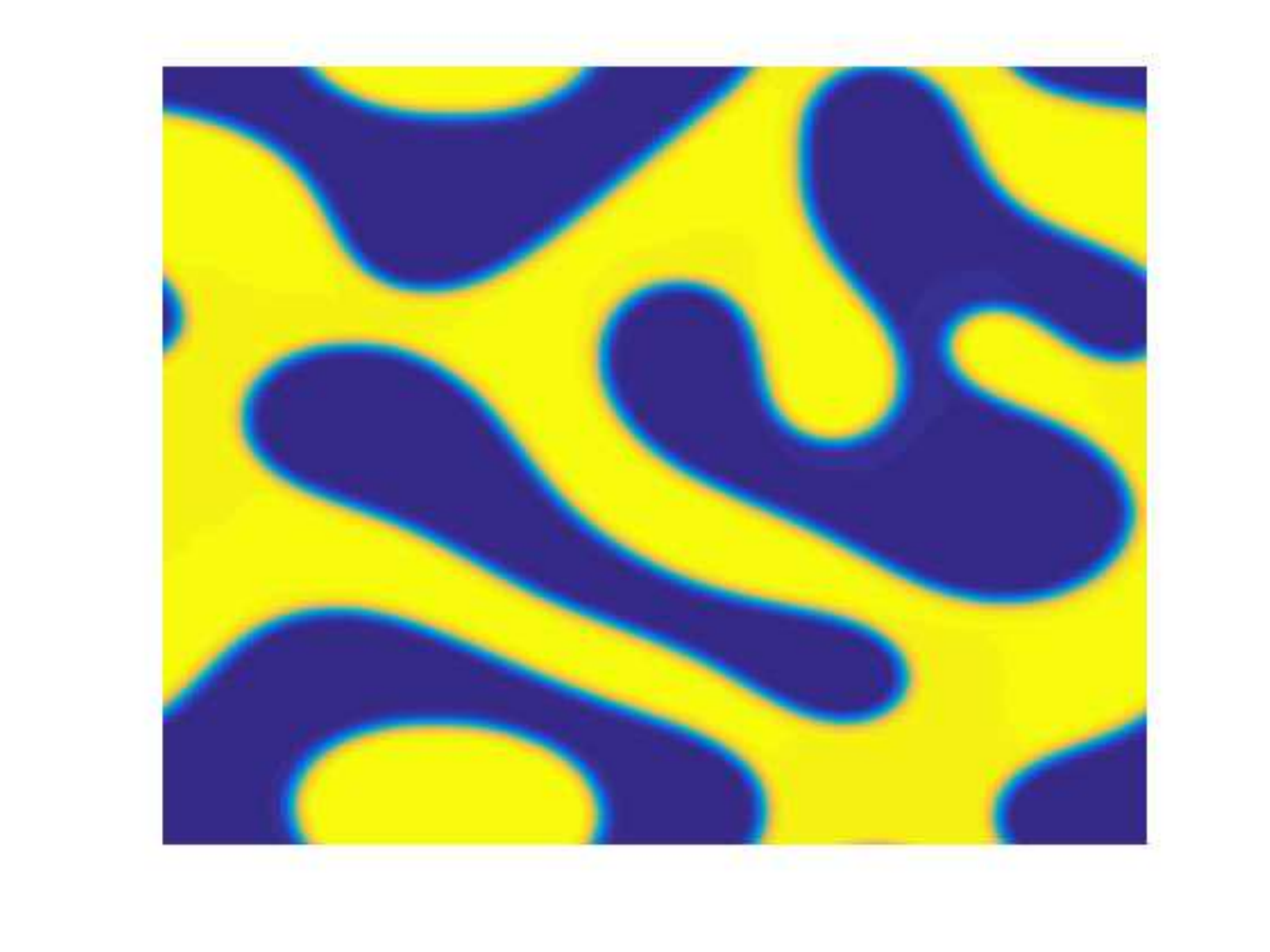}
\includegraphics[width=3.5cm,height=3.5cm]{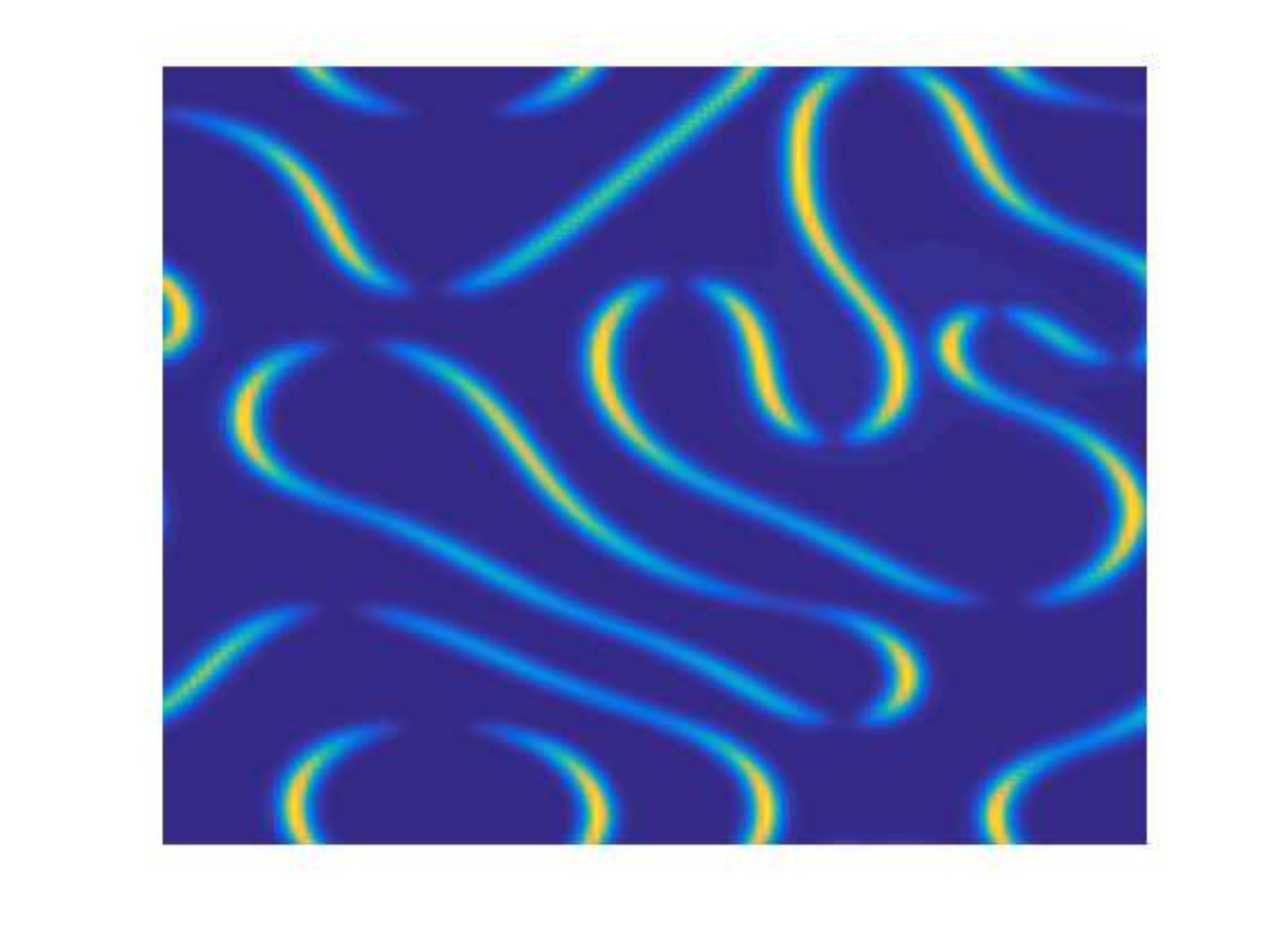}
}
\subfigure[t=800]
{
\includegraphics[width=3.5cm,height=3.5cm]{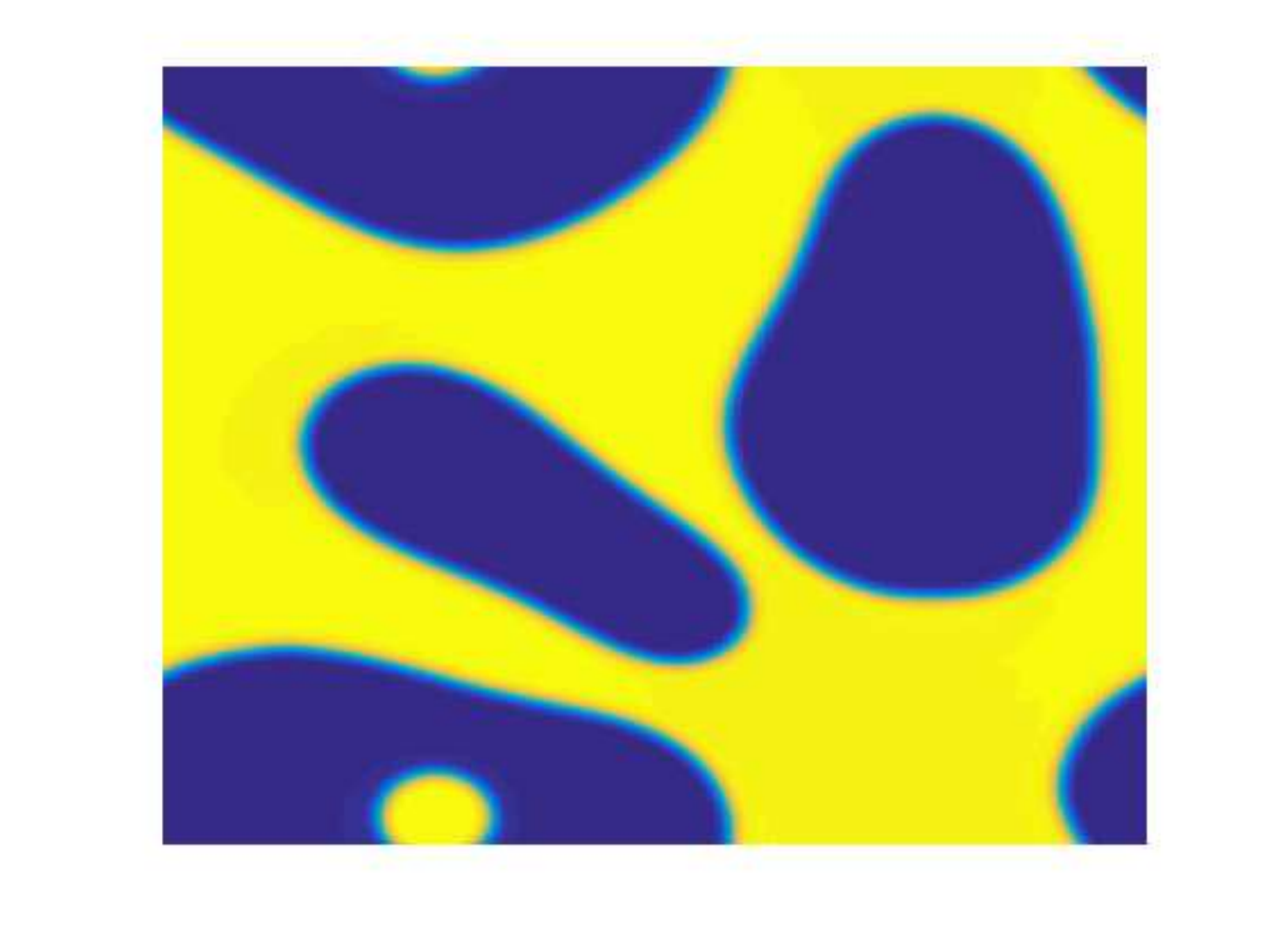}
\includegraphics[width=3.5cm,height=3.5cm]{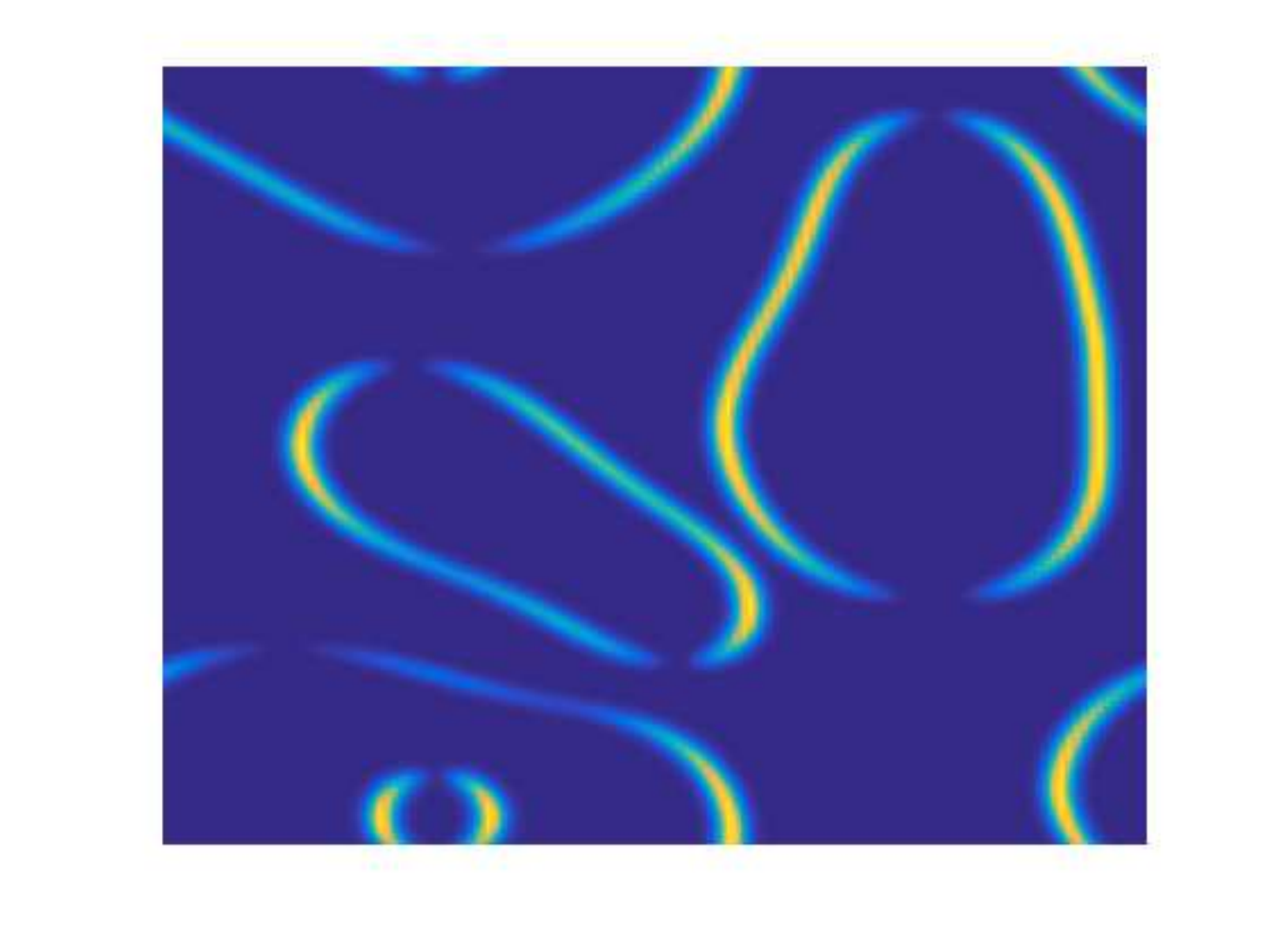}
}
\subfigure[t=1500]
{
\includegraphics[width=3.5cm,height=3.5cm]{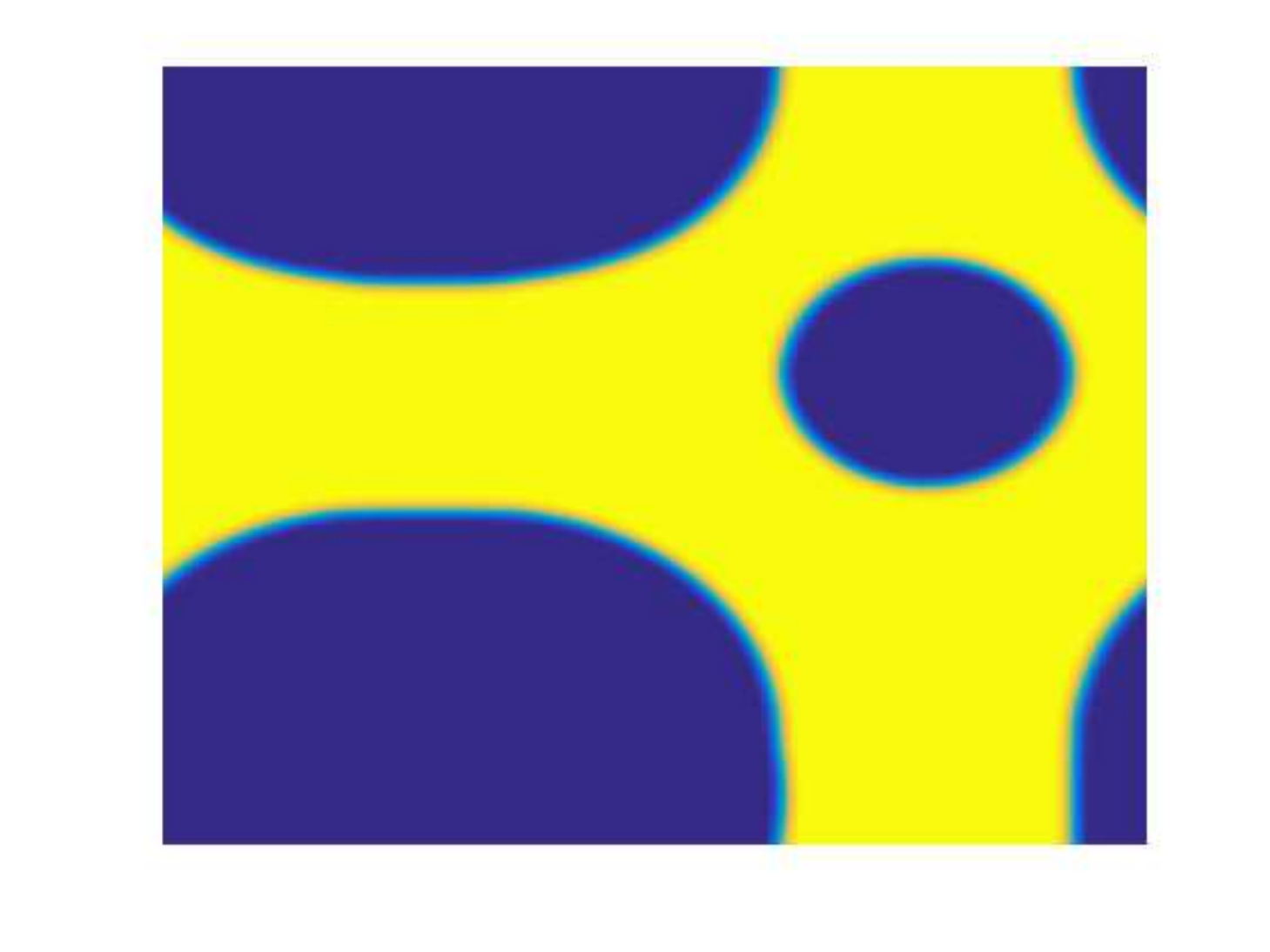}
\includegraphics[width=3.5cm,height=3.5cm]{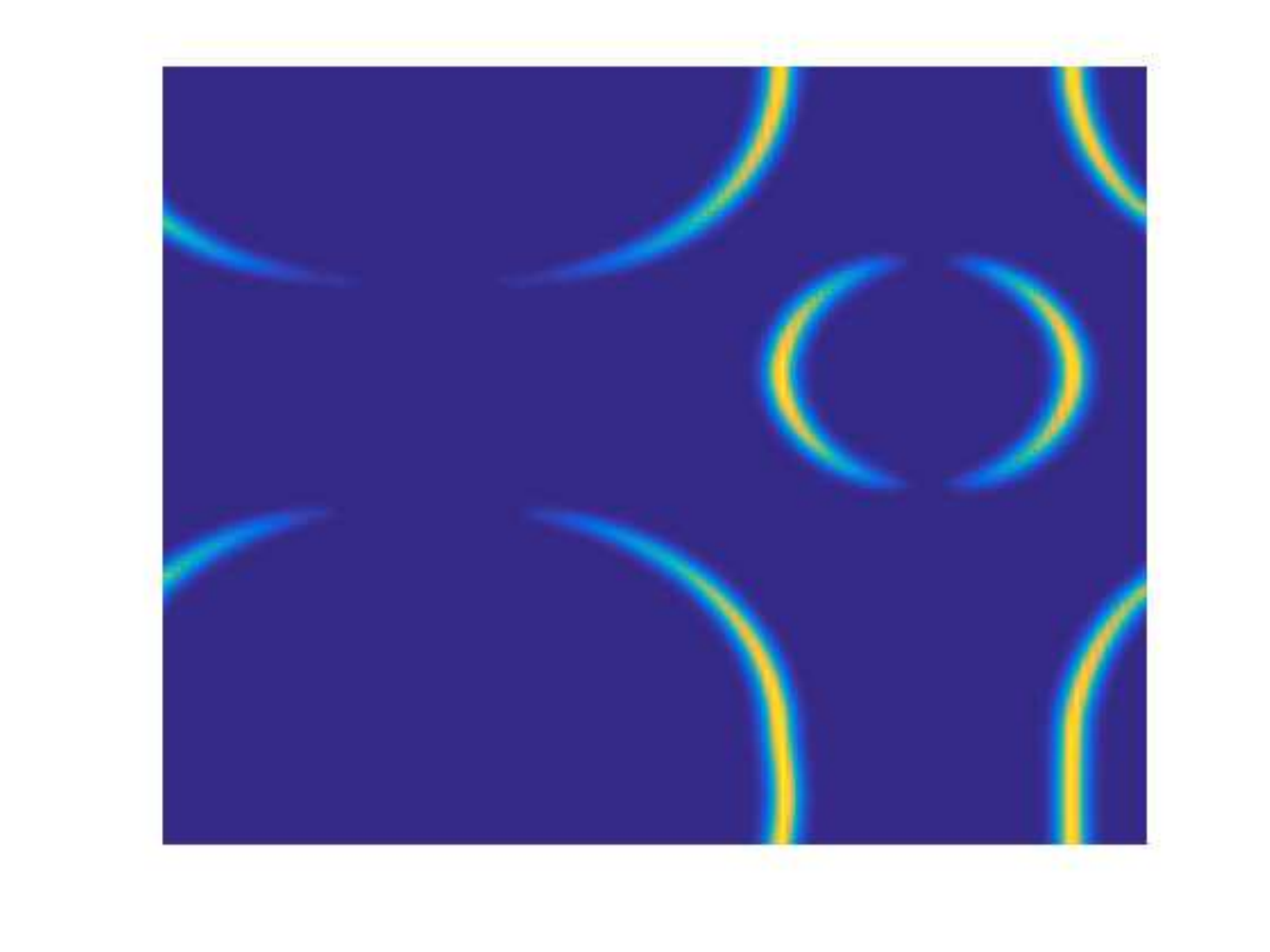}
}
\caption{Snapshots of the phase variable $\phi$ (left) and $\rho$ (right) are taken at t=1, 100, 800, 1500 for example 7.}\label{fig:fig7}
\end{figure}
\begin{figure}[htp]
\centering
\includegraphics[width=10cm,height=7cm]{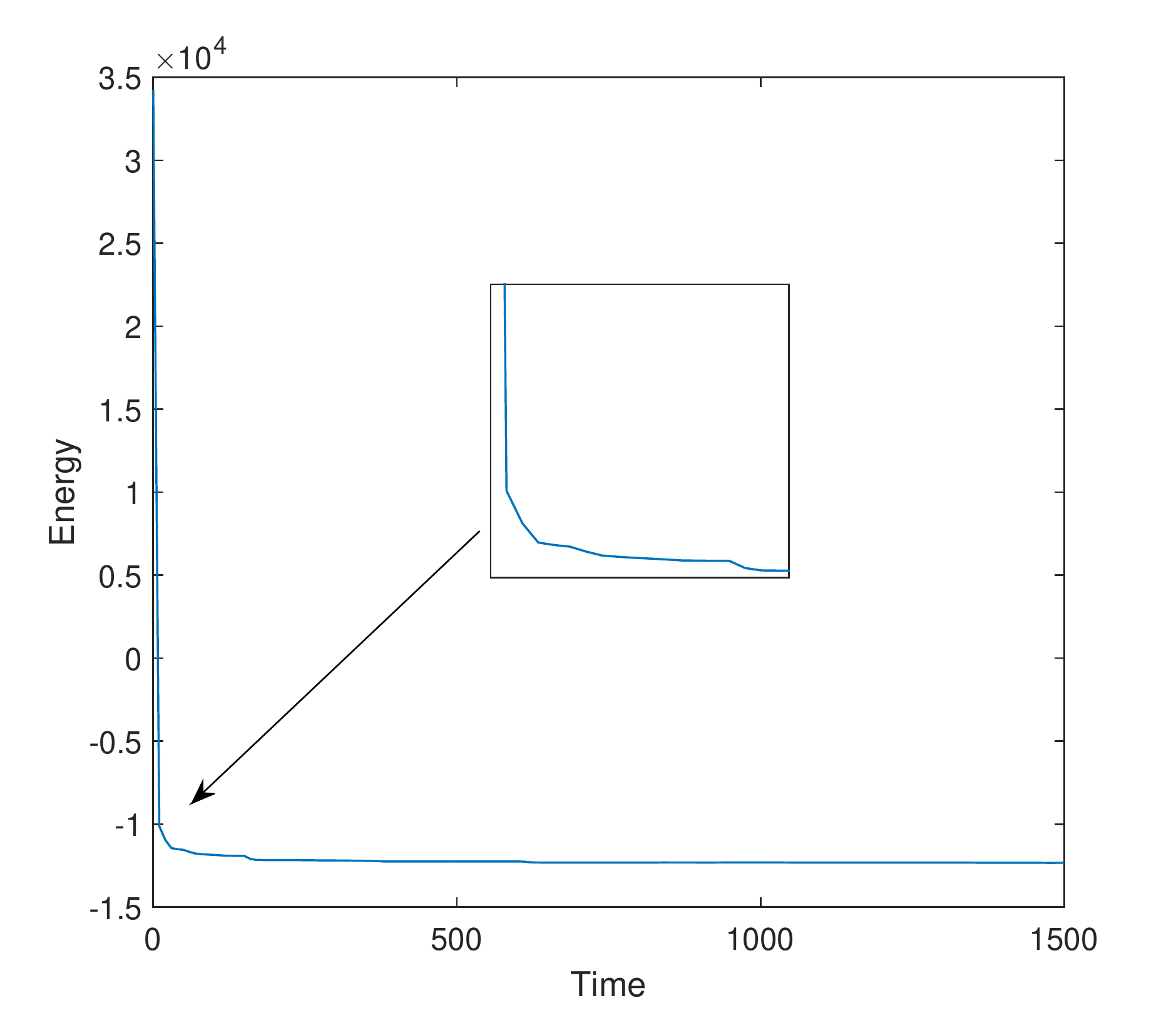}
\caption{Time evolution of the free energy functional for spinodal decomposition.}\label{fig:fig8}
\end{figure}
\section*{Acknowledgement}
No potential conflict of interest was reported by the author. We would like to acknowledge the assistance of volunteers in putting together this example manuscript and supplement.
\bibliographystyle{siamplain}
\bibliography{Reference}

\end{document}